\documentclass[11pt]{amsart}

\usepackage[utf8]{inputenc}

\usepackage{etoolbox}

\usepackage[new]{old-arrows}

\makeatletter
\let\old@tocline\@tocline
\let\section@tocline\@tocline
\newcommand{\subsection@dotsep}{4.5}
\newcommand{\subsubsection@dotsep}{4.5}
\patchcmd{\@tocline}
  {\hfil}
  {\nobreak
     \leaders\hbox{$\m@th
        \mkern \subsection@dotsep mu\hbox{.}\mkern \subsection@dotsep mu$}\hfill
     \nobreak}{}{}
\let\subsection@tocline\@tocline
\let\@tocline\old@tocline

\patchcmd{\@tocline}
  {\hfil}
  {\nobreak
     \leaders\hbox{$\m@th
        \mkern \subsubsection@dotsep mu\hbox{.}\mkern \subsubsection@dotsep mu$}\hfill
     \nobreak}{}{}
\let\subsubsection@tocline\@tocline
\let\@tocline\old@tocline

\let\old@l@subsection\l@subsection
\let\old@l@subsubsection\l@subsubsection

\def\@tocwriteb#1#2#3{%
  \begingroup
    \@xp\def\csname #2@tocline\endcsname##1##2##3##4##5##6{%
      \ifnum##1>\c@tocdepth
      \else \sbox\z@{##5\let\indentlabel\@tochangmeasure##6}\fi}%
    \csname l@#2\endcsname{#1{\csname#2name\endcsname}{\@secnumber}{}}%
  \endgroup
  \addcontentsline{toc}{#2}%
    {\protect#1{\csname#2name\endcsname}{\@secnumber}{#3}}}%

\newlength{\@tocsectionindent}
\newlength{\@tocsubsectionindent}
\newlength{\@tocsubsubsectionindent}
\newlength{\@tocsectionnumwidth}
\newlength{\@tocsubsectionnumwidth}
\newlength{\@tocsubsubsectionnumwidth}
\newcommand{\settocsectionnumwidth}[1]{\setlength{\@tocsectionnumwidth}{#1}}
\newcommand{\settocsubsectionnumwidth}[1]{\setlength{\@tocsubsectionnumwidth}{#1}}
\newcommand{\settocsubsubsectionnumwidth}[1]{\setlength{\@tocsubsubsectionnumwidth}{#1}}
\newcommand{\settocsectionindent}[1]{\setlength{\@tocsectionindent}{#1}}
\newcommand{\settocsubsectionindent}[1]{\setlength{\@tocsubsectionindent}{#1}}
\newcommand{\settocsubsubsectionindent}[1]{\setlength{\@tocsubsubsectionindent}{#1}}

\renewcommand{\l@section}{\section@tocline{1}{\@tocsectionvskip}{\@tocsectionindent}{}{\@tocsectionformat}}%
\renewcommand{\l@subsection}{\subsection@tocline{2}{\@tocsubsectionvskip}{\@tocsubsectionindent}{}{\@tocsubsectionformat}}%
\renewcommand{\l@subsubsection}{\subsubsection@tocline{3}{\@tocsubsubsectionvskip}{\@tocsubsubsectionindent}{}{\@tocsubsubsectionformat}}%
\newcommand{\@tocsectionformat}{}
\newcommand{\@tocsubsectionformat}{}
\newcommand{\@tocsubsubsectionformat}{}
\expandafter\def\csname toc@1format\endcsname{\@tocsectionformat}
\expandafter\def\csname toc@2format\endcsname{\@tocsubsectionformat}
\expandafter\def\csname toc@3format\endcsname{\@tocsubsubsectionformat}
\newcommand{\settocsectionformat}[1]{\renewcommand{\@tocsectionformat}{#1}}
\newcommand{\settocsubsectionformat}[1]{\renewcommand{\@tocsubsectionformat}{#1}}
\newcommand{\settocsubsubsectionformat}[1]{\renewcommand{\@tocsubsubsectionformat}{#1}}
\newlength{\@tocsectionvskip}
\newcommand{\settocsectionvskip}[1]{\setlength{\@tocsectionvskip}{#1}}
\newlength{\@tocsubsectionvskip}
\newcommand{\settocsubsectionvskip}[1]{\setlength{\@tocsubsectionvskip}{#1}}
\newlength{\@tocsubsubsectionvskip}
\newcommand{\settocsubsubsectionvskip}[1]{\setlength{\@tocsubsubsectionvskip}{#1}}

\patchcmd{\tocsection}{\indentlabel}{\makebox[\@tocsectionnumwidth][l]}{}{}
\patchcmd{\tocsubsection}{\indentlabel}{\makebox[\@tocsubsectionnumwidth][l]}{}{}
\patchcmd{\tocsubsubsection}{\indentlabel}{\makebox[\@tocsubsubsectionnumwidth][l]}{}{}

\newcommand{\@sectypepnumformat}{}
\renewcommand{\contentsline}[1]{%
  \expandafter\let\expandafter\@sectypepnumformat\csname @toc#1pnumformat\endcsname%
  \csname l@#1\endcsname}
\newcommand{\@tocsectionpnumformat}{}
\newcommand{\@tocsubsectionpnumformat}{}
\newcommand{\@tocsubsubsectionpnumformat}{}
\newcommand{\setsectionpnumformat}[1]{\renewcommand{\@tocsectionpnumformat}{#1}}
\newcommand{\setsubsectionpnumformat}[1]{\renewcommand{\@tocsubsectionpnumformat}{#1}}
\newcommand{\setsubsubsectionpnumformat}[1]{\renewcommand{\@tocsubsubsectionpnumformat}{#1}}
\renewcommand{\@tocpagenum}[1]{%
  \hfill {\mdseries\@sectypepnumformat #1}}

\let\oldappendix\appendix
\renewcommand{\appendix}{%
  \leavevmode\oldappendix%
  \addtocontents{toc}{%
    \protect\settowidth{\protect\@tocsectionnumwidth}{\protect\@tocsectionformat\sectionname\space}%
    \protect\addtolength{\protect\@tocsectionnumwidth}{2em}}%
}
\makeatother

\makeatletter
\settocsectionnumwidth{2em}
\settocsubsectionnumwidth{2.5em}
\settocsubsubsectionnumwidth{3em}
\settocsectionindent{1pc}%
\settocsubsectionindent{\dimexpr\@tocsectionindent+\@tocsectionnumwidth}%
\settocsubsubsectionindent{\dimexpr\@tocsubsectionindent+\@tocsubsectionnumwidth}%
\makeatother

\settocsectionvskip{10pt}
\settocsubsectionvskip{0pt}
\settocsubsubsectionvskip{0pt}
    
\settocsectionformat{\bfseries}
\settocsubsectionformat{\mdseries}
\settocsubsubsectionformat{\mdseries}
\setsectionpnumformat{\bfseries}
\setsubsectionpnumformat{\mdseries}
\setsubsubsectionpnumformat{\mdseries}

\let\oldtableofcontents\tableofcontents
\renewcommand{\tableofcontents}{%
  \vspace*{-\linespacing}
  \oldtableofcontents}

\usepackage{amsmath, amsfonts, amssymb, amsthm, amscd, tikz-cd, graphicx, float, epstopdf
}
\usepackage[scr=boondox]{mathalpha}
\usetikzlibrary{decorations.markings}
\usetikzlibrary{arrows.meta}

\usepackage{caption}
\usepackage{fullpage}

\allowdisplaybreaks[4]

\usepackage[pdfpagemode={UseOutlines},bookmarks=true,bookmarksopen=true, bookmarksopenlevel=0,bookmarksnumbered=true,hypertexnames=false, colorlinks,linkcolor={blue},citecolor={blue},urlcolor={blue}, pdfstartview={FitV},unicode,breaklinks=true,backref=page]{hyperref}

\newtheorem{theorem}{Theorem}[section]

\newtheorem{lemma}[theorem]{Lemma}
\newtheorem{proposition}[theorem]{Proposition}

\newtheorem{conj}[theorem]{Conjecture}

\theoremstyle{definition}

\newtheorem{remark}[theorem]{Remark}
\newtheorem{example}[theorem]{Example}
\newtheorem{defn}[theorem]{Definition}

\colorlet{lightblue}{blue!30!white}
\colorlet{lightorange}{orange!30!white}

\title{Intersections of dual $\operatorname{SL}_3$-webs}
\author{Linhui Shen}
\address{Department of Mathematics, Michigan State University, 619 Red Cedar Road, 302
Wells Hall, East Lansing, Michigan 48824, United States}
\email{linhui@math.msu.edu}
\author{Zhe Sun}
\address{Key Laboratory of Wu Wen-Tsun Mathematics, Chinese Academy of Sciences, \newline
School of Mathematical Sciences, University of Science and Technology of China, 96 Jinzhai Road, 230026 Hefei, Anhui, China}
\email{sunz@ustc.edu.cn}
\author{Daping Weng}
\address{Department of Mathematics, University of California, Davis, One Shields Avenue, Davis, CA 95616, United States}
\email{dweng@ucdavis.edu}

\date{\today}

\begin{document}

\begin{abstract}
We introduce a topological intersection number for an ordered pair of $\operatorname{SL}_3$-webs on a decorated surface. Using this intersection pairing between reduced $(\operatorname{SL}_3,\mathcal{A})$-webs and a collection of $(\operatorname{SL}_3,\mathcal{X})$-webs associated with the Fock--Goncharov cluster coordinates, 
we provide a natural combinatorial interpretation of the bijection from the set of reduced $(\operatorname{SL}_3,\mathcal{A})$-webs to the tropical set $\mathcal{A}^+_{\operatorname{PGL}_3,\hat{S}}(\mathbb{Z}^t)$, as established by Douglas and Sun in \cite{DS20a, DS20b}. We provide a new proof of the flip equivariance of the above bijection, which is crucial for proving the Fock--Goncharov duality conjecture of higher Teichm\"uller spaces for $\operatorname{SL}_3$.
\end{abstract}
\maketitle

\tableofcontents

\section{Introduction}

The ${\rm SL}_3$-webs, introduced by Kuperberg \cite{K96}, are oriented trivalent graphs on a disk used to construct natural bases for the tensor invariant spaces of irreducible representations of ${\rm SL}_3$. These webs have been extended to surfaces and have been utilized in studying character varieties and skein algebras associated with those surfaces. Sikora and Westbury \cite{SW07} developed a theory of confluence of graphs, and as an application, constructed natural bases of ${\rm SL}_3$-skein algebras associated with surfaces by using reduced ${\rm SL}_3$-webs. 

\smallskip 

Let $\hat{S}$ be a compact oriented topological surface with marked points on its boundary and punctures inside. There are two versions of reduced webs on $\hat{S}$, denoted by $\mathscr{W}_{\hat{S}}^{\mathcal{A}}$ the set of reduced $\mathcal{A}$-webs and by $\mathscr{W}_{\hat{S}}^{\mathcal{X}}$ the set of reduced $\mathcal{X}$-webs respectively, whose definition will be recalled in Section \ref{subsec2.1}. Recently, there are several works focus on the $\mathcal{A}$-webs (cf. \cite{DS20a, DS20b, FS22, Kim20, NY21}), while other works focus on the $\mathcal{X}$-webs (cf. \cite{FP16, Fr22, FrP23, IK22, IOS22, IY23}).

\smallskip 

The present paper, for the first time, puts these two versions of webs together and proposes a mapping class group equivariant intersection pairing 
\begin{equation}
\label{intersection.paring,intro}
\mathbb{I}: ~~ \mathscr{W}_{\hat{S}}^\mathcal{A}\times \mathscr{W}_{\hat{S}}^{\mathcal{X}} \longrightarrow \frac{1}{3} \mathbb{Z}.
\end{equation}
We place the pairing $\mathbb{I}$ in the context of the cluster duality of  higher Teichm\"uller spaces.

\subsection{Background}

The moduli  spaces $\mathcal{X}_{{\rm G},\hat{S}}$ and $\mathcal{A}_{{\rm G}^L,\hat{S}}$, introduced by Fock and Goncharov in \cite{FG06}, are variations of the character varieties of the fundamental group $\pi_1(S)$ into a pair of Langlands dual groups ${\rm G}$ and ${\rm G}^L$. For ${\rm PGL_2}$, these moduli spaces recover the enhanced Teichm\"uller space and the decorated Teichm\"uller space of Penner in \cite{Pen87}.
Fock and Goncharov proposed a  duality conjecture (\cite[Conjecture 12.3]{FG06}) relating $\mathcal{X}_{{\rm G},\hat{S}}$ and $\mathcal{A}_{{\rm G}^L,\hat{S}}$. This conjecture has found numerous interactions with cluster algebras, representation theory, mirror symmetry, low dimensional topology, and Poisson
geometry.  

\smallskip

To elaborate, when $\hat{S}$ has no marked points, the duality conjecture asserts that the ring of regular functions on $\mathcal{X}_{{\rm G},\hat{S}}$ admits a canonical linear basis parameterized by the integer tropical points of $\mathcal{A}_{{\rm G}^L,\hat{S}}$, and the parametrization is equivariant under the action of the mapping class group of the surface.\footnote{The original conjecture also includes the reversed direction; namely, the ring of regular functions on $\mathcal{A}_{{\rm G}^L,\hat{S}}$ has a canonical linear basis parameterized by the tropical points of $\mathcal{X}_{{\rm G}^L,\hat{S}}$. However, when $\hat{S}$ has punctures, the conjecture does not hold and requires a slight modification.} As explained in \cite{GS15}, when  $\hat{S}$ has marked points on its boundary, the space $\mathcal{X}_{{\rm G},\hat{S}}$ should be replaced by $\mathcal{P}_{{\rm G},\hat{S}}$, an enhancement of $\mathcal{X}_{{\rm G}, \hat{S}}$.  It is  further conjectured that when the tropical points of $\mathcal{A}_{{\rm G}^L,\hat{S}}$ are cut out by the tropical potential function, the dual space should be changed to the character variety of $\hat{S}$. 

\smallskip

The duality conjecture is proven for ${\rm G}=\operatorname{PGL}_2$ in \cite[Theorem 12.3]{FG06} by relating Thurston's transversely integer measured laminations \cite{Thu79} to the tropical  coordinates using the topological intersection numbers with the ideal triangulation of $\hat{S}$, and to the trace functions of the laminations. 
Motivated by the lamination approach,  the  duality for ${\rm G}=\operatorname{PGL}_3$ has recently been investigated  by using webs.  The first step is to provide a canonical bijection between  the tropical points and the reduced webs, which has been established by  Douglas and Sun \cite{DS20a,DS20b}. 
When $\hat{S}$ is a punctured surface without marked points, based on the bijection of Douglas and Sun, Kim \cite{Kim20}  showed that the highest degrees of the trace functions of the reduced webs are equal to the tropical coordinates of $\mathcal{A}_{{\rm SL}_3, \hat{S}}(\mathbb{Z}^t)$, completing the proof of the duality conjecture for ${\rm PGL}_3$. 

\smallskip 

An alternative approach to the duality conjecture is through cluster theory. Fock and Goncharov generalized the duality conjecture to the setting of cluster ensembles (\cite[Conjecture 4.1]{FG09}). Using scattering diagrams and broken lines, Gross, Hacking, Keel, and Kontsevich \cite{GHKK18} proved the duality conjecture under certain conditions, which follow from the existence of cluster Donaldson-Thomas transformations (a.k.a. reddening sequences).  Goncharov and Shen \cite{GS18, GS19} showed that the pair $(\mathcal{P}_{{\rm G},\hat{S}}, \mathcal{A}_{{\rm G}^L,\hat{S}})$ forms a cluster ensemble, and explicitly determine their cluster Donaldson-Thomas transformations when $\hat{S}$ is not a once-punctured surface $S_{g,1}$. For the once punctured surface $S_{g,1}$, by slightly adjusting Goncharov--Shen's construction, Fraser and Pylyavskyy \cite{FrP23} obtained the cluster Donaldson-Thomas transformation  for $\mathcal{A}_{\operatorname{SL}_n,S_{g,1}}$ when $n>2$. Combining the above results, the duality conjecture holds for these cases. 

\smallskip 

Comparing the two distinct approaches to the duality conjecture mentioned above, one advantage of the lamination-web approach is the explicitness in its construction,
which makes it more convenient for concrete calculations. The theta bases in \cite{GHKK18}, although applicable to all cluster ensembles with reddening sequences, are challenging to compute explicitly.  Recently, Mandel and Qin \cite{MQ23} proved that for $\operatorname{PGL}_2$, the theta bases and the Fock--Goncharov's lamination bases coincide. For $\operatorname{PGL}_3$, the comparison between the theta bases and the web bases is an interesting direction for future research. 
Moreover, the lamination-web approach aligns more closely with the study of skein algebras and the quantum trace maps \cite{BW11, LY23}, which establish connections between two different quantizations of the rings of regular functions of the ${\rm SL}_n$ character varieties.

\subsection{Main results}
Following Knutson and Tao \cite{KT98}, a hive is an arrangement of numbers satisfying a family of rhombus conditions. In Section \ref{sec:hives},  we recall the set ${\bf Hive}(\mathcal{T})$ of ${\rm SL}_3$ hives associated with an arbitrary ideal triangulation $\mathcal{T}$ of $\hat{S}$. As in \eqref{oact.rec}, the hives for different ideal triangulations are related by  a sequence $\varphi_{\mathcal{T}', \mathcal{T}}$ of octahedron relations
\[
\varphi_{\mathcal{T}', \mathcal{T}}: ~ {\bf Hive}(\mathcal{T}) \stackrel{\sim}{\longrightarrow}  {\bf Hive}(\mathcal{T}').
\]
In Section \ref{sec.42}, we review the cone $\mathcal{A}^+_{{\rm PGL}_3, \hat{S}}(\mathbb{Z}^t)$ of integer tropical points cut out by the tropical potential function of Goncharov and Shen \cite{GS15}. The paper {\it loc.cit.} constructed a  bijection $\alpha_{\mathcal{T}}$ between $\mathcal{A}^+_{{\rm PGL}_3, \hat{S}}(\mathbb{Z}^t)$ and  ${\bf Hive}(\mathcal{T})$, and obtained the following commutative diagram of bijections
\begin{equation} 
\label{comm.hive.1}
\begin{tikzcd}
 & {\bf Hive}(\mathcal{T}) \arrow{dd}{\varphi_{\mathcal{T}',\mathcal{T}}}\\
\mathcal{A}_{{\rm PGL}_3,\hat{S}}^+(\mathbb{Z}^t) \arrow{ru}{\alpha_{\mathcal{T}}} \arrow{rd}{\alpha_{\mathcal{T}'}} &  \\%
& {\bf Hive}(\mathcal{T}')
\end{tikzcd}.
\end{equation}

In Section \ref{sec:intersection}, we associate with each 
ideal triangulation $\mathcal{T}$ of $\hat{S}$ a family $\{[V_i]\}$ of $\mathcal{X}$-webs, which consists of the oriented ideal edges of $\mathcal{T}$ and the inward tripods inside  ideal triangles of $\mathcal{T}$. These webs correspond to the Fock-Goncharov cluster $\mathcal{A}$-coordinates.  Using the intersection pairing with $[V_i]$, we define the following intersection map in Definition \ref{definition:i}
\[
{i}_{\mathcal{T}}: \mathscr{W}_{\hat{S}}^{\mathcal{A}} \longrightarrow \left(\frac{1}{3}\mathbb{Z}\right)^n.
\]

The following theorem is the main result of this paper.

\begin{theorem}
\label{main.theorem1}
The map $i_{\mathcal{T}}$ gives rise to a bijection between $\mathscr{W}_{\hat{S}}^{\mathcal{A}}$ and  ${\bf Hive}(\mathcal{T})$. For any two ideal triangulations $\mathcal{T}$ and $\mathcal{T}'$ of $\hat{S}$, we have the following commutative diagram of bijections
\[ \begin{tikzcd}
 & {\bf Hive}(\mathcal{T}) \arrow{dd}{\varphi_{\mathcal{T}',\mathcal{T}}}\\
\mathscr{W}_{\hat{S}}^\mathcal{A} \arrow{ru}{i_{\mathcal{T}}} \arrow{rd}{i_{\mathcal{T}'}} &  \\%
& {\bf Hive}(\mathcal{T}')
\end{tikzcd}.
\]
Comparing with \eqref{comm.hive.1}, we get a canonical bijection between $\mathscr{W}_{\hat{S}}^{\mathcal{A}}$ and $\mathcal{A}_{{\rm PGL}_3, \hat{S}}^+(\mathbb{Z}^t)$.
\end{theorem}

A further tropicalization of the duality between $\mathcal{A}_{{\rm G}^L, \hat{S}}$ and $\mathcal{P}_{{\rm G}, \hat{S}}$ yields a mysterious mapping class group equivariant pairing 
\begin{equation}
\label{trop. duality.paring.intro}
\mathcal{I}:~\mathcal{A}_{{\rm G}^L, \hat{S}}(\mathbb{R}^t) \times \mathcal{P}_{{\rm G},\hat{S}}(\mathbb{R}^t) \longrightarrow \mathbb{R}.
\end{equation}
A better understanding of the pairing 
is important to the study of the moduli spaces $\mathcal{A}_{{\rm G}^L, \hat{S}}$ and $\mathcal{P}_{{\rm G}, \hat{S}}$. 
We conjecture a natural bijection between $\mathscr{W}_{\hat{S}}^{\mathcal{X}}$ and the cone $\mathcal{P}^{+}_{{\rm PGL}_3, \hat{S}}(\mathbb{Z}^t)$. See \cite{IK22} for recent results in this direction. We expect that our intersection paring \eqref{intersection.paring,intro} provides a geometric-combinatorial way to realize the pairing $\mathcal{I}$.
We wish to apply the paring \eqref{intersection.paring,intro} to investigate the compactification of the higher Teichm\"uller spaces in the future. See \cite{FG11} for related results for ${\rm PGL}_2$.

\subsection{Historical Comments}

We include a list of results related to this paper.

\smallskip 

 Douglas and Sun \cite{DS20a, DS20b} constructed a mapping class group  equivariant bijection
\[
 \Phi: ~ \mathscr{W}_{\hat{S}}^{\mathcal{A}}\stackrel{\sim}{\longrightarrow}  \mathcal{A}_{{\rm PGL}_3,\hat{S}}^+(\mathbb{Z}^t). 
\]
Frohman and Sikora \cite{FS22} provided a different topological integer coordinate system for the reduced $\operatorname{SL}_3$-webs depending on the ideal triangulation chosen. 
Theorem \ref{main.theorem1} provides an alternative interpretation of the bijection $\Phi$ in terms of the intersections.

\smallskip 

For the disk case and for $\operatorname{SL}_3$, Fontaine, Kamnitzer, and Kuperberg \cite{FKK13} showed that reduced $\mathcal{A}$-webs are dual to CAT(0) triangulated diskoids in the affine building, and that each reduced web with minuscule boundary $\vec{\lambda}$ corresponds to a component of the Satake fiber $F(\vec{\lambda})$. See Theorem 1.4 of {\it loc. cit.} In \cite{Ak20}, Akhmejanov showed that each diskoid is the intersection of min-convex set and max-convex hulls of a generic point of a component of $F(\vec{\lambda})$. Goncharov and Shen \cite{GS15} give a bijection between the top dimensional components of the Satake fibers and the tropical points. Thus, the above works implicitly imply a  bijection between reduced webs and tropical points. In this paper, we crucially use the dual graphs of reduced webs on surfaces, which we propose to call {\it surfacoids}, as a slight extension of the diskoids. 

\medskip 

Several approaches in the literature aim to understand the pairing \eqref{trop. duality.paring.intro}. In \cite{Le21}, Le utilized the length-minimal weighted networks in the affine buildings to interpret tropical Fock-Goncharov coordinates on $\mathcal{A}_{{\rm SL}_n, \hat{S}}$ for disks. However, as explained in Section 4.4 of {\it loc.cit.}, crucial ingredients are still missing to fully realize the pairing $\mathcal{I}$. In some special cases, a representation-theoretic understanding of the pairing $\mathcal{I}$ can be found in \cite{Fei23}. It is worth mentioning that our intersection number pairing \eqref{intersection.paring,intro} is more suitable for topological visualization and can be broadly defined for any ordered pairs of webs on marked surfaces, not limited to disks. 

\smallskip 

In the ongoing joint work \cite{ISY} for the $\operatorname{Sp}_4$ case, Ishibashi, Yuasa, and Sun constructed the intersection number coordinates for the crossroad webs \cite{IY22}.

\medskip

\noindent {\bf Acknowledgements}.  We are grateful to Alexander Goncharov for his encouragement and enlightening discussions. 
We thank Ian Le for the conversations on the interactions with affine buildings.
L. Shen is supported by the NSF grant DMS-2200738.

\medskip 

\section{Definitions}
This section recalls the definitions of $\mathcal{A}$- and $\mathcal{X}$- webs of type $A_2$ over decorated surfaces. We introduce a canonical intersection pairing between them. 

\subsection{Ideal triangulations}
A {\em decorated surface} $\hat{S}$ is a pair $(S,m_b)$ where $S$ is a connected oriented compact surface with boundary and $m_b$ is a finite (possibly empty) set of marked points on its boundary considered up to isotopy. Note that $\partial S\backslash m_b$ is a union of circles and open intervals, which are called {\em boundary circles} and {\em boundary intervals} respectively. The boundary circles are considered as {\em punctures}. We denote the set of punctures by $m_p$.

An {\em ideal triangulation} $\mathcal{T}$ of $\hat{S}$ is a maximal collection of arcs joining the marked points and the punctures of $\hat{S}$ such that these arcs are pairwise disjoint on the interior of $S$ and non-homotopic. The edges (respectively triangles) of the ideal triangulations are called {\em ideal edges} (respectively {\em ideal triangles}). 
In this paper, we focus on decorated surfaces that admit ideal triangulations. Suppose the underlying surface of $\hat{S}$ is of genus $g$ and of $c$ many boundary components. Let $m$ be the total number of marked points on $\hat{S}$. A direct calculation shows that every ideal triangulation of $\hat{S}$ has $f$ many ideal triangles and $e$ many ideal edges, where 
\[
f=2c+m+4g-4, \qquad e= 3c+2m+6g-6.
\]

We place two vertices on each ideal edge, and place one vertex on the center of each ideal triangle. As illustrated by Figure \ref{quiver-T}, we draw nine arrows within each ideal triangle, obtaining a quiver $Q_{\mathcal{T}}$. Denote by $\Theta_{\mathcal{T}}$ the set of vertices of $Q_{\mathcal{T}}$. By definition, the cardinality of $\Theta_{\mathcal{T}}$ is 
\[2e+f= 8c+5m+16g-16.\] 
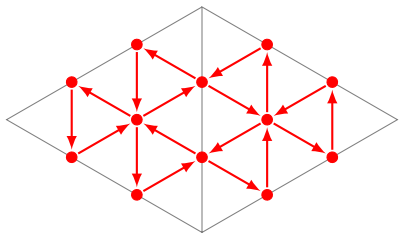
\begin{figure}[H]
\begin{tikzpicture}
\begin{scope}[rotate=-30]
\draw[gray=50!] (0,0)--(120:3)--(60:3)--(0,0)--(-3,0)--(120:3);
\draw[red,fill=red] (120:1) circle(2pt);
\draw[red,fill=red] (120:2) circle(2pt);
\draw[red,fill=red] (60:1) circle(2pt);
\draw[red,fill=red] (60:2) circle(2pt);
\draw[red,fill=red] (180:1) circle(2pt);
\draw[red,fill=red] (180:2) circle(2pt);
\draw[red,fill=red] (120:3)++(1,0) circle(2pt);
\draw[red,fill=red] (120:3)++(2,0) circle(2pt);
\draw[red,fill=red] (120:2)++(1,0) circle(2pt);
\draw[red,fill=red] (120:2)++(-1,0) circle(2pt);
\draw[red,fill=red] (120:1)++(-1,0) circle(2pt);
\draw[red,fill=red] (120:1)++(-2,0) circle(2pt);
 \foreach \position in {(120:2), (120:1)}
    {\draw[-latex, thick, red] \position++(-.1,0) -- ++(-.8,0);
    \draw[-latex, thick, red] \position++(-120:0.9) -- ++(60:0.8);
    \draw[-latex, thick, red] \position++(-1,0)++(-60:0.1) -- ++(-60:0.8);
    \draw[-latex, thick, red] \position++(60:0.9) -- ++(-120:0.8);
    \draw[-latex, thick, red] \position++(0.1,0) -- ++(0.8,0);
    \draw[-latex, thick, red] \position++(1,0)++(120:0.1) -- ++(120:0.8);}

\draw[-latex, thick, red] (120:1)++(-1,0)++(-.1,0) -- ++(-.8,0);
    \draw[-latex, thick, red] (120:1)++(-1,0)++(-120:0.9) -- ++(60:0.8);
    \draw[-latex, thick, red] (120:1)++(-1,0)++(-1,0)++(-60:0.1) -- ++(-60:0.8);
    \draw[-latex, thick, red] (120:2)++(1,0)++(60:0.9) -- ++(-120:0.8);
    \draw[-latex, thick, red] (120:2)++(1,0)++(0.1,0) -- ++(0.8,0);
    \draw[-latex, thick, red] (120:2)++(1,0)++(1,0)++(120:0.1) -- ++(120:0.8);   
\end{scope}    
\end{tikzpicture}
\caption{The quiver $Q_{\mathcal{T}}$ associated with an ideal triangulation $\mathcal{T}$}
\label{quiver-T}
\end{figure}

\subsection{Webs}
\label{subsec2.1}
\begin{defn}\label{defn: SL3 A web}
A {\em $(\operatorname{SL}_3,\mathcal{A})$-web} is an oriented graph embedded in $\hat{S}$ with finitely many connected components, where each component
is one of the following graphs
\begin{itemize}
\item a simple oriented loop,
\item an oriented arc connecting the boundary intervals of $\hat{S}$,
\item an oriented graph such that  every interior vertex is a $3$-valent sink or a $3$-valent source and the rest vertices are 1-valent vertices lying on the boundary intervals of ${\hat{S}}$.
\end{itemize}

A {\em $k$-gon} of a web $W$ is a contractible component of $\hat{S}\backslash W$ with $k$ sides. We say a $k$-gon is internal if all of its vertices are interior vertices of  $W$.  A web is {\em non-elliptic} if it contains no contractible loops and all of its internal $k$-gons have at least $6$ sides. A non-elliptic web is {\em reduced} if each $k$-gon with only one side along $\partial \hat{S}$ must satisfy $k\geq 5$.

Denote by $\mathscr{W}_{\hat{S}}^{\mathcal{A}}$ be the space of reduced 
$(\operatorname{SL}_3,\mathcal{A})$-webs up to homotopy equivalence on $\hat{S}\times [0,1]$. 

\end{defn}

\begin{example} 
\label{exmp2.1} Let $\hat{S}$ be a once-punctured disk with two marked points on its boundary. In Figure \ref{Figure.webs.a}, we represent the puncture and the marked points by red circles. We denote the sinks of the webs by black dots and the sources by white dots. The left web is not reduced because the shaded region is a square with one side along $\partial \hat{S}$. The right web is elliptic because the shaded region is an internal square. 
\begin{figure}[H]
\begin{tikzpicture}[scale=.3]
\begin{scope}
\draw[draw=gray!50!white,fill=gray!50!white] (135:4.5)--(135:2)--(225:2)--(225:4.5) arc (225:135:4.5);
\node (D1) at (135:4.5) {};
\node (E1) at (225:4.5) {};
\node (F1) at (45:4.5) {};
\node (G1) at (-45:4.5) {};
\node (H1) at (135:2) {};
\node (I1) at (225:2) {};
\node (J1) at (45:2) {};
\node (K1) at (-45:2) {};
\draw[-latex] (90:0.8) arc (90:450:0.8);
\draw (135:4.5)--(135:2)--(45:4.5);
\draw (225:4.5)--(225:2) -- (-45:4.5);
\draw (135:2)--(225:2);
\draw (-60:4.5)--(-120:4.5);
\draw (-52:4.5)--(-128:4.5);
\draw[blue] (0,0) circle (4.5cm);
\draw[fill=black] (-60:4.5) circle(4pt); 
\draw[fill=white] (-120:4.5) circle(4pt);
\draw[fill=white] (-52:4.5) circle(4pt); 
\draw[fill=black] (-128:4.5) circle(4pt);
\draw[fill=black] (H1) circle(4pt); 
\draw[fill=white] (I1) circle(4pt);
\draw[fill=white] (D1) circle(4pt); 
\draw[fill=black] (E1) circle(4pt);
\draw[fill=white] (F1) circle(4pt); 
\draw[fill=black] (G1) circle(4pt);
\draw[red,fill=white] (0,0) circle(7pt);    
\draw[red,fill=white] (0,4.5) circle(7pt);
\draw[red,fill=white] (0,-4.5) circle(7pt);

\end{scope}
\begin{scope}[shift={(13,-1.2)}]
\draw[blue] (1,1.2) circle (4.5cm);
\draw[red,fill=white] (1,1.2) circle(7pt);    
\draw[red,fill=white] (1,5.7) circle(7pt);
\draw[red,fill=white] (1,-3.3) circle(7pt);
\begin{scope}[shift={(1,1.2)}]
\draw[fill=white] (184:4.5) circle(4pt);
\draw[fill=white] (207.5:4.5) circle(4pt);
\draw[fill=white] (235:4.5) circle(4pt);
\draw[fill=black]  (60:4.5) circle(4pt);
\draw[fill=black] (7:4.5) circle(4pt);
\draw[fill=black] (-16:4.5) circle(4pt);
\draw[fill=black] (-40.6:4.5) circle(4pt);
\draw[fill=black] (-55:4.5) circle(4pt);
\node (D) at (184:4.9) {};
\node (E) at (205.8:4.9) {};
\node (F) at (-50:4.9) {};
\node (G) at (58.8:5) {};
\node (H) at (7:4.9) {};
\node (I) at (-15.5:4.9) {};
\node (K) at (-37:4.9) {};
\draw (235:4.5) -- (-55:4.5);
\end{scope}
   \draw[draw=gray!50!white,fill=gray!50!white]
    (-60:2)++(1,0) -- ++ (120:1)-- ++(-1,0)-- ++ (-120:1)--++(2,0);
\draw (0,0)++(-2,0)++(60:1) -- (D);
\draw (0,0)++(-2,0)++(-60:1) -- (E);
\draw (-120:2) -- (F);
\draw (120:2) -- (G);
\draw (60:2)++(1,0) -- (H);
\draw (2,0) -- (I);
\draw (-60:2)++(1,0) -- (K);
\foreach \angle in {0,120,240} {
  \begin{scope}[rotate=\angle]
  \draw (0,0) -- ++(-60:1) -- ++(1,0) -- ++(60:1) -- ++(120:1)
        -- ++(-1,0) -- (0,0);
  \filldraw (0:0) circle(4pt) ++(0:2) ++(120:1) circle(4pt)++(-60:2)++(-1,0) circle(4pt);
  \end{scope}}
  \draw (60:1)++(120:1) -- ++(2,0) --++(-120:1);
  \draw (-60:1)++(-120:1) -- ++(2,0) --++(120:1);
    \draw[fill=white] (60:1) circle(4pt);
      \draw[fill=white] (180:1) circle(4pt);
        \draw[fill=white] (300:1) circle(4pt);
          \draw[fill=white] (120:2) circle(4pt);
      \draw[fill=white] (240:2) circle(4pt);
        \draw[fill=white] (0:2) circle(4pt);
    \draw[fill=white] (-60:2)++(1,0) circle(4pt); 
    \draw[fill=white] (60:2)++(1,0) circle(4pt);
\end{scope}
\end{tikzpicture}
\caption{$(\operatorname{SL}_3, \mathcal{A})$-webs on a once-punctured disk}
\label{Figure.webs.a}
\end{figure}
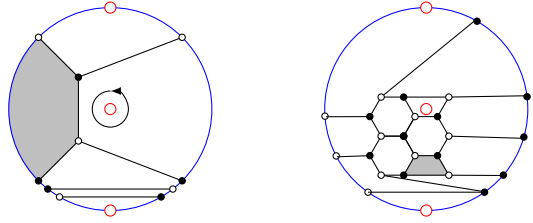
\end{example}

\medskip

\begin{example}
\label{minimal.ladder}
A bigon is a disk with two marked points on its boundary. 
As shown on Figure \ref{figure:bigon}(2), we consider a collection of arcs going between the two boundaries of a bigon such that 
\begin{itemize}
\item every pair of arcs intersect with each other at most once,
\item if two arcs intersect each other, then they point to different boundaries of the bigon.
\end{itemize}
We replace each crossing of the arcs with an ``I", obtaining a web as in Figure \ref{figure:bigon}(1). The webs obtained this way are called {\it minimal ladders}. It is easy to see that every minimal ladder is uniquely determined by its intersection with the bigon. 

\begin{figure}[H]
\includegraphics[scale=0.45]{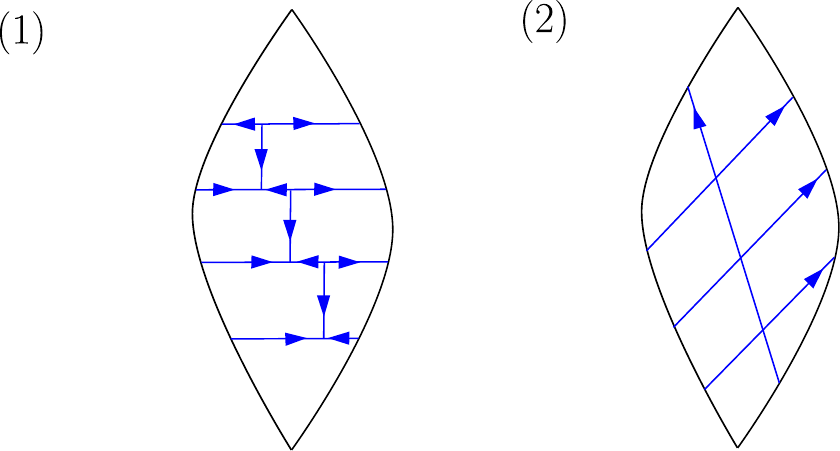}
\caption{(1) A minimal ladder on a bigon. (2) Its schematic diagram.}
\label{figure:bigon}
\end{figure}

Moreover, by choosing one of the two marked points of the bigon, we can label the endpoints of the arcs along each boundary by $1,2,3,\dots, n$ in the order according to their distance from the chosen marked point (there are necessarily the same number of endpoints along each boundary by construction). Along a chosen boundary $\alpha$ of the bigon, we say an arc is \emph{ascending} if its endpoint $i$ at $\alpha$ is smaller than its other endpoint $i'$ at the other boundary. Now for each ascending arc along a boundary, we define its \emph{staircase} to be the collection of ``I''s in the minimal ladder that correspond to crossings between this particular ascending arc and other arcs. For example, if we choose the top marked point in Figure \ref{figure:bigon}(2), then the left boundary has three ascending arcs and the right boundary has one ascending arc; correspondingly, each ``I'' in Figure \ref{figure:bigon}(1) is a staircase for an ascending arc along the left boundary, and the three ``I''s together form a staircase for the ascending arc along the right boundary.

\end{example}

\medskip

\begin{example} 
\label{tri-redu-web}
Let $\Delta$ be a disk with three marked points. 
A reduced $({\rm SL}_3, \mathcal{A})$-web on $\Delta$ is shown on Figure \ref{figure:localtr}(1). In general, every reduced  web on $\Delta$ consists of an oriented honeycomb in the middle and oriented arcs in the corners. Let $\mathbb{N}$ be the set of non-negative integers. Each reduced web on $\Delta$ corresponds to a tuple 
\[
(x, y,z,t,u,v,w) \in \mathbb{Z}\times \mathbb{N}^6.
\]
Here $|x|$ denotes the number of the arcs of the center oriented   honeycomb intersecting with each side of $\Delta$. The sign of $x$ determines the orientation of the honeycomb: when $x$ is positive, then the arrows of the honeycomb are coming out from the center; otherwise,  they are getting into the center. The numbers of oriented corner arcs around each marked points are determined by the nonnegative integers $y,z,t,u,v,w$. Hence, we obtain a bijection
\[
\mathscr{W}_{\Delta}^\mathcal{A}\stackrel{\sim}{=}\mathbb{Z}\times\mathbb{N}^6
\]

 \begin{figure}[H]
\includegraphics[scale=0.45]{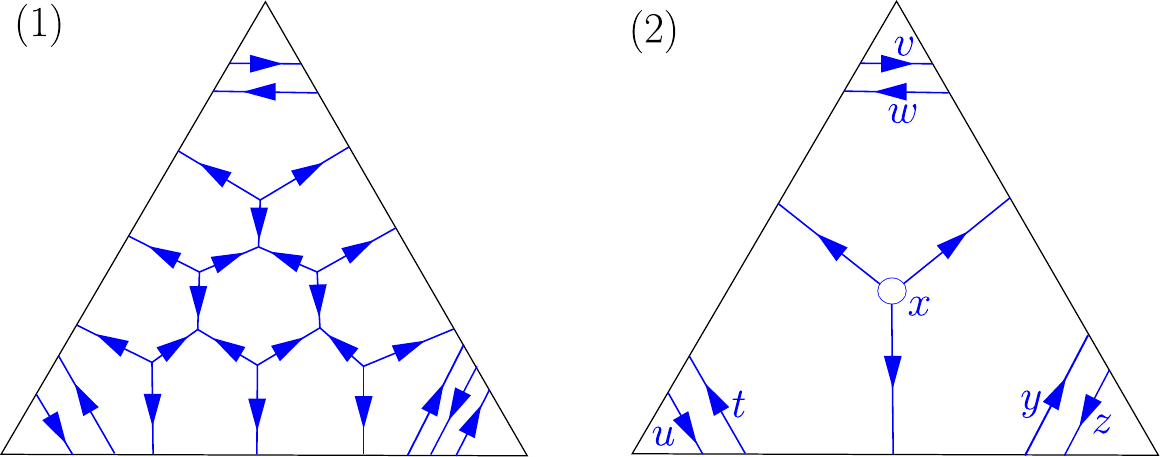}
\caption{A reduced web on $\Delta$. Here $x=3$, $y=2$, $z=t=u=v=w=1$.}
\label{figure:localtr}
\end{figure}
\end{example}

Below we recall the {\it good position} property of reduced webs in Proposition \ref{proposition:gp}. It shows every reduced web can be obtained by gluing minimal ladders and reduced webs on ideal triangles.
The good position property is originally due to Kuperberg \cite{K96} on the disk cases and can be easily generalized to surfaces (see e.g. \cite{DS20a},\cite{FS22}).

Let $\mathcal{T}$ be an ideal triangulation of  $\hat{S}$. The {\em split ideal triangulation} $\mathcal{T}_2$ associated with $\mathcal{T}$ is constructed by splitting each internal ideal edge of  $\mathcal{T}$ into two disjoint homotopic ideal edges. As a result, $\hat{S}$ is cut into bigons and triangles as in Figure \ref{figure:tr}.

\begin{figure}[H]
\includegraphics[scale=0.45]{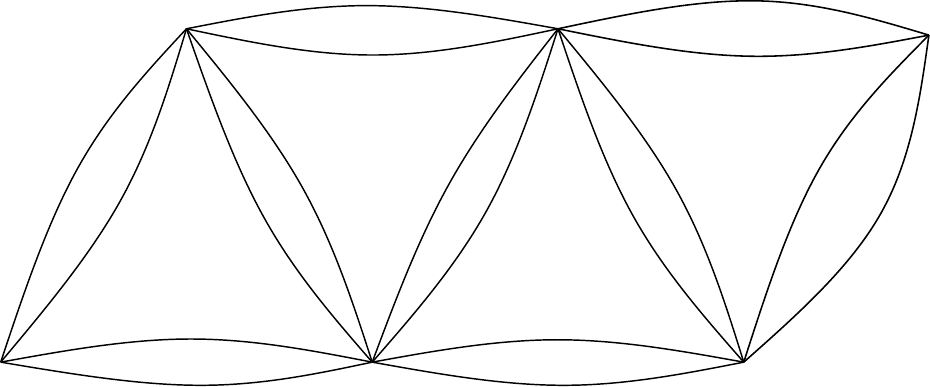}
\caption{Split ideal triangulation.}
\label{figure:tr}
\end{figure}

\begin{proposition}
\label{proposition:gp}
Every reduced $(\operatorname{SL}_3,\mathcal{A})$-web $W$ on $\hat{S}$ can be placed {\em in a good position} with respect to $\mathcal{T}_2$ such that
\begin{enumerate}
\item $W$ intersects transversely  with every edge of  $\mathcal{T}_2$;
\item the restriction of $W$ to every bigon of $\mathcal{T}_2$ is a minimal ladder as in Example \ref{minimal.ladder}; 
\item the restriction of $W$ to every triangle of $\mathcal{T}_2$ is a reduced web as in Example \ref{tri-redu-web}.
\end{enumerate}
\end{proposition}

\medskip

In contrast to the $(\operatorname{SL}_3,\mathcal{A})$-webs, we introduce their dual counterparts, the $(\operatorname{SL}_3,\mathcal{X})$-webs, which have their endpoints placed at the marked points and the punctures of $\hat{S}$. 
\begin{defn}
A {\em $(\operatorname{SL}_3,\mathcal{X})$-web} is an oriented graph embedded in $\hat{S}$ with finitely many connected components, where each component
is one of the following graphs
\begin{itemize}
\item a simple oriented loop,
\item an oriented arc connecting points in $m_b\cup m_p$,
\item a connected oriented graph on $S$ such that each interior vertex is a $3$-valent sink or a $3$-valent source and the rest  vertices are in $m_b\cup m_p$. 
\end{itemize}

A $(\operatorname{SL}_3, \mathcal{X})$-web is {\it reduced} if 
\begin{itemize}
\item it contains no peripheral loops, a.k.a., the loops surrounding a puncture;
\item each of its internal $k$-gons has at least 6 sides;
\item each of its $k$-gons with only one vertex in $m_b\cup m_p$ must satisfy $k\geq 4$.
\end{itemize}
Denote by $\mathscr{W}_{\hat{S}}^{\mathcal{X}}$ the space of reduced $(\operatorname{SL}_3, \mathcal{X})$-webs up to homotopy equivalence on $\hat{S} \times [0, 1]$.
\end{defn}
\begin{example} Figure \ref{figure.x.red} is a reduced $(\operatorname{SL}_3, \mathcal{X})$-web on a once-punctured disk with two marked points. Note that, unlike $\mathcal{A}$-webs,  $\mathcal{X}$-webs can enter into marked points and punctures.
\begin{figure}[H]
\begin{tikzpicture}[scale=.45]
    \draw[blue] (0,0) circle (3cm);
    \draw[fill=black] (0,1) circle(3pt);
    \draw (-1,1)--(1,1)--(0,3)--(-1,1)--(0,-3)--(1,1);
      \draw[fill=white] (1,1) circle(3pt);
    \draw[fill=white] (-1,1) circle(3pt);
    \draw (0,1)--(0,0);
    \draw[red,fill=white] (0,0) circle(5pt);    
    \draw[red,fill=white] (0,3) circle(5pt);
    \draw[red,fill=white] (0,-3) circle(5pt);
\end{tikzpicture}
\caption{A reduced $(\operatorname{SL}_3, \mathcal{X})$-web.}
\label{figure.x.red}
\end{figure}
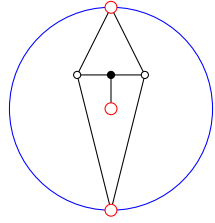
\end{example}
From simplicity, we will call the aforementioned webs $\mathcal{A}$- and $\mathcal{X}$-webs from now on, respectively, without mentioning $\operatorname{SL}_3$.

\subsection{Intersections} 
\label{sec:intersection}
Let $W$ be an $\mathcal{A}$-web and let $V$ be a $\mathcal{X}$-web on $\hat{S}$ such that they intersect transversely. For each intersection point $p\in W\cap V$,  the {\em intersection number} of the ordered pair $(W,V)$ at $p$ is defined as
\begin{equation*}
\epsilon_p(W,V):=\left\{
\begin{aligned}
&\frac{1}{3} \;\;\;  \text{if $W$ crosses $V$ through $p$ towards the right side of $V$};  \\
&\\
&\frac{2}{3} \;\;\; \text{if $W$ crosses $V$ through $p$ towards the left side of $V$}.
\end{aligned}
\right.
\end{equation*}
The {\em intersection number} of the pair $(W,V)$ is defined as
\begin{equation}
\label{definition:int}
i(W,V):= \sum_{p\in\,W\cap V} \epsilon_p(W,V).
\end{equation}
\begin{example}
Figure \ref{figure:Inttri} shows the intersection numbers for several pairs of webs  on a triangle. 
\begin{figure}[H]
 \includegraphics[scale=0.5]{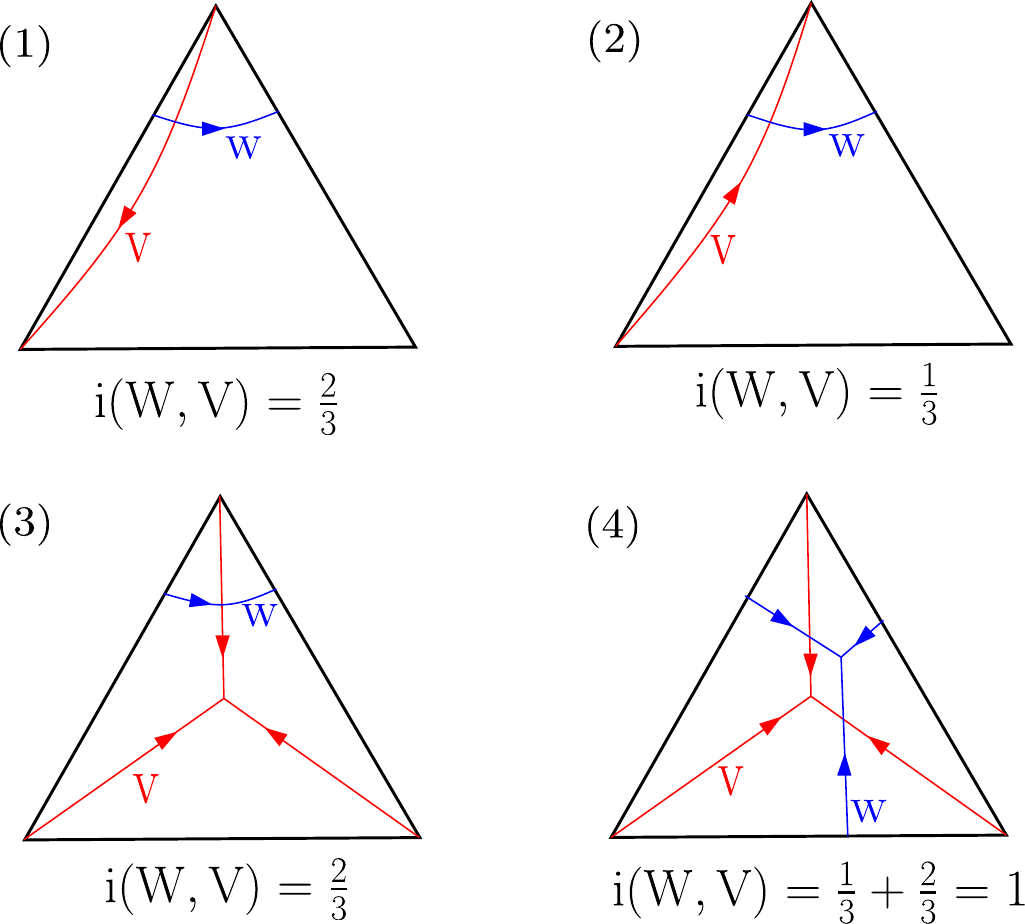}
\caption{Intersection numbers for ordered pairs of webs on a triangle.}
\label{figure:Inttri}
\end{figure}
\end{example}

\begin{defn} We define the intersection paring
\begin{equation}
\mathbb{I}:~ \mathscr{W}_{\hat{S}}^{\mathcal{A}} \times \mathscr{W}_{\hat{S}}^{\mathcal{X}}\longrightarrow \frac{1}{3}\mathbb{Z},\qquad ([W], [V])\, \longmapsto \,\min_{W\in [W],V\in [V]} \left\{i(W,V)\right\}
\end{equation}
where $W$ and $V$ are in the homotopy classes  $[W]$ and $[V]$ respectively such that they intersect  transversely.
\end{defn}

\begin{remark} Note that the intersection number between $[W]$ and $[V]$ only depends on their relative position. As a result, in order to compute $\mathbb{I}([W],[V])$, one can fix the web $W$ and vary $V$ within its homotopy class.
\end{remark}

\begin{remark} Let $\mathfrak{g}$ be a simple Lie algebra and let $\mathfrak{g}^\vee$ be its Langlands dual. We expect the existences of the notions of $(\mathfrak{g}, \mathcal{A})$ and $(\mathfrak{g}^\vee, \mathcal{X})$ , whose oriented edges are labelled by the dominant weights of $\mathfrak{g}$ and $\mathfrak{g}^\vee$ respectively. The change of orientations of an edge is equivalent to changing the label $\lambda$ to $-w_0(\lambda)$, where $w_0$ is the longest Weyl group element. 
\[
\begin{tikzpicture}[scale=2]
\draw[-latex
] (0,0) -- (0,1);
\draw[dashed] (0,0.5) circle (0.5);
\draw[dashed] (3,0.5) circle (0.5);
\draw[latex-] (3,0)--(3,1);
\node at (0.2,0.5) {$\lambda$};
\node at (3.4,0.5) {$-w_0(\lambda)$};
\node at (1.5,0.5) {$=$};
\end{tikzpicture}
\]
Let $X$ and $X^\vee$ be the weight lattices of  $\mathfrak{g}$ and $\mathfrak{g}^\vee$ respectively. There is a canonical pairing
\[
\langle~, ~\rangle: ~ X\times X^\vee \longrightarrow \frac{1}{n}\mathbb{Z},
\]
where $n$ is the determinant of the Cartan matrix of $\mathfrak{g}$. Given a $(\mathfrak{g},\mathcal{A})$ web $W$ and a $(\mathfrak{g},\mathcal{X})$ web $V$, intersecting transversely at $p$ as illustrated by the figure
\begin{center}
\begin{tikzpicture}
\draw[dashed] (0,0) circle (1);
\node at (-.2,-.2) {\tiny $p$};
\node[blue] at (0.5, -.2) {\tiny$W$};
\node[blue] at (0.5, .2) {\tiny $\lambda$};
\node[red] at (-.2, 0.5) {\tiny $V$};
\node[red] at (.2, 0.5) {\tiny $\mu$};
\draw[blue, -latex] (-1,0)--(1,0);
\draw[red, -latex] (0, 1)--(0,-1);
\end{tikzpicture}
\end{center}
We shall define the intersection number of $W$ and $V$ at $p$ as 
\[
\epsilon_p(W, V)= \langle \lambda, \mu \rangle.
\]
\end{remark}

\medskip

Now fix  an ideal triangulation $\mathcal{T}$ of $\hat{S}$. Recall the quiver $Q_{\mathcal{T}}$ with the vertex set $\Theta_{\mathcal{T}}$. We assign a reduced $\mathcal{X}$ web $[V_i]$ to each vertex $i\in \Theta_{\mathcal{T}}$. The vertex at the center of each ideal triangle corresponds to a tripod inside the same triangle, as depicted in Figure \ref{figure:triA}(8). Every vertex at each ideal edge $e$ corresponds to an oriented arc that is homotopy to $e$, with the orientation determined by the position of the vertex on $e$. 

\begin{figure}[H]
\includegraphics[scale=0.6]{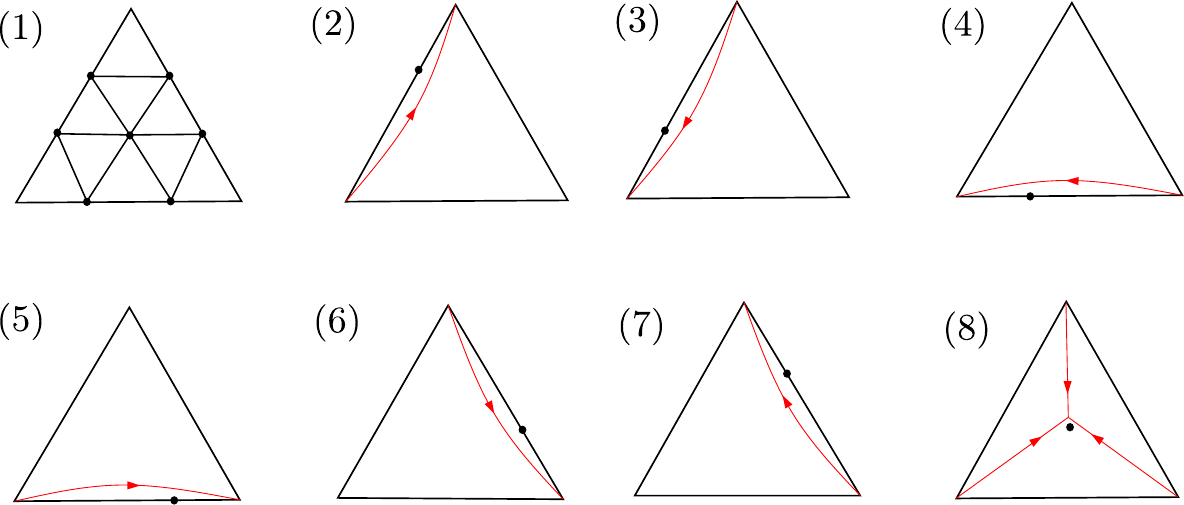}
\caption{The correspondence between $(\operatorname{SL}_3,\mathcal{X})$-webs and vertices of $Q_{\mathcal{T}}$.}
\label{figure:triA}
\end{figure}
By using the above collection of reduced $\mathcal{X}$ webs associated with an ideal triangulation, we arrive at the definition of the intersection map.
\begin{defn}
\label{definition:i}
For each ideal triangulation $\mathcal{T}$ of $\hat{S}$, we define the \emph{intersection map} 
\begin{equation}
\label{IT.map}
i_\mathcal{T}:\mathscr{W}_{\hat{S}}^{\mathcal{A}}\longrightarrow\left(\frac{1}{3}\mathbb{Z}\right)^{\Theta_{\mathcal{T}}}
\end{equation}
by 
\[i_\mathcal{T}([W]):=\left\{\mathbb{I}([W],[V_i])\right\}_{i\in\Theta_{\mathcal{T}}}.\]
\end{defn}

\bigskip

\section{Reduced Webs and Hives}
In this section, we show that the intersection map \eqref{IT.map} establishes a natural one-to-one correspondence between  reduced $\mathcal{A}$-webs and {\it Knutson-Tao hives}. For different ideal triangulations of $\hat{S}$, the corresponding Knutson-Tao hives are related by octahedron relations. 

\subsection{Hives}  
\label{sec:hives}
Following Knutson and Tao \cite{KT98}, a hive is an arrangement of numbers satisfying a family of rhombus conditions. 
Below, we include the definition of hives for ${\rm SL}_3$.

Let $\Delta$ be a disk with three marked points.  Let $a_1, \ldots, a_7$ be real numbers associated with the vertices of $Q_\Delta$ as follows
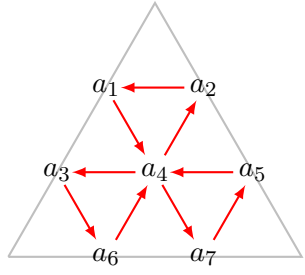
\begin{figure}[H]
\begin{tikzpicture}[scale=1.3]
\draw[gray!50, thick] (0,0)--(60:3)--(3,0)--(0,0);
\draw[red, thick, latex-] (60:2)++(0.15,0) -- ++ (0.65, 0);
\draw[red, thick, -latex] (60:2)++(-60:0.15) -- ++ (-60:0.65);
\draw[red, thick, -latex] (1,0)++(60:1.2) -- ++ (60:0.65);
\draw[red, thick, latex-] (60:1)++(0.15,0) -- ++ (0.65, 0);
\draw[red, thick, -latex] (60:1)++(-60:0.15) -- ++ (-60:0.65);
\draw[red, thick, -latex] (1,0)++(60:0.2) -- ++ (60:0.65);
\draw[red, thick, latex-] (60:1)++(1.15,0) -- ++ (0.65, 0);
\draw[red, thick, -latex] (60:2)++(-60:1.15) -- ++ (-60:0.65);
\draw[red, thick, -latex] (2,0)++(60:0.2) -- ++ (60:0.65);
\node (A) at (1,0) {$a_6$};
\node (B) at (2,0) {$a_7$};
\node at (60:1) {$a_3$};
\node at (60:2) {$a_1$};
\begin{scope}[shift={(3,0)}] 
\node at (120:1) {$a_5$};
\node at (120:2) {$a_2$};
\end{scope}
\begin{scope}[shift={(2,0)}] 
\node at (120:1) {$a_4$};
\end{scope}
\end{tikzpicture}
\caption{Coordinates associated with vertices of $Q_\Delta$.}
\label{quiver.delta}
\end{figure}
We say $(a_1,\ldots, a_7)$ is a hive if it satisfies the {\it rhombus condition}, that is, the following  
\[
a_1+a_2-a_4, \qquad a_3+a_4-a_1-a_6,\qquad a_4+a_5-a_2-a_7, 
\]
\[
a_5+a_7-a_4, \qquad a_2+a_4-a_1-a_5,\qquad a_4+a_6-a_3-a_7,
\]
\[
a_3+a_6-a_4, \qquad a_4+a_7-a_5-a_6, \qquad a_1+a_4-a_2-a_3
\]
are non-negative integers. 

Note that every arrow of $Q_\Delta$ is the short diagonal of a unit rhombus. The rhombus condition can be interpreted as saying that for any unit rhombus, the sum across the short diagonal minus the sum across the long diagonal is a non-negative integer.

\smallskip

Let $\mathcal{T}$ be an ideal triangulation of a decorated surface $\hat{S}$. A hive of $\mathcal{T}$ is an assignment of numbers to the vertices of $Q_{\mathcal{T}}$ such that they satisfy the rhombus condition within each ideal triangle of  $\mathcal{T}$. Denote by ${\bf Hive}(\mathcal{T})$ the set of hives of $\mathcal{T}$.

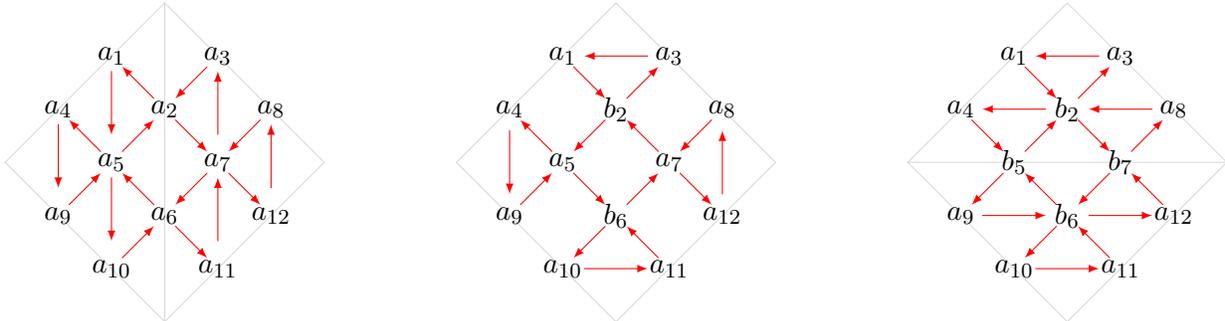
\begin{figure}[H]

\begin{tikzpicture}
\begin{scope}[rotate=-45]
\draw[gray!30] (0,0)--(-3,3)--(0,3)--(0,0)--(-3,0)--(-3,3);
\node at (-1,1) {$a_6$};
\node at (-2,2) {$a_2$};
\node at (0,1) {$a_{11}$};
\node at (0,2) {$a_{12}$};
\node at (-1,0) {$a_{10}$};
\node at (-2,0) {$a_9$};
\node at (-2,3) {$a_3$};
\node at (-1,3) {$a_8$};
\node at (-1,2) {$a_7$};
\node at (-2,1) {$a_5$};
\node at (-3,2) {$a_1$};
\node at (-3,1) {$a_4$};
 \foreach \position in {(-2,1), (-2,2), (-1,1)}
    {\draw[-latex,  red] \position++(-.2,0) -- ++(-.6,0);
    \draw[-latex,  red] \position++(0,-0.8) -- ++(0,0.6);
    \draw[-latex,  red] \position++(-1,0)++(-45:0.2) -- ++(0.6,-0.6);}
 \foreach \position in {(-1,2), (-2,2), (-1,1)}
    {\draw[latex-,  red] \position++(0,0.2) -- ++(0,0.6);
    \draw[latex-,  red] \position++(0.8,0) -- ++(-0.6,0);
    \draw[latex-,  red] \position++(0,1)++(-45:0.2) -- ++(0.6,-0.6);}   
\end{scope}

\begin{scope}[xshift=6cm]
\begin{scope}[rotate=-45]
\draw[gray!30] (-3,3)--(0,3)--(0,0)--(-3,0)--(-3,3);
\node at (-1,1) {$b_6$};
\node at (-2,2) {$b_2$};
\node at (0,1) {$a_{11}$};
\node at (0,2) {$a_{12}$};
\node at (-1,0) {$a_{10}$};
\node at (-2,0) {$a_9$};
\node at (-2,3) {$a_3$};
\node at (-1,3) {$a_8$};
\node at (-1,2) {$a_7$};
\node at (-2,1) {$a_5$};
\node at (-3,2) {$a_1$};
\node at (-3,1) {$a_4$};
 \foreach \position in {(-2,2), (-1,1)}
    {\draw[latex-,  red] \position++(-.2,0) -- ++(-.6,0);
    \draw[latex-,  red] \position++(0,-0.8) -- ++(0,0.6);
    \draw[latex-,  red] \position++(0, 0.8) -- ++(0, -0.6);
    \draw[latex-, red] \position++(0.2,0) -- ++(0.6,0);
   }

\draw[-latex,  red] (-1,0)++(0.2,0.2)--++(0.6,0.6);
\draw[latex-, red] (-3,2)++(0.2,.2)--++(0.6,.6);
\draw[-latex, red] (-2.2,1) -- ++(-.6,0);
    \draw[-latex, red] (-2,.2) -- ++(0,0.6);
    \draw[-latex, red] (-2.8,0.8) -- ++(0.6,-0.6);
    \draw[latex-,  red] (-0.8,2.8) -- ++(0.6,-0.6);
    \draw[-latex, red] (-1,2.8) -- ++(0,-0.6);
    \draw[-latex, red] (-.8,2) -- ++(0.6,0);   
\end{scope}  
\end{scope}

\begin{scope}[xshift=12cm]
\begin{scope}[rotate=-45]
\draw[gray!30] (-3,3)--(0,3)--(0,0)--(-3,0)--(-3,3);
\draw[gray!30] (-3,0)--(0,3);
    \draw[latex-,  red] (-3,2)++(0.2,0.2) -- ++(0.6,0.6);
    \draw[latex-,  red] (-3,1)++(0.2,0.2) -- ++(0.6,0.6);
    \draw[latex-,  red] (-2,2)++(0.2,0.2) -- ++(0.6,0.6);
    \draw[-latex,  red] (-2,0)++(0.2,0.2) -- ++(0.6,0.6);
    \draw[-latex,  red] (-1,0)++(0.2,0.2) -- ++(0.6,0.6);
    \draw[-latex,  red] (-1,1)++(0.2,0.2) -- ++(0.6,0.6);
    
    \draw[-latex,  red] (-3,2)++(0.2,0) -- ++(0.6,0);
    \draw[-latex,  red] (-3,1)++(0.2,0) -- ++(0.6,0);
    \draw[-latex,  red] (-2,2)++(0.2,0) -- ++(0.6,0);
    \draw[latex-,  red] (-2,0)++(0,0.2) -- ++(0,0.6);
    \draw[latex-,  red] (-1,0)++(0,0.2) -- ++(0,0.6);
    \draw[latex-,  red] (-1,1)++(0,0.2) -- ++(0,0.6);

    \draw[-latex,  red] (-2,2)++(0,0.2) -- ++(0,0.6);
    \draw[-latex,  red] (-2,1)++(0,0.2) -- ++(0,0.6);
    \draw[-latex,  red] (-1,2)++(0,0.2) -- ++(0,0.6);
    \draw[latex-,  red] (-2,0)++(0.2,1) -- ++(0.6,0);
    \draw[latex-,  red] (-1,0)++(0.2,1) -- ++(0.6,0);
    \draw[latex-,  red] (-1,1)++(0.2,1) -- ++(0.6,0);

  \node at (-1,1) {$b_6$};
\node at (-2,2) {$b_2$};
\node at (0,1) {$a_{11}$};
\node at (0,2) {$a_{12}$};
\node at (-1,0) {$a_{10}$};
\node at (-2,0) {$a_9$};
\node at (-2,3) {$a_3$};
\node at (-1,3) {$a_8$};
\node at (-1,2) {$b_7$};
\node at (-2,1) {$b_5$};
\node at (-3,2) {$a_1$};
\node at (-3,1) {$a_4$};   
\end{scope}  
\end{scope}

\end{tikzpicture}
\caption{Hives for different ideal triangulations}
\label{ideal.trop.2}
\end{figure}

Let $\{a_\bullet\}\in {\bf Hive}(\mathcal{T})$ be a hive whose restriction to a quadrilateral is illustrated by the left graph of Figure \ref{ideal.trop.2}. Flipping the diagonal of the quadrilateral, we obtain a new ideal triangulation $\mathcal{T}'$. 
We impose the {\it octahedron  relation}
\begin{equation}
\label{oact.rec}
\begin{array}{lll}
a_2+b_2=\max\{a_1+a_7, \, a_5+a_3\},&\hskip 25mm & a_6+b_6=\max\{a_5+a_{11}, \, a_7+a_{10}\};\\
b_5+a_5=\max\{a_4+b_6,\, a_9+b_2\}, && b_7+a_7=\max\{b_2+a_{12}, \,a_8+b_6\}. 
\end{array}
\end{equation}
It gives rise to the numbers $b_2, b_5, b_6, b_7$ to the four new vertices of $Q_{\mathcal{T}'}$ as depicted on the right graph of Figure \ref{ideal.trop.2}. We keep the rest $a_{\bullet}$ invariant. The newly obtained assignment satisfies the rhombus condition. In this way, we obtain a hive for $\mathcal{T}'$. 

The above procedure can be reversed. Any two ideal triangulations $\mathcal{T}$ and $\mathcal{T}'$ can be related by a sequence of flips of diagonals. By applying the relations \eqref{oact.rec} recursively, we obtain a bijection 
\begin{equation}
\label{sacndo}
\varphi_{\mathcal{T}', \mathcal{T}}: \, {\bf Hive}(\mathcal{T})\stackrel{\sim}{\longrightarrow}{\bf Hive}(\mathcal{T}').
\end{equation}

Note that there might be different sequences of flips relating $\mathcal{T}$ and $\mathcal{T}'$. The paper \cite{GS15} shows that the hives can be viewed as a tropicalization of the Fock-Goncharov moduli space $\mathcal{A}_{{\rm PGL}_3, \hat{S}}$ (See Section \ref{section:wa}). As a result, the bijection \eqref{sacndo} is canonical and does not depend on the sequence of flips used. 

In the rest of this section, we prove the following main theorem of this paper.

\begin{theorem} 
\label{Int.hive}
The intersection map $i_{\mathcal{T}}$ in \eqref{IT.map} gives rise to a bijection 
\[
i_{\mathcal{T}}:~ \mathscr{W}_{\hat{S}}^{\mathcal{A}} \stackrel{\sim}{\longrightarrow} {\bf Hive}(\mathcal{T}).
\]
For any two ideal triangulations $\mathcal{T}$ and $\mathcal{T}'$ of $\hat{S}$, the bijections $i_{\mathcal{T}}$ and $i_{\mathcal{T}'}$ are compatible with the transition map $\varphi_{\mathcal{T}',\mathcal{T}}$:
\[
i_{\mathcal{T}'}=\varphi_{\mathcal{T}',\mathcal{T}}\circ i_{\mathcal{T}}.
\]
\end{theorem}

\subsection{Intersection metric}
Let $G=(V, E)$ be a connected oriented graph with the set of vertices $V$ and the set of oriented edges $E$. In preparation for proving Theorem \ref{Int.hive}, we introduce the following intersection metric on $G$. 

\begin{defn} 
 A path $\gamma$ in an oriented graph $G$ is an alternating sequence of vertices and edges $(a_0,e_1,a_1,e_2,a_2,...,e_n,a_n)$ such that $a_{i-1}$ and $a_i$ are endpoints of $e_i$ for each $i$. We define the length of $\gamma$ as
\[
|\gamma|:=\sum_{i=1}^ld_{e_i}(a_{i-1},a_i),
\]
where 
\begin{equation*}
d_{e_i}(a_{i-1},a_i):=\left\{
\begin{aligned}
&\frac{1}{3} \;\;\;  \text{if $e_i$ goes from $a_{i-1}$ to $a_i$};  \\
&\frac{2}{3} \;\;\; \text{if $e_i$ goes from $a_{i}$ to $a_{i-1}$}.
\end{aligned}
\right.
\end{equation*}
Let $a, b$ be a pair of vertices of $G$. Let $P(a,b)$ denote the set of paths from $a$ to $b$. We define the {\em intersection  metric} on $G$ by
\begin{equation}
\label{distan.s}
d(a,b):=\min_{\gamma\in P(a,b)}|\gamma|.
\end{equation}
A path from $a$ to $b$ attaining the minimum value is called a \emph{geodesic}. 
\end{defn}

\begin{remark} \label{lem: triangle inequality}
 It is easy to see that the intersection metric satisfies the following properties:
\begin{itemize}
    \item for any $a,b$, we have $d(a,b)\geq 0$ and the equality attains if and only if $a=b$;
    \item for any $a,b,c$, we have $d(a,b)+d(b,c)\geq d(a,c)$.
\end{itemize} 
However, it is {\em asymmetric}, i.e., $d(a,b)\neq d(b,a)$ in general. 
\end{remark}

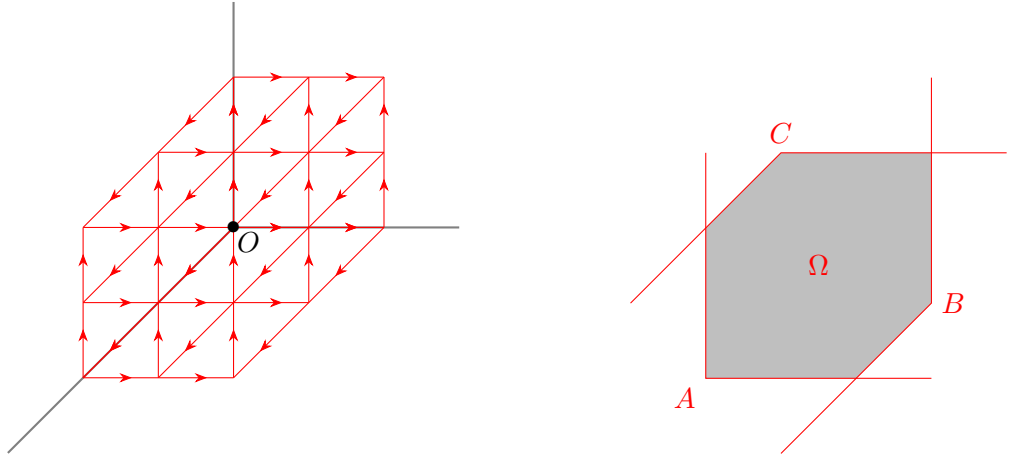
\begin{figure}[H]
\begin{tikzpicture}
\draw[thick,gray] (0,0)--(3,0);
\draw[thick,gray] (0,0)--(0,3);
\draw[thick,gray] (0,0)--(-3,-3);
\foreach \x in {0, 1}
{\foreach \y in {0,...,2}
{ \draw [red, decoration={markings, mark=at position 0.65 with {\arrow{Stealth}}}, postaction=decorate] (\x, \y) -- (\x+1, \y);
\draw [red, decoration={markings, mark=at position 0.65 with {\arrow{Stealth}}}, postaction=decorate]  (-\y, \x-\y) -- (-\y, \x-\y+1) ;
\draw [red, decoration={markings, mark=at position 0.65 with {\arrow{Stealth}}}, postaction=decorate]  (\y-\x, -\x) -- (\y-\x-1, -\x-1) ;
 \draw [red, decoration={markings, mark=at position 0.65 with {\arrow{Stealth}}}, postaction=decorate] (\y, \x)-- (\y, \x+1);
  \draw [red, decoration={markings, mark=at position 0.65 with {\arrow{Stealth}}}, postaction=decorate] (-\x, \y-\x)-- (-\x-1, \y-\x-1);
   \draw [red, decoration={markings, mark=at position 0.65 with {\arrow{Stealth}}}, postaction=decorate] (\x-\y, -\y)-- (\x-\y+1, -\y);}
\foreach \z in {0, 1}
{
\draw [red, decoration={markings, mark=at position 0.65 with {\arrow{Stealth}}}, postaction=decorate] (-\z-1, -\z+\x) -- (-\z, -\z+\x);
\draw [red, decoration={markings, mark=at position 0.65 with {\arrow{Stealth}}}, postaction=decorate] (\x+1, \z+1) -- (\x, \z);
\draw [red, decoration={markings, mark=at position 0.65 with {\arrow{Stealth}}}, postaction=decorate] (\z-\x, -\x-1) -- (\z-\x, -\x);
}
}
\node at (0, 0) {$\bullet$};
\node at (0.2,-0.2) {$O$};
\end{tikzpicture} \hspace{2cm}
\begin{tikzpicture}
    \path [fill=lightgray] (0,0) -- (2,0) -- (3,1) -- (3,3) -- (1,3) -- (0,2) -- cycle;
    \draw [red] (0,3) -- (0,0) node [below left] {$A$} -- (3,0);
    \draw [red] (3,4) -- (3,1) node [right] {$B$} -- (1,-1);
    \draw [red] (-1,1) -- (1,3) node [above] {$C$} -- (4,3);
    \node [red] at (1.5,1.5) [] {$\Omega$};
\end{tikzpicture}
\caption{Left: the graph $\Gamma$ associated with the lattice $\mathbb{Z}^2$. Right: the region $\Omega$ defined by a triple of points $(A,B,C)$.}
\label{G.Gamma}
\end{figure}

Below we consider a particular oriented graph $\Gamma$ with the integer lattice $\mathbb{Z}^2$ as its vertex set and with an edge from $(x,y)$ to each of $(x+1,y)$, $(x,y+1)$, and $(x-1,y-1)$ for every vertex $(x,y)$. See the left picture in Figure \ref{G.Gamma}.

\begin{lemma} 
\label{adnijn}
Let $O$ be the origin and let $A=(x,y)\in \mathbb{Z}^2$. Then 
\[
d(O,A)= \frac{\max\{x+y,\, y-2x, \, x-2y\}}{3}=\left\{
\begin{array}{ll}
\frac{x+y}{3} & \mbox{if $x\geq 0$ and $y\geq 0$;}\\
\frac{y-2x}{3} & \mbox{if $x\leq 0$ and $y\geq x$;} \\
\frac{x-2y}{3} &\mbox{if $y\leq 0$ and $x \geq y$.}\\
\end{array}
\right.
\]
\end{lemma}
\begin{proof} 
We prove the case when $x\geq 0$ and $y\geq 0$. The proof for the rest two cases goes along the same line. Consider the vectors
\[
\vec{u}=(1,0), \qquad \vec{v}=(0,1), \qquad \vec{w}=(-1,-1).
\]
Let $\gamma$ be a path from $O$ to $A$. Let $k, p, q, r, s, t$ be the numbers of the oriented edges on $\gamma$ that are equal to $\vec{u}$, $-\vec{u}$, $\vec{v}$, $-\vec{v}$, $\vec{w}$, and $-\vec{w}$, respectively. We have
\[
(x,y)=(k-p+t-s, q-r+t-s).
\]
Since $k,p,q,r,s,t$ are  non-negative, we get
\[
|\gamma|=\frac{(k+q+s)+2(p+r+t)}{3}\geq \frac{(k+q+2t)-(p+r+2s)}{3}=\frac{x+y}{3}.
\]
Meanwhile, if $\gamma$ is a path that first goes horizontally from $O$ to $(x,0)$ and then vertically to $A=(x,y)$, we have $|\gamma|=\frac{x+y}{3}$, which concludes the proof of the Lemma.
\end{proof}

\begin{lemma} 
\label{distance.asc}
Let $A=(0,p)$ and $B=(q,0)$ with $pq\geq 0$. The path $\gamma$ going vertically from $A$ to $O$ and then horizontally from $O$ to $B$ is a geodesic.    
\end{lemma}

\begin{proof} Without loss of generality we may assume $p,q\geq 0$. By translating the plane by $(0,-p)$, we move $A$ to $O$ and $B$ to $B'=(q,-p)$. Note that due to the translational symmetry of $\Gamma$, $d(A,B)=d(O,B')$. By Lemma \ref{adnijn}, $d(O,B')=\frac{q+2p}{3}$. But this is exactly $|\gamma|$.
\end{proof}

\begin{lemma}
\label{Ylemma.trop}
    Let $A=(a_1,a_2), B=(b_1,b_2), C=(c_1,c_2)$ be vertices on $\Gamma$ such that the region
    \[
    \Omega:=\left\{(x,y)\in \mathbb{Z}^2 ~\middle | ~a_1\leq x\leq b_1, ~~ a_2\leq y\leq c_2, ~~b_2-b_1\leq y-x\leq c_2-c_1\right\}
    \]
    is nonempty (the right picture in Figure \ref{G.Gamma}).  Let $X$ be an arbitrary vertex of $\Gamma$. We have 
    \[
    d(A,X)+d(B,X)+d(C,X)\geq \frac{-a_1-a_2+2b_1-b_2-c_1+2c_2}{3}.
    \]
    It attains the minimum when $X\in \Omega$. 
\end{lemma}
\begin{proof}
Let $X=(x,y)$. By Lemma \ref{adnijn}, we have
\begin{align*}
&d(A,X)+d(B,X)+d(C,X)\\
&\geq \frac{(x-a_1)+(y-a_2)}{3}+\frac{(y-b_2)-2(x-b_1)}{3}+\frac{(x-c_1)-2(y-c_2)}{3}\\
&=\frac{-a_1-a_2-b_2+2b_1-c_1+2c_2}{3}.
\end{align*}
It attains equality when and only when $X\in \Omega$. 
\end{proof}

\subsection{Dual surfacoids}

The dual diskoids, introduced in \cite{FKK13, FONT12}, are dual graphs of webs on disks. 
 We now extend this notion from disks to decorated surfaces. 

\begin{defn} Given an  $\mathcal{A}$-web $W$ on $\hat{S}$, its {\em dual surfacoid $W^*$} is obtained as follows.
\begin{enumerate}
\item Place a vertex in each connected component of $\hat{S}\setminus W$; 
\item For every oriented edge $e$ of $W$ separating two connected components of  $\hat{S}\setminus W$, we  draw an oriented edge $e^*$ connecting the corresponding vertices of $W^*$ such that $i(e,e^*)=\frac{1}{3}$.
\end{enumerate}  
As in Figure \ref{figure:dualgr}, each trivalent sink in $W$ corresponds to a clockwise oriented triangle in $W^*$, and each trivalent source in $W$ corresponds to a counterclockwise oriented triangle in $W^*$. 
\end{defn}

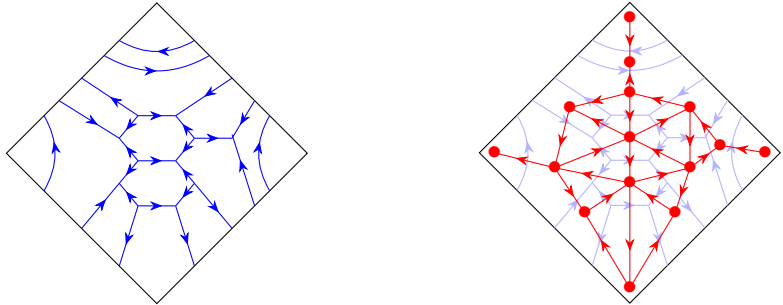
\begin{figure}[H]
\begin{tikzpicture}
\draw (0,-2) -- (2,0) -- (0,2) -- (-2,0) -- cycle;
\draw [blue, decoration={markings, mark=at position 0.5 with {\arrow{Stealth}}}, postaction=decorate] (0.5,1.5) to [bend left] (-0.5,1.5);
\draw [blue, decoration={markings, mark=at position 0.5 with {\arrow{Stealth}}}, postaction=decorate] (-0.7,1.3) to [bend right] (0.7,1.3);
\draw [blue, decoration={markings, mark=at position 0.5 with {\arrow{Stealth}}}, postaction=decorate] (1,1) -- (0.25,0.5);
\draw [blue, decoration={markings, mark=at position 0.65 with {\arrow{Stealth}}}, postaction=decorate] (-0.25,0.5) -- (0.25,0.5);
\draw [blue, decoration={markings, mark=at position 0.5 with {\arrow{Stealth}}}, postaction=decorate] (-0.25,0.5) -- (-1,1);
\draw [blue, decoration={markings, mark=at position 0.65 with {\arrow{Stealth}}}, postaction=decorate] (-0.25,0.5) -- (-0.5,0.2);
\draw [blue, decoration={markings, mark=at position 0.65 with {\arrow{Stealth}}}, postaction=decorate] (0.5,0.2) -- (0.25,0.5);
\draw [blue, decoration={markings, mark=at position 0.65 with {\arrow{Stealth}}}, postaction=decorate] (-1.3,0.7) -- (-0.5,0.2);
\draw [blue, decoration={markings, mark=at position 0.65 with {\arrow{Stealth}}}, postaction=decorate] (0.5,0.2) -- (1,0.2);
\draw [blue, decoration={markings, mark=at position 0.65 with {\arrow{Stealth}}}, postaction=decorate] (1.3,0.7) -- (1,0.2);
\draw [blue, decoration={markings, mark=at position 0.65 with {\arrow{Stealth}}}, postaction=decorate] (1.5,-0.5 ) to [bend left] (1.5,0.5);
\draw [blue, decoration={markings, mark=at position 0.65 with {\arrow{Stealth}}}, postaction=decorate] (1.3,-0.7) -- (1,0.25);
\draw [blue, decoration={markings, mark=at position 0.65 with {\arrow{Stealth}}}, postaction=decorate] (0.5,0.2) -- (0.25,-0.1);
\draw [blue, decoration={markings, mark=at position 0.65 with {\arrow{Stealth}}}, postaction=decorate] (-0.25,-0.1) -- (0.25,-0.1);
\draw [blue, decoration={markings, mark=at position 0.65 with {\arrow{Stealth}}}, postaction=decorate] (-0.25,-0.1) -- (-0.5,0.2);
\draw [blue, decoration={markings, mark=at position 0.65 with {\arrow{Stealth}}}, postaction=decorate] (-0.25,-0.1) -- (-0.5,-0.4);
\draw [blue, decoration={markings, mark=at position 0.65 with {\arrow{Stealth}}}, postaction=decorate] (-1,-1) -- (-0.5,-0.4);
\draw [blue, decoration={markings, mark=at position 0.65 with {\arrow{Stealth}}}, postaction=decorate] (-0.25,-0.7) -- (-0.5,-0.4);
\draw [blue, decoration={markings, mark=at position 0.65 with {\arrow{Stealth}}}, postaction=decorate] (-0.25,-0.7) -- (0.25,-0.7);
\draw [blue, decoration={markings, mark=at position 0.65 with {\arrow{Stealth}}}, postaction=decorate] (0.5,-0.4) -- (0.25,-0.7);
\draw [blue, decoration={markings, mark=at position 0.65 with {\arrow{Stealth}}}, postaction=decorate] (0.5,-0.4) -- (0.25,-0.1);
\draw [blue, decoration={markings, mark=at position 0.65 with {\arrow{Stealth}}}, postaction=decorate] (0.5,-0.4) -- (1,-1);
\draw [blue, decoration={markings, mark=at position 0.65 with {\arrow{Stealth}}}, postaction=decorate] (-0.25,-0.7) -- (-0.5,-1.5);
\draw [blue, decoration={markings, mark=at position 0.65 with {\arrow{Stealth}}}, postaction=decorate] (0.25,-0.7) -- (0.5,-1.5);
\draw [blue, decoration={markings, mark=at position 0.65 with {\arrow{Stealth}}}, postaction=decorate] (-1.5,-0.5) to [bend right] (-1.5,0.5);
\end{tikzpicture} \hspace{2cm}
\begin{tikzpicture}
\draw (0,-2) -- (2,0) -- (0,2) -- (-2,0) -- cycle;
\draw [lightblue, decoration={markings, mark=at position 0.5 with {\arrow{Stealth}}}, postaction=decorate] (0.5,1.5) to [bend left] (-0.5,1.5);
\draw [lightblue, decoration={markings, mark=at position 0.5 with {\arrow{Stealth}}}, postaction=decorate] (-0.7,1.3) to [bend right] (0.7,1.3);
\draw [lightblue, decoration={markings, mark=at position 0.5 with {\arrow{Stealth}}}, postaction=decorate] (1,1) -- (0.25,0.5);
\draw [lightblue, decoration={markings, mark=at position 0.65 with {\arrow{Stealth}}}, postaction=decorate] (-0.25,0.5) -- (0.25,0.5);
\draw [lightblue, decoration={markings, mark=at position 0.5 with {\arrow{Stealth}}}, postaction=decorate] (-0.25,0.5) -- (-1,1);
\draw [lightblue, decoration={markings, mark=at position 0.65 with {\arrow{Stealth}}}, postaction=decorate] (-0.25,0.5) -- (-0.5,0.2);
\draw [lightblue, decoration={markings, mark=at position 0.65 with {\arrow{Stealth}}}, postaction=decorate] (0.5,0.2) -- (0.25,0.5);
\draw [lightblue, decoration={markings, mark=at position 0.65 with {\arrow{Stealth}}}, postaction=decorate] (-1.3,0.7) -- (-0.5,0.2);
\draw [lightblue, decoration={markings, mark=at position 0.65 with {\arrow{Stealth}}}, postaction=decorate] (0.5,0.2) -- (1,0.2);
\draw [lightblue, decoration={markings, mark=at position 0.65 with {\arrow{Stealth}}}, postaction=decorate] (1.3,0.7) -- (1,0.2);
\draw [lightblue, decoration={markings, mark=at position 0.65 with {\arrow{Stealth}}}, postaction=decorate] (1.5,-0.5 ) to [bend left] (1.5,0.5);
\draw [lightblue, decoration={markings, mark=at position 0.65 with {\arrow{Stealth}}}, postaction=decorate] (1.3,-0.7) -- (1,0.25);
\draw [lightblue, decoration={markings, mark=at position 0.65 with {\arrow{Stealth}}}, postaction=decorate] (0.5,0.2) -- (0.25,-0.1);
\draw [lightblue, decoration={markings, mark=at position 0.65 with {\arrow{Stealth}}}, postaction=decorate] (-0.25,-0.1) -- (0.25,-0.1);
\draw [lightblue, decoration={markings, mark=at position 0.65 with {\arrow{Stealth}}}, postaction=decorate] (-0.25,-0.1) -- (-0.5,0.2);
\draw [lightblue, decoration={markings, mark=at position 0.65 with {\arrow{Stealth}}}, postaction=decorate] (-0.25,-0.1) -- (-0.5,-0.4);
\draw [lightblue, decoration={markings, mark=at position 0.65 with {\arrow{Stealth}}}, postaction=decorate] (-1,-1) -- (-0.5,-0.4);
\draw [lightblue, decoration={markings, mark=at position 0.65 with {\arrow{Stealth}}}, postaction=decorate] (-0.25,-0.7) -- (-0.5,-0.4);
\draw [lightblue, decoration={markings, mark=at position 0.65 with {\arrow{Stealth}}}, postaction=decorate] (-0.25,-0.7) -- (0.25,-0.7);
\draw [lightblue, decoration={markings, mark=at position 0.65 with {\arrow{Stealth}}}, postaction=decorate] (0.5,-0.4) -- (0.25,-0.7);
\draw [lightblue, decoration={markings, mark=at position 0.65 with {\arrow{Stealth}}}, postaction=decorate] (0.5,-0.4) -- (0.25,-0.1);
\draw [lightblue, decoration={markings, mark=at position 0.65 with {\arrow{Stealth}}}, postaction=decorate] (0.5,-0.4) -- (1,-1);
\draw [lightblue, decoration={markings, mark=at position 0.65 with {\arrow{Stealth}}}, postaction=decorate] (-0.25,-0.7) -- (-0.5,-1.5);
\draw [lightblue, decoration={markings, mark=at position 0.65 with {\arrow{Stealth}}}, postaction=decorate] (0.25,-0.7) -- (0.5,-1.5);
\draw [lightblue, decoration={markings, mark=at position 0.65 with {\arrow{Stealth}}}, postaction=decorate] (-1.5,-0.5) to [bend right] (-1.5,0.5);
\coordinate (0) at (0,1.8);
\coordinate (1) at (0,1.2);
\coordinate (2) at (0,0.8);
\coordinate (3) at (-0.8,0.6);
\coordinate (4) at (0.8,0.6);
\coordinate (5) at (0,0.2);
\coordinate (6) at (-1.8,0);
\coordinate (7) at (1.2,0.1);
\coordinate (8) at (1.8,0);
\coordinate (9) at (-1,-0.2);
\coordinate (10) at (0,-0.4);
\coordinate (11) at (0.8,-0.2);
\coordinate (12) at (-0.6,-0.8);
\coordinate (13) at (0.6,-0.8);
\coordinate (14) at (0,-1.8);
\foreach \i in {0,...,14}
{
\node [red] at (\i) [] {$\bullet$};
}
\draw [red, decoration={markings, mark=at position 0.65 with {\arrow{Stealth}}}, postaction=decorate] (0) -- (1);
\draw [red, decoration={markings, mark=at position 0.65 with {\arrow{Stealth}}}, postaction=decorate] (2) -- (1);
\draw [red, decoration={markings, mark=at position 0.65 with {\arrow{Stealth}}}, postaction=decorate] (2) -- (3);
\draw [red, decoration={markings, mark=at position 0.65 with {\arrow{Stealth}}}, postaction=decorate] (4) -- (2);
\draw [red, decoration={markings, mark=at position 0.65 with {\arrow{Stealth}}}, postaction=decorate] (2) -- (5);
\draw [red, decoration={markings, mark=at position 0.65 with {\arrow{Stealth}}}, postaction=decorate] (5) -- (3);
\draw [red, decoration={markings, mark=at position 0.65 with {\arrow{Stealth}}}, postaction=decorate] (5) -- (4);
\draw [red, decoration={markings, mark=at position 0.65 with {\arrow{Stealth}}}, postaction=decorate] (9) -- (6);
\draw [red, decoration={markings, mark=at position 0.65 with {\arrow{Stealth}}}, postaction=decorate] (3) -- (9);
\draw [red, decoration={markings, mark=at position 0.65 with {\arrow{Stealth}}}, postaction=decorate] (9) -- (5);
\draw [red, decoration={markings, mark=at position 0.65 with {\arrow{Stealth}}}, postaction=decorate] (5) -- (10);
\draw [red, decoration={markings, mark=at position 0.65 with {\arrow{Stealth}}}, postaction=decorate] (10) -- (9);
\draw [red, decoration={markings, mark=at position 0.65 with {\arrow{Stealth}}}, postaction=decorate] (9) -- (12);
\draw [red, decoration={markings, mark=at position 0.65 with {\arrow{Stealth}}}, postaction=decorate] (12) -- (10);
\draw [red, decoration={markings, mark=at position 0.65 with {\arrow{Stealth}}}, postaction=decorate] (10) -- (14);
\draw [red, decoration={markings, mark=at position 0.65 with {\arrow{Stealth}}}, postaction=decorate] (14) -- (12);
\draw [red, decoration={markings, mark=at position 0.65 with {\arrow{Stealth}}}, postaction=decorate] (4) -- (11);
\draw [red, decoration={markings, mark=at position 0.65 with {\arrow{Stealth}}}, postaction=decorate] (11) -- (5);
\draw [red, decoration={markings, mark=at position 0.65 with {\arrow{Stealth}}}, postaction=decorate] (10) -- (11);
\draw [red, decoration={markings, mark=at position 0.65 with {\arrow{Stealth}}}, postaction=decorate] (7) -- (4);
\draw [red, decoration={markings, mark=at position 0.65 with {\arrow{Stealth}}}, postaction=decorate] (11) -- (7);
\draw [red, decoration={markings, mark=at position 0.65 with {\arrow{Stealth}}}, postaction=decorate] (8) -- (7);
\draw [red, decoration={markings, mark=at position 0.65 with {\arrow{Stealth}}}, postaction=decorate] (11) -- (13);
\draw [red, decoration={markings, mark=at position 0.65 with {\arrow{Stealth}}}, postaction=decorate] (13) -- (10);
\draw [red, decoration={markings, mark=at position 0.65 with {\arrow{Stealth}}}, postaction=decorate] (14) -- (13);
\end{tikzpicture}
\caption{Local picture of an $\mathcal{A}$-web $W$ and its dual surfacoid $W^*$.}
\label{figure:dualgr}
\end{figure}

\begin{example} Continuing as in Example \ref{minimal.ladder}, let $W$ be a minimal ladder in a disk with two marked points. The dual surfacoid $W^*$ is as in the left picture of Figure \ref{fig: bigon regions}. Note that each ``I'' in $W$ corresponds to a rhombus formed by two adjacent triangles in $W^*$. Moreover, by choosing a boundary marked point and a boundary $\alpha$ of the bigon, the ``I''s in the staircase of each ascending arc along $\alpha$ give rise to a parallelogram formed by a row of rhombi; we call such a parallelogram a \emph{pendant}. Moreover, we can joint each pendant with two chains of edges that go counterclockwise around all other pendants and go to the two vertices near the marked points of the bigon. We call the union of a pendant and its two chains of edges a \emph{necklace}. We call each of the two chains of edges a \emph{string} of the necklace. For each necklace, we can travel from the region near one marked point of the bigon, along a string, and then either clockwise or counterclockwise around the pendant, and then along the other string, to get to the region near the other marked point of the bigon; we call these two paths of edges the two \emph{boundary paths} of the necklace. Note that necklaces can only intersect along their boundary paths, and the union of all necklaces is the whole of $W^*$. 
\end{example}

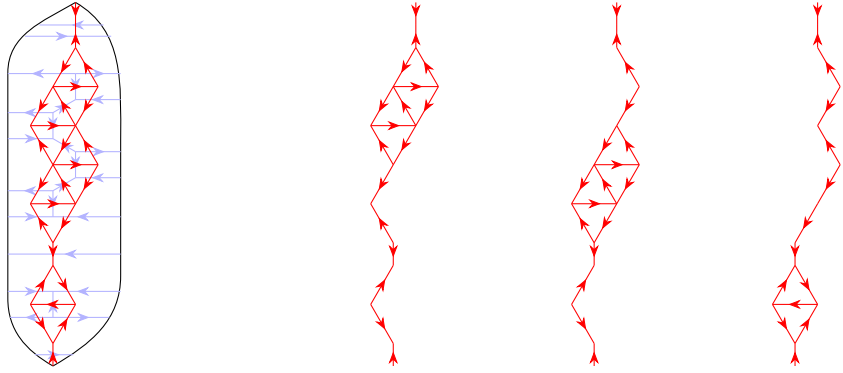
\begin{figure}[H]

    \centering
    \begin{tikzpicture}[scale=0.6]
        \draw (0.5,0.5+0.866*2+1.5+0.866*5) to [out=-30,in=90] (1.5,5) -- (1.5,2) to [out=-90,in=30] (0,0);
        \draw (0.5,0.5+0.866*2+1.5+0.866*5) to [out=-150,in=90] (-1,6.5) -- (-1,1.5) to [out=-90,in=150] (0,0);
        
        \draw [lightblue, decoration={markings, mark=at position 0.65 with {\arrow{Stealth}}}, postaction=decorate] (0,0.5+0.866*2/3) -- (-0.95,0.5+0.866*2/3);
        \draw [lightblue, decoration={markings, mark=at position 0.65 with {\arrow{Stealth}}}, postaction=decorate] (0,0.5+0.866*2/3) -- (1.25,0.5+0.866*2/3);
        \draw [lightblue, decoration={markings, mark=at position 0.65 with {\arrow{Stealth}}}, postaction=decorate] (-0.4,0.25) -- (0.45,0.25);
        
        \draw [lightblue, decoration={markings, mark=at position 0.65 with {\arrow{Stealth}}}, postaction=decorate] (0,0.5+0.866*2/3) -- (0,0.5+0.866*4/3);
        \draw [lightblue, decoration={markings, mark=at position 0.65 with {\arrow{Stealth}}}, postaction=decorate] (-1,0.5+0.866*4/3) -- (0,0.5+0.866*4/3);
        
        \draw [lightblue, decoration={markings, mark=at position 0.65 with {\arrow{Stealth}}}, postaction=decorate] (1.5,0.5+0.866*4/3) -- (0,0.5+0.866*4/3);
        
        \draw [lightblue, decoration={markings, mark=at position 0.5 with {\arrow{Stealth}}}, postaction=decorate] (1.5,0.5+0.866*2+0.25) -- (-1,0.5+0.866*2+0.25);
        
        \draw [lightblue, decoration={markings, mark=at position 0.65 with {\arrow{Stealth}}}, postaction=decorate] (-1,0.5+0.866*8/3+0.5) -- (0,0.5+0.866*8/3+0.5);
        \draw [lightblue, decoration={markings, mark=at position 0.65 with {\arrow{Stealth}}}, postaction=decorate] (1.5,0.5+0.866*8/3+0.5) -- (0,0.5+0.866*8/3+0.5);
        \draw [lightblue, decoration={markings, mark=at position 0.65 with {\arrow{Stealth}}}, postaction=decorate] (0,0.5+0.866*10/3+0.5) -- (0,0.5+0.866*8/3+0.5);
        \draw [lightblue, decoration={markings, mark=at position 0.65 with {\arrow{Stealth}}}, postaction=decorate] (0,0.5+0.866*10/3+0.5) -- (-1, 0.5+0.866*10/3+0.5);
        \draw [lightblue, decoration={markings, mark=at position 0.65 with {\arrow{Stealth}}}, postaction=decorate] (0,0.5+0.866*10/3+0.5) -- (0.5,0.5+0.866*11/3+0.5);
        \draw [lightblue, decoration={markings, mark=at position 0.65 with {\arrow{Stealth}}}, postaction=decorate] (1.5,0.5+0.866*11/3+0.5) -- (0.5,0.5+0.866*11/3+0.5);
        \draw [lightblue, decoration={markings, mark=at position 0.65 with {\arrow{Stealth}}}, postaction=decorate] (0.5,0.5+0.866*13/3+0.5) -- (0.5,0.5+0.866*11/3+0.5);
        \draw [lightblue, decoration={markings, mark=at position 0.65 with {\arrow{Stealth}}}, postaction=decorate] (0.5,0.5+0.866*13/3+0.5) -- (0,0.5+0.866*14/3+0.5);
           
        \draw [lightblue, decoration={markings, mark=at position 0.65 with {\arrow{Stealth}}}, postaction=decorate] (-1,0.5+0.866*14/3+0.5) -- (0,0.5+0.866*14/3+0.5);
        \draw [lightblue, decoration={markings, mark=at position 0.65 with {\arrow{Stealth}}}, postaction=decorate] (0.5,0.5+0.866*13/3+0.5) -- (1.5,0.5+0.866*13/3+0.5);
        
        \draw [lightblue, decoration={markings, mark=at position 0.65 with {\arrow{Stealth}}}, postaction=decorate] (0,0.5+0.866*16/3+0.5) -- (0,0.5+0.866*14/3+0.5);
        \draw [lightblue, decoration={markings, mark=at position 0.65 with {\arrow{Stealth}}}, postaction=decorate] (0,0.5+0.866*16/3+0.5) -- (-1, 0.5+0.866*16/3+0.5);
        \draw [lightblue, decoration={markings, mark=at position 0.65 with {\arrow{Stealth}}}, postaction=decorate] (0,0.5+0.866*16/3+0.5) -- (0.5,0.5+0.866*17/3+0.5);

         \draw [lightblue, decoration={markings, mark=at position 0.65 with {\arrow{Stealth}}}, postaction=decorate] (1.5,0.5+0.866*17/3+0.5) -- (0.5,0.5+0.866*17/3+0.5);
        \draw [lightblue, decoration={markings, mark=at position 0.65 with {\arrow{Stealth}}}, postaction=decorate] (0.5,0.5+0.866*19/3+0.5) -- (0.5,0.5+0.866*17/3+0.5);
        \draw [lightblue, decoration={markings, mark=at position 0.65 with {\arrow{Stealth}}}, postaction=decorate] (0.5,0.5+0.866*19/3+0.5) -- (-1,0.5+0.866*19/3+0.5);
        \draw [lightblue, decoration={markings, mark=at position 0.65 with {\arrow{Stealth}}}, postaction=decorate] (0.5,0.5+0.866*19/3+0.5) -- (1.5,0.5+0.866*19/3+0.5);

(0.5,0.5+0.866*2+0.5+0.866*5)
        
        \draw [lightblue, decoration={markings, mark=at position 0.5 with {\arrow{Stealth}}}, postaction=decorate] (-0.62,0.5+0.866*7+0.75) -- (1.29,0.5+0.866*7+0.75);
        \draw [lightblue, decoration={markings, mark=at position 0.5 with {\arrow{Stealth}}}, postaction=decorate] (1.12,0.5+0.866*7+1) -- (-0.3,0.5+0.866*7+1);

        \draw [red, decoration={markings, mark=at position 0.65 with {\arrow{Stealth}}}, postaction=decorate] (0,0) -- (0,0.5);
        
        \draw [red, decoration={markings, mark=at position 0.65 with {\arrow{Stealth}}}, postaction=decorate] (0,0.5) -- (0.5,0.5+0.866);
        \draw [red, decoration={markings, mark=at position 0.65 with {\arrow{Stealth}}}, postaction=decorate] (0.5,0.5+0.866) -- (-0.5,0.5+0.866);
        \draw [red, decoration={markings, mark=at position 0.65 with {\arrow{Stealth}}}, postaction=decorate](-0.5,0.5+0.866) -- (0,0.5);
        
        \draw [red, decoration={markings, mark=at position 0.65 with {\arrow{Stealth}}}, postaction=decorate] (-0.5,0.5+0.866) -- (0,0.5+0.866*2);
        \draw [red, decoration={markings, mark=at position 0.65 with {\arrow{Stealth}}}, postaction=decorate] (0,0.5+0.866*2) -- (0.5,0.5+0.866);
        
        \draw [red, decoration={markings, mark=at position 0.65 with {\arrow{Stealth}}}, postaction=decorate] (0,1+0.866*2) -- (0,0.5+0.866*2);
        
        \draw [red, decoration={markings, mark=at position 0.65 with {\arrow{Stealth}}}, postaction=decorate] (0.5,0.5+0.866+0.5+0.866*2) --(0,0.5+0.5+0.866*2);
        \draw [red, decoration={markings, mark=at position 0.65 with {\arrow{Stealth}}}, postaction=decorate] (-0.5,0.5+0.866+0.5+0.866*2)--(0.5,0.5+0.866+0.5+0.866*2);
        \draw [red, decoration={markings, mark=at position 0.65 with {\arrow{Stealth}}}, postaction=decorate] (0,0.5+0.5+0.866*2)--(-0.5,0.5+0.866+0.5+0.866*2);
        
        \draw [red, decoration={markings, mark=at position 0.65 with {\arrow{Stealth}}}, postaction=decorate] (0,0.5+0.866*2+0.5+0.866*2)--(-0.5,0.5+0.866+0.5+0.866*2);
        \draw [red, decoration={markings, mark=at position 0.65 with {\arrow{Stealth}}}, postaction=decorate]  (0.5,0.5+0.866+0.5+0.866*2) --(0,0.5+0.866*2+0.5+0.866*2);
        
        \draw [red, decoration={markings, mark=at position 0.65 with {\arrow{Stealth}}}, postaction=decorate] (0,0.5+0.866*2+0.5+0.866*2) -- (1,0.5+0.866*2+0.5+0.866*2);
        \draw [red, decoration={markings, mark=at position 0.65 with {\arrow{Stealth}}}, postaction=decorate] (1,0.5+0.866*2+0.5+0.866*2) --(0.5,0.5+0.866+0.5+0.866*2) ;
        \draw [red, decoration={markings, mark=at position 0.65 with {\arrow{Stealth}}}, postaction=decorate] (1,0.5+0.866*2+0.5+0.866*2) --(0.5,0.5+0.866*3+0.5+0.866*2);
        \draw [red, decoration={markings, mark=at position 0.65 with {\arrow{Stealth}}}, postaction=decorate] (0.5,0.5+0.866*3+0.5+0.866*2) --(0,0.5+0.866*2+0.5+0.866*2);
        
        \draw [red, decoration={markings, mark=at position 0.65 with {\arrow{Stealth}}}, postaction=decorate] (0,0.5+0.866*2+0.5+0.866*2) --(-0.5,0.5+0.866*2+0.5+0.866*3);
        \draw [red, decoration={markings, mark=at position 0.65 with {\arrow{Stealth}}}, postaction=decorate] (-0.5,0.5+0.866*2+0.5+0.866*3) -- (0.5,0.5+0.866*2+0.5+0.866*3);
        \draw [red, decoration={markings, mark=at position 0.65 with {\arrow{Stealth}}}, postaction=decorate]  (0.5,0.5+0.866*2+0.5+0.866*3) -- (0,0.5+0.866*2+0.5+0.866*4);
        \draw [red, decoration={markings, mark=at position 0.65 with 
        {\arrow{Stealth}}}, postaction=decorate]  (0,0.5+0.866*2+0.5+0.866*4) -- (-0.5,0.5+0.866*2+0.5+0.866*3); 
        \draw [red, decoration={markings, mark=at position 0.65 with 
        {\arrow{Stealth}}}, postaction=decorate]  (0,0.5+0.866*2+0.5+0.866*4) -- (1,0.5+0.866*2+0.5+0.866*4);
        \draw [red, decoration={markings, mark=at position 0.65 with 
        {\arrow{Stealth}}}, postaction=decorate]  (1,0.5+0.866*2+0.5+0.866*4)--(0.5,0.5+0.866*2+0.5+0.866*3);
        
        \draw [red, decoration={markings, mark=at position 0.65 with 
        {\arrow{Stealth}}}, postaction=decorate]  (1,0.5+0.866*2+0.5+0.866*4)--(0.5,0.5+0.866*2+0.5+0.866*5);
        \draw [red, decoration={markings, mark=at position 0.65 with 
        {\arrow{Stealth}}}, postaction=decorate]  (0.5,0.5+0.866*2+0.5+0.866*5)--(0,0.5+0.866*2+0.5+0.866*4);

        \draw [red, decoration={markings, mark=at position 0.65 with 
        {\arrow{Stealth}}}, postaction=decorate]  (0.5,0.5+0.866*2+0.5+0.866*5)
        --(0.5,0.5+0.866*2+1+0.866*5);
        \draw [red, decoration={markings, mark=at position 0.65 with 
        {\arrow{Stealth}}}, postaction=decorate]  (0.5,0.5+0.866*2+1.5+0.866*5)
        --(0.5,0.5+0.866*2+1+0.866*5);
   
    \end{tikzpicture}
    \hspace{3cm}
      \begin{tikzpicture}[scale=0.6]
       
        \draw [red, decoration={markings, mark=at position 0.65 with {\arrow{Stealth}}}, postaction=decorate] (0,0) -- (0,0.5);

        \draw [red, decoration={markings, mark=at position 0.65 with {\arrow{Stealth}}}, postaction=decorate](-0.5,0.5+0.866) -- (0,0.5);
        
        \draw [red, decoration={markings, mark=at position 0.65 with {\arrow{Stealth}}}, postaction=decorate] (-0.5,0.5+0.866) -- (0,0.5+0.866*2);

        \draw [red, decoration={markings, mark=at position 0.65 with {\arrow{Stealth}}}, postaction=decorate] (0,1+0.866*2) -- (0,0.5+0.866*2);
        \draw [red, decoration={markings, mark=at position 0.65 with {\arrow{Stealth}}}, postaction=decorate] (0,0.5+0.5+0.866*2)--(-0.5,0.5+0.866+0.5+0.866*2);
        
        \draw [red, decoration={markings, mark=at position 0.65 with {\arrow{Stealth}}}, postaction=decorate] (0,0.5+0.866*2+0.5+0.866*2)--(-0.5,0.5+0.866+0.5+0.866*2);
        \draw [red, decoration={markings, mark=at position 0.65 with {\arrow{Stealth}}}, postaction=decorate] (0.5,0.5+0.866*3+0.5+0.866*2) --(0,0.5+0.866*2+0.5+0.866*2);
        
        \draw [red, decoration={markings, mark=at position 0.65 with {\arrow{Stealth}}}, postaction=decorate] (0,0.5+0.866*2+0.5+0.866*2) --(-0.5,0.5+0.866*2+0.5+0.866*3);
        \draw [red, decoration={markings, mark=at position 0.65 with {\arrow{Stealth}}}, postaction=decorate] (-0.5,0.5+0.866*2+0.5+0.866*3) -- (0.5,0.5+0.866*2+0.5+0.866*3);
        \draw [red, decoration={markings, mark=at position 0.65 with {\arrow{Stealth}}}, postaction=decorate]  (0.5,0.5+0.866*2+0.5+0.866*3) -- (0,0.5+0.866*2+0.5+0.866*4);
        \draw [red, decoration={markings, mark=at position 0.65 with 
        {\arrow{Stealth}}}, postaction=decorate]  (0,0.5+0.866*2+0.5+0.866*4) -- (-0.5,0.5+0.866*2+0.5+0.866*3); 
        \draw [red, decoration={markings, mark=at position 0.65 with 
        {\arrow{Stealth}}}, postaction=decorate]  (0,0.5+0.866*2+0.5+0.866*4) -- (1,0.5+0.866*2+0.5+0.866*4);
        \draw [red, decoration={markings, mark=at position 0.65 with 
        {\arrow{Stealth}}}, postaction=decorate]  (1,0.5+0.866*2+0.5+0.866*4)--(0.5,0.5+0.866*2+0.5+0.866*3);
        
        \draw [red, decoration={markings, mark=at position 0.65 with 
        {\arrow{Stealth}}}, postaction=decorate]  (1,0.5+0.866*2+0.5+0.866*4)--(0.5,0.5+0.866*2+0.5+0.866*5);
        \draw [red, decoration={markings, mark=at position 0.65 with 
        {\arrow{Stealth}}}, postaction=decorate]  (0.5,0.5+0.866*2+0.5+0.866*5)--(0,0.5+0.866*2+0.5+0.866*4);

        \draw [red, decoration={markings, mark=at position 0.65 with 
        {\arrow{Stealth}}}, postaction=decorate]  (0.5,0.5+0.866*2+0.5+0.866*5)
        --(0.5,0.5+0.866*2+1+0.866*5);
        \draw [red, decoration={markings, mark=at position 0.65 with {\arrow{Stealth}}}, postaction=decorate]  (0.5,0.5+0.866*2+1.5+0.866*5) --(0.5,0.5+0.866*2+1+0.866*5);
   
    \end{tikzpicture}
  \hspace{1.5cm}
      \begin{tikzpicture}[scale=0.6]
       
        \draw [red, decoration={markings, mark=at position 0.65 with {\arrow{Stealth}}}, postaction=decorate] (0,0) -- (0,0.5);
        \draw [red, decoration={markings, mark=at position 0.65 with {\arrow{Stealth}}}, postaction=decorate](-0.5,0.5+0.866) -- (0,0.5);
        
        \draw [red, decoration={markings, mark=at position 0.65 with {\arrow{Stealth}}}, postaction=decorate] (-0.5,0.5+0.866) -- (0,0.5+0.866*2);

        \draw [red, decoration={markings, mark=at position 0.65 with {\arrow{Stealth}}}, postaction=decorate] (0,1+0.866*2) -- (0,0.5+0.866*2);
        
        \draw [red, decoration={markings, mark=at position 0.65 with {\arrow{Stealth}}}, postaction=decorate] (0.5,0.5+0.866+0.5+0.866*2) --(0,0.5+0.5+0.866*2);
        \draw [red, decoration={markings, mark=at position 0.65 with {\arrow{Stealth}}}, postaction=decorate] (-0.5,0.5+0.866+0.5+0.866*2)--(0.5,0.5+0.866+0.5+0.866*2);
        \draw [red, decoration={markings, mark=at position 0.65 with {\arrow{Stealth}}}, postaction=decorate] (0,0.5+0.5+0.866*2)--(-0.5,0.5+0.866+0.5+0.866*2);
        
        \draw [red, decoration={markings, mark=at position 0.65 with {\arrow{Stealth}}}, postaction=decorate] (0,0.5+0.866*2+0.5+0.866*2)--(-0.5,0.5+0.866+0.5+0.866*2);
        \draw [red, decoration={markings, mark=at position 0.65 with {\arrow{Stealth}}}, postaction=decorate]  (0.5,0.5+0.866+0.5+0.866*2) --(0,0.5+0.866*2+0.5+0.866*2);
        
        \draw [red, decoration={markings, mark=at position 0.65 with {\arrow{Stealth}}}, postaction=decorate] (0,0.5+0.866*2+0.5+0.866*2) -- (1,0.5+0.866*2+0.5+0.866*2);
        \draw [red, decoration={markings, mark=at position 0.65 with {\arrow{Stealth}}}, postaction=decorate] (1,0.5+0.866*2+0.5+0.866*2) --(0.5,0.5+0.866+0.5+0.866*2) ;
        \draw [red, decoration={markings, mark=at position 0.65 with {\arrow{Stealth}}}, postaction=decorate] (1,0.5+0.866*2+0.5+0.866*2) --(0.5,0.5+0.866*3+0.5+0.866*2);
        \draw [red, decoration={markings, mark=at position 0.65 with {\arrow{Stealth}}}, postaction=decorate] (0.5,0.5+0.866*3+0.5+0.866*2) --(0,0.5+0.866*2+0.5+0.866*2);
        \draw [red, decoration={markings, mark=at position 0.65 with 
        {\arrow{Stealth}}}, postaction=decorate]  (1,0.5+0.866*2+0.5+0.866*4)--(0.5,0.5+0.866*2+0.5+0.866*3);
        \draw [red, decoration={markings, mark=at position 0.65 with 
        {\arrow{Stealth}}}, postaction=decorate]  (1,0.5+0.866*2+0.5+0.866*4)--(0.5,0.5+0.866*2+0.5+0.866*5);
        \draw [red, decoration={markings, mark=at position 0.65 with 
        {\arrow{Stealth}}}, postaction=decorate]  (0.5,0.5+0.866*2+0.5+0.866*5)
        --(0.5,0.5+0.866*2+1+0.866*5);
        \draw [red, decoration={markings, mark=at position 0.65 with 
        {\arrow{Stealth}}}, postaction=decorate]  (0.5,0.5+0.866*2+1.5+0.866*5)
        --(0.5,0.5+0.866*2+1+0.866*5);
   
    \end{tikzpicture}
    \hspace{1.5cm}
      \begin{tikzpicture}[scale=0.6]
       
        \draw [red, decoration={markings, mark=at position 0.65 with {\arrow{Stealth}}}, postaction=decorate] (0,0) -- (0,0.5);
        \draw [red, decoration={markings, mark=at position 0.65 with {\arrow{Stealth}}}, postaction=decorate](-0.5,0.5+0.866) -- (0,0.5);
        \draw [red, decoration={markings, mark=at position 0.65 with {\arrow{Stealth}}}, postaction=decorate] (0.5,0.5+0.866) -- (-0.5,0.5+0.866);
        
        \draw [red, decoration={markings, mark=at position 0.65 with {\arrow{Stealth}}}, postaction=decorate] (-0.5,0.5+0.866) -- (0,0.5+0.866*2);
        \draw [red, decoration={markings, mark=at position 0.65 with {\arrow{Stealth}}}, postaction=decorate] (0,0.5) -- (0.5,0.5+0.866);

        \draw [red, decoration={markings, mark=at position 0.65 with {\arrow{Stealth}}}, postaction=decorate] (0,0.5+0.866*2) -- (0.5,0.5+0.866);
        
        \draw [red, decoration={markings, mark=at position 0.65 with {\arrow{Stealth}}}, postaction=decorate] (0,1+0.866*2) -- (0,0.5+0.866*2);
        
        \draw [red, decoration={markings, mark=at position 0.65 with {\arrow{Stealth}}}, postaction=decorate] (0.5,0.5+0.866+0.5+0.866*2) --(0,0.5+0.5+0.866*2);
       
        \draw [red, decoration={markings, mark=at position 0.65 with {\arrow{Stealth}}}, postaction=decorate] (1,0.5+0.866*2+0.5+0.866*2) --(0.5,0.5+0.866+0.5+0.866*2) ;
        \draw [red, decoration={markings, mark=at position 0.65 with {\arrow{Stealth}}}, postaction=decorate] (1,0.5+0.866*2+0.5+0.866*2) --(0.5,0.5+0.866*3+0.5+0.866*2);
        
        \draw [red, decoration={markings, mark=at position 0.65 with 
        {\arrow{Stealth}}}, postaction=decorate]  (1,0.5+0.866*2+0.5+0.866*4)--(0.5,0.5+0.866*2+0.5+0.866*3);
        \draw [red, decoration={markings, mark=at position 0.65 with 
        {\arrow{Stealth}}}, postaction=decorate]  (1,0.5+0.866*2+0.5+0.866*4)--(0.5,0.5+0.866*2+0.5+0.866*5);
        \draw [red, decoration={markings, mark=at position 0.65 with 
        {\arrow{Stealth}}}, postaction=decorate]  (0.5,0.5+0.866*2+0.5+0.866*5)
        --(0.5,0.5+0.866*2+1+0.866*5);
        \draw [red, decoration={markings, mark=at position 0.65 with 
        {\arrow{Stealth}}}, postaction=decorate]  (0.5,0.5+0.866*2+1.5+0.866*5)
        --(0.5,0.5+0.866*2+1+0.866*5);
   
    \end{tikzpicture}
    
    \caption{Left: a bigon region and an example of a dual surfacoid $W^*$. Right: the three necklaces in the surfacoid $W^*$ defined by the three ascending arcs along the left boundary of the bigon with respect to the top marked point.}
    \label{fig: bigon regions}
\end{figure}

\begin{example} Continuing as in Example \ref{tri-redu-web}, let $W$ be a reduced $\mathcal{A}$ web on a disk $\Delta$ with three marked points.  The dual surfacoid $W^*$ is as in Figure \ref{fig: two types of regions}. It is the union of a triangle consisting of oriented $3$-cycles and three possible oriented curves going from each corner of the triangle to a marked point. We call dual surfacoid $W^*$ of this form a \emph{net}; in particular, we refer to the subgraph spanned by oriented $3$-cycles as the \emph{mesh} of the net $W^*$ and the three curves as the \emph{strings} of the net $W^*$.
Let $\alpha$ be a boundary interval of $\Delta$. 
Tracing the edges of $W$ intersected by $\alpha$ results in a path in $W^*$ with two strings and one side of the mesh, which we continue to call $\alpha$ and refer to it as a \emph{boundary path} of the net $W^*$.

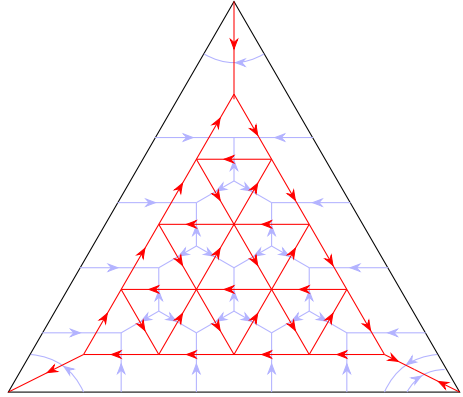
\begin{figure}[H]
    \centering
    \begin{tikzpicture}
    \draw (-3*0.5+3+0.5,6*0.866-0.5) -- (-1,-0.5) -- (5,-0.5) -- cycle;
    \foreach \i in {0,1,2}
     {
     \foreach \j in {\i,...,2}
     {
     \draw [lightblue, decoration={markings, mark=at position 0.5 with {\arrow{Stealth}}}, postaction=decorate] (-\i*0.5+\j+1,\i*0.866+0.866*2/3) -- (-\i*0.5+\j+1,\i*0.866+0.866*4/3);
     \draw [lightblue, decoration={markings, mark=at position 0.5 with {\arrow{Stealth}}}, postaction=decorate] (-\i*0.5+\j+1,\i*0.866+0.866*2/3) -- (-\i*0.5+\j+0.5,\i*0.866+0.866/3);
     \draw [lightblue, decoration={markings, mark=at position 0.5 with {\arrow{Stealth}}}, postaction=decorate] (-\i*0.5+\j+1,\i*0.866+0.866*2/3) -- (-\i*0.5+\j+1.5,\i*0.866+0.866/3);
     }
     }
     \foreach \i in {0,...,3}
     {
     \draw [lightblue, decoration={markings, mark=at position 0.5 with {\arrow{Stealth}}}, postaction=decorate] (\i*0.5-0.55, \i*0.866+0.866/3)--(-\i*0.5+\i+0.5,\i*0.866+0.866/3);
     \draw [lightblue, decoration={markings, mark=at position 0.5 with {\arrow{Stealth}}}, postaction=decorate] (4.55-\i*0.5, \i*0.866+0.866/3)--(-\i*0.5+3+0.5,\i*0.866+0.866/3);
     \draw [lightblue, decoration={markings, mark=at position 0.5 with {\arrow{Stealth}}}, postaction=decorate] (\i+0.5, -0.5)--(\i+0.5,0.866/3);
     }
     \foreach \i in {0,...,3}
     {
     \foreach \j in {\i,...,3}
     {
     \draw [red, decoration={markings, mark=at position 0.65 with {\arrow{Stealth}}}, postaction=decorate] (-\i*0.5+\j,\i*0.866) -- (-\i*0.5+\j+0.5,\i*0.866+0.866);
     \draw [red, decoration={markings, mark=at position 0.65 with {\arrow{Stealth}}}, postaction=decorate] (-\i*0.5+\j+0.5,\i*0.866+0.866) -- (-\i*0.5+\j+1,\i*0.866);
     \draw [red, decoration={markings, mark=at position 0.65 with {\arrow{Stealth}}}, postaction=decorate] (-\i*0.5+\j+1,\i*0.866) -- (-\i*0.5+\j,\i*0.866);
     }
     }
     \draw [lightblue, decoration={markings, mark=at position 0.5 with {\arrow{Stealth}}}, postaction=decorate] (2.4,4) to [bend left] (1.6,4);
     \draw [lightblue, decoration={markings, mark=at position 0.5 with {\arrow{Stealth}}}, postaction=decorate] (0,-0.5) to [bend right] (-0.7,0);
     \draw [lightblue, decoration={markings, mark=at position 0.5 with {\arrow{Stealth}}}, postaction=decorate] (4.3,-0.5) to [bend left] (4.8,-0.2);
     \draw [lightblue, decoration={markings, mark=at position 0.5 with {\arrow{Stealth}}}, postaction=decorate] (4.7,0) to [bend right] (4,-0.5);
     \draw [red, decoration={markings, mark=at position 0.5 with {\arrow{Stealth}}}, postaction=decorate] (0,0) -- (-1,-0.5);
     \draw [red, decoration={markings, mark=at position 0.5 with {\arrow{Stealth}}}, postaction=decorate] (5,-0.5) -- (4.4,-0.2);
     \draw [red, decoration={markings, mark=at position 0.5 with {\arrow{Stealth}}}, postaction=decorate] (4,0) -- (4.4,-0.2);
     \draw [red, decoration={markings, mark=at position 0.5 with {\arrow{Stealth}}}, postaction=decorate] (2,6*0.866-0.5) -- (2,6*0.866-1.8);
    \end{tikzpicture}
    \caption{A triangle region}
    \label{fig: two types of regions}
\end{figure}
\end{example}

\begin{proposition}\label{prop: refinement} The dual surfacoid $W^*$ for a reduced web $W$ can be obtained by gluing nets and necklaces along their boundary paths. 
\end{proposition} 
\begin{proof} It follows from Proposition \ref{proposition:gp}.
\end{proof}

\begin{defn} Let $\alpha$ be a boundary path of a net or a necklace and let $A$ and $B$ be vertices on $\alpha$. The path $\gamma_{AB}$ on $\alpha$ that goes directly from $A$ to $B$ is called the {\it straight path} from $A$ to $B$. 
\end{defn}

\begin{lemma}\label{lem: homotoping a path} 
Let $W^*$ be a net or a necklace and let $A$ and $B$ be vertices on a boundary path $\alpha$ of $W^*$. Then the straight path $\gamma_{AB}$ is a geodesic on $W^*$. 
\end{lemma}
\begin{proof} 
Let $\gamma$ be an arbitrary path from $A$ to $B$. It is equivalent to show that $|\gamma|\geq |\gamma_{AB}|$. 

First, we claim that if $W^*$ is a net, then we may assume $\gamma$ never travels along any edges inside the string $\beta$ that is not part of $\alpha$. To see this, let $C$ be the vertex jointing $\beta$ to the mesh. If $\gamma$ ever travels along an edge in $\beta$, then there must be a subpath $\gamma'$ of $\gamma$ that goes from $C$ into $\beta$ and then back out to $C$; removing this subpath $\gamma'$ from $\gamma$ shortens $\gamma$. 

Under the assumption above, we can find vertices $A=A_0, A_1,A_2,\dots, A_n=B$ in $\gamma\cap \alpha$ to split $\gamma$ into consecutive paths $\gamma_1,  \ldots, \gamma_n$ (with each $\gamma_i$ goes from $A_{i-1}$ to $A_i$) such that 
 
\begin{itemize}
\item[(a)] $\gamma_i$ is contained in the mesh (when $W^*$ is a net);
\item[(b)] $\gamma_i$ is contained in the pendant (when $W^*$ is a necklace);
\item[(c)] $\gamma_i$ is contained in $\alpha$;
\end{itemize}

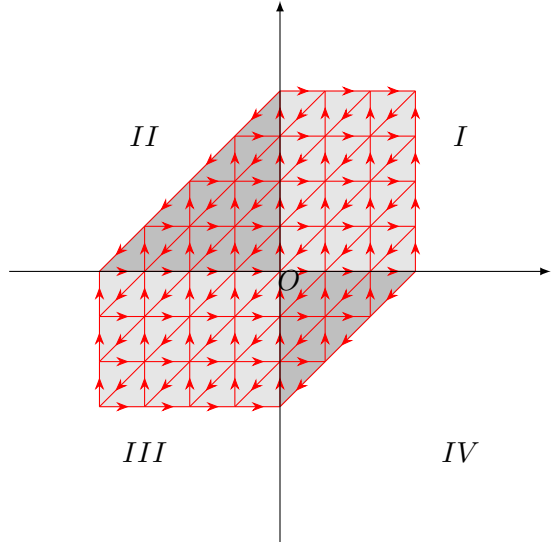
\begin{figure}[H]
\begin{tikzpicture}[scale=0.6]
\draw[gray!60, fill=gray!50] (-4,0)--(0,0)--(0,4)--(-4,0); 
\draw[gray!60, fill=gray!50] (0,-3)--(0,0)--(3,0)--(0,-3);
\draw[gray!20, fill=gray!20] (0,0)--(3,0)--(3,4)--(0,4)--(0,0);
\draw[gray!20, fill=gray!20] (-4,0)--(0,0)--(0,-3)--(-4,-3)--(-4,0);
\foreach \x in {0, ..., 3}
{\foreach \y in {0, ..., \x}
{ \draw [red, decoration={markings, mark=at position 0.65 with {\arrow{Stealth}}}, postaction=decorate] (\y-\x-1, \y) -- (\y-\x, \y);
 \draw [red, decoration={markings, mark=at position 0.65 with {\arrow{Stealth}}}, postaction=decorate] (-\y, \x-\y)-- (-\y, \x-\y+1);
 \draw [red, decoration={markings, mark=at position 0.65 with {\arrow{Stealth}}}, postaction=decorate] (\x-3, \y+1) -- (\x-4, \y);}}
\foreach \x in {0, ..., 2}
{\foreach \y in {0, ..., \x}
{ \draw [red, decoration={markings, mark=at position 0.65 with {\arrow{Stealth}}}, postaction=decorate] (\x-\y, -\y) -- (\x-\y+1, -\y);
 \draw [red, decoration={markings, mark=at position 0.65 with {\arrow{Stealth}}}, postaction=decorate] (\y, \y-\x-1)-- (\y, \y-\x);
 \draw [red, decoration={markings, mark=at position 0.65 with {\arrow{Stealth}}}, postaction=decorate] (3-\x, -\y) -- (2-\x, -\y-1);}} 
 \foreach \x in {0,...,3}
 {
\foreach \y in {0,...,2}
{
\draw [red, decoration={markings, mark=at position 0.65 with {\arrow{Stealth}}}, postaction=decorate] (\y+1,\x)--(\y+1, \x+1);
\draw [red, decoration={markings, mark=at position 0.65 with {\arrow{Stealth}}}, postaction=decorate] (\y,\x+1)--(\y+1, \x+1);
\draw [red, decoration={markings, mark=at position 0.65 with {\arrow{Stealth}}}, postaction=decorate] (\y+1,\x+1)--(\y, \x);
\draw [red, decoration={markings, mark=at position 0.65 with {\arrow{Stealth}}}, postaction=decorate] (-\x-1,-\y-1)--(-\x, -\y-1);
\draw [red, decoration={markings, mark=at position 0.65 with {\arrow{Stealth}}}, postaction=decorate] (-\x,-\y)--(-\x-1, -\y-1);
\draw [red, decoration={markings, mark=at position 0.65 with {\arrow{Stealth}}}, postaction=decorate] (-\x-1,-\y-1)--(-\x-1, -\y);
}
 }
\node at (0.2,-0.2) {$O$};
\node at (4,3) {$I$};
\node at (-3,3) {$II$};
\node at (-3,-4) {$III$};
\node at (4,-4) {$IV$};
\draw[-latex] (-6,0) --(6,0);
\draw[-latex] (0,-6)--(0,6);
\end{tikzpicture}
\caption{Subgraphs of $\Gamma$.}
\label{G.Gamma.2}
\end{figure}

As in Figure \ref{G.Gamma.2}, we place the mesh in the second or the fourth quadrant and place a pendant in the first or the third quadrant, realizing them as subgraphs of $\Gamma$. Let $\gamma_{A_{i-1}A_i}$ be the straight path from $A_{i-1}$ to $A_i$. By Lemma \ref{distance.asc}, we have
\[
|\gamma_i|\geq |\gamma_{A_{i-1}A_i}|.
\]
Therefore
\[
|\gamma|=\sum_{i} |\gamma_{i}|\geq \sum_{i} |\gamma_{A_{i-1}A_i}|\geq |\gamma_{AB}|. \qedhere
\]
\end{proof}

\begin{lemma}\label{lem: homotoping a tripod} Let $W^*$ be a net or a necklace and let $A_1, A_2,A_3$ be vertices on a boundary path $\alpha$ of $W^*$. Let $m$ be the minimum of 
\[
d(A_1,X)+d(A_2,X)+d(A_3,X)
\]
as $X$ varies over all vertices of $W^*$. Then there exists a vertex $Y$ on $\alpha$ such that $d(A_1,Y)+d(A_2,Y)+d(A_3,Y)=m$. 
\end{lemma}
\begin{proof} 
For $i\in\{1,2,3\}$, let $\gamma_{i}$ be a path from $A_i$ to $X$. It suffices to show that there exists a vertex $Y$ on $\alpha$ such that 
\begin{equation}
\label{wcndo}
d(A_1,Y)+d(A_2,Y)+d(A_3,Y) \leq |\gamma_{1}|+|\gamma_{2}|+|\gamma_{3}|.
\end{equation}

If $X$ is on $\alpha$, then \eqref{wcndo} follows directly.

If $W^*$ is a net and $X$ is on the string $\beta$ that is not part of $\alpha$, then we may replace $X$ by the vertex $C$ where $\beta$ joints the mesh, and this replacement shortens all three paths $\gamma_{i}$. Thus, without loss of generality, we assume that $X$ is either contained in the mesh (when $W^*$ is a net) or in the pendant (when $W^*$ is a necklace). 

For each path $\gamma_i$, let us assume that $A_i$ is the only vertex contained in $\alpha$. In this case, the paths  $\gamma_{1}$, $\gamma_{2}$, and $\gamma_{3}$ are all contained in the mesh (when $W^*$ is a net) or in the pendant (when $W^*$ is a necklace). As in Figure \ref{G.Gamma.2}, we may realize them as subgraphs on $\Gamma$ such that the vertices $A_1$, $A_2$, and $A_3$ are on the axes of $\mathbb{Z}^2$. It is then reduced into the following three cases
\begin{itemize}
\item $A_1=(a,0)$, $A_2=(b,0)$, $A_3=(c,0)$, where $c\leq a\leq b$;
\item $A_1=(0,a)$, $A_2=(b,0)$, $A_3=(0,c)$, where $b\geq 0$, and $c\geq a\geq 0$;
\item $A_1=(0,a)$, $A_2=(0,b)$, $A_3=(c,0)$, where $b\leq a\leq 0$, and $c\leq 0$. 
\end{itemize}
Then the Lemma follows directly from Lemmas \ref{distance.asc} and \ref{Ylemma.trop}.

Now for a general path $\gamma_{i}$, let $B_i$ be the last vertex on $\gamma_{i}$ such that $B_i\in \alpha$. It splits $\gamma_{i}$ into two paths: the path $\gamma_{i}'$ from $A_i$ to $B_i$, and the path $\gamma_{i}''$ from $B_i$ to $X$. Then 
\begin{equation}
\label{l11}
|\gamma_{i}|=|\gamma_{i}'|+|\gamma_{i}''|.
\end{equation} 
By the above discussion, there exists a $Y\in \alpha$, such that
\begin{equation}
\label{l12}
 d(B_1,Y)+d(B_2, Y)+d(B_3,Y) \leq |\gamma_{1}''|+|\gamma_{2}''|+|\gamma_{3}''|.
\end{equation}
Meanwhile
\begin{equation}
\label{l13}
d(A_i,Y)\leq |\gamma_{i}'|+d(B_i,Y).
\end{equation}
Combining \eqref{l11}-\eqref{l13}, we obtain the inequality \eqref{wcndo}.
\end{proof}

\subsection{From webs to hives {}}
We first prove Theorem \ref{Int.hive} when $\hat{S}=\Delta$ is a triangle. 

\begin{lemma} 
\label{dwndij}
Let $W\in \mathscr{W}_\Delta^{\mathcal{A}}$ be a reduced web parametrized by $(x,y,z,t,u,v,w)\in \mathbb{Z}\times \mathbb{N}^6$ as in Example \ref{tri-redu-web}. Let $(a_1,\ldots,a_7)$ be the coordinates of $i_\Delta(W)$ associated with the vertices of $Q_\Delta$ as in Figure \ref{quiver.delta}. Then 
\begin{align*}
a_1&=\frac{2t+u+2w+v+\max\{2x,-x\}}{3}, & a_2&=\frac{2w+v+2z+y+\max\{x, -2x\}}{3},\\
a_3&=\frac{2v+w+2u+t+\max\{x,-2x\}}{3}, & a_4&=\frac{2v+w+2t+u+2z+y+\max\{3x,-3x\}}{3},\\
a_5&=\frac{2v+w+2y+z+\max\{2x,-x\}}{3}, & a_6&=\frac{2z+y+2u+t+\max\{2x,-x\}}{3},\\
a_7&=\frac{2t+u+2y+z+\max\{x,-2x\}}{3}. & &
\end{align*}
\end{lemma}

\begin{proof}
Recall that the dual surfacoid $W^*$, which is a net consisting of a mesh with three strings attached. We denote the endpoints of the strings by $A,A',B,B',C,C'$, as in Figure \ref{bracn,ys}.

\begin{figure}[H]
\begin{tikzpicture}
\foreach \i in {0,...,2}
     {
     \foreach \j in {\i,...,2}
     {
     \draw [red, decoration={markings, mark=at position 0.65 with {\arrow{Stealth}}}, postaction=decorate] (-\i*0.5+\j,\i*0.866) -- (-\i*0.5+\j+0.5,\i*0.866+0.866);
     \draw [red, decoration={markings, mark=at position 0.65 with {\arrow{Stealth}}}, postaction=decorate] (-\i*0.5+\j+0.5,\i*0.866+0.866) -- (-\i*0.5+\j+1,\i*0.866);
     \draw [red, decoration={markings, mark=at position 0.65 with {\arrow{Stealth}}}, postaction=decorate] (-\i*0.5+\j+1,\i*0.866) -- (-\i*0.5+\j,\i*0.866);
     }
     }
   \draw [red, decoration={markings, mark=at position 0.65 with {\arrow{Stealth}}}, postaction=decorate]
   (-150:2)--++(30:1);
     \draw [red, decoration={markings, mark=at position 0.65 with {\arrow{Stealth}}}, postaction=decorate]
   (0,0)--++(-150:1);
   \draw [red, decoration={markings, mark=at position 0.65 with {\arrow{Stealth}}}, postaction=decorate]
   (60:3)--++(0,1);
    \draw [red, decoration={markings, mark=at position 0.65 with {\arrow{Stealth}}}, postaction=decorate]
   (1.5,4.6)--(1.5,3.6);
    \draw [red, decoration={markings, mark=at position 0.65 with {\arrow{Stealth}}}, postaction=decorate]
   (3,0)--++(-30:1);
    \draw [red, decoration={markings, mark=at position 0.65 with {\arrow{Stealth}}}, postaction=decorate]
   (3,0)++(-30:2)--++(150:1);
\node[label=above:$A$] at (1.5,4.6) {$\bullet$};
\node at (1.5,3.6) {$\bullet$};
\node[label=right:$A'$] at (60:3) {$\bullet$};
\node[label=left:$B$] at (-150:2) {$\bullet$};
\node at (-150:1) {$\bullet$};
\node at (0,0) {$\bullet$};
\node at (0,0.4) {$B'$};
\node at (3,0) {$\bullet$};
\node at (3,0.4) {$C'$};
\node at (3.866,-.5) {$\bullet$};
\node[label=right:$C$] at (4.732,-1) {$\bullet$};
\end{tikzpicture}
\caption{The dual surfacoid $W^*$ on a triangle}
\label{bracn,ys}
\end{figure}
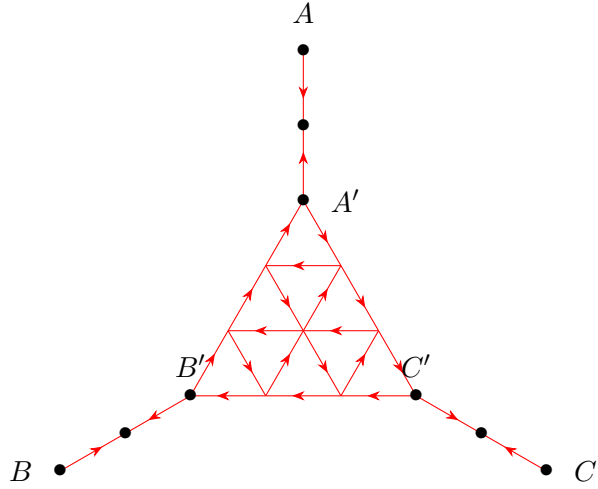

Let $V$ be an inward tripod as in Figure \ref{figure:triA}. Let $X$ be an arbitrary vertex on $W^*$. By definition, we have
\[
a_4=\mathbb{I}([W],[V])=\min_{X \in W^*} \{d(A,X)+d(B,X)+d(C,X)\}.
\]
Note that if $X$ is placed on one of the strings, say $AA'$, then we can move $X$ to the vertex $A'$ without increasing the value of the sum $d(A,X)+d(B,X)+d(C,X)$. Thus, we may assume that $X$ is inside the mesh. On the one hand, by Lemma \eqref{Ylemma.trop}, we get 
\[
d(A',X)+d(B',X)+d(C',X)=|x|.
\]
On the other hand, it is not hard to see that $d(A,X)=d(A,A')+d(A',X)$, $d(B,X)=d(B,B')+d(B',X)$, and $d(C,X)=d(C,A')+d(C',X)$. Therefore
\begin{align*}
a_4=&d(A,A')+d(A',X)+d(B,B')+d(B',X)+d(C,C')+d(C',X)\\
=& \frac{2v+w+2t+u+2z+y+\max\{3x,-3x\}}{3}
\end{align*}

If $V$ is a boundary arc as in Figure \ref{figure:triA}, then by Lemma  \ref{lem: homotoping a path}, the intersection $i(W,V)$ attains the minimum when $V$ is a straight path on $W^*$. In this way, we obtain the values of the rest $a_i$'s as desired.
\end{proof}

\begin{proposition} 
\label{hive=red=tri}
    The map $i_\Delta$ is a bijection from $\mathscr{W}_\Delta^\mathcal{A}$ to ${\bf Hive}(\Delta)$.
\end{proposition}

\begin{proof}
First, we show that the data $i_\Delta(W)$ satisfies the rhombus condition. Indeed, by Lemma \ref{dwndij}, 
\[
a_1+a_2-a_4=w; \qquad a_3+a_4-a_1-a_6=v+\max\{0,-x\};\qquad a_4+a_5-a_2-a_7=v+\max\{x,0\}
\]
are all non-negative integers. The same hold for the rest six values. 

Conversely, let $(a_1,\ldots, a_7)$ be a hive, then we get $(x,y,z,u,v,w,t)\in \mathbb{Z}\times \mathbb{N}^6$, where
\begin{align*}
x&= a_1+a_5+a_6-a_2-a_3-a_7,\\
w&=a_1+a_2-a_4, \quad \quad u=a_3+a_6-a_4, \quad \quad y=a_5+a_7-a_4, \\
v&=\min\{a_3+a_4-a_6-a_1, ~a_4+a_5-a_2-a_7\},\\
t&=\min\{a_1+a_4-a_3-a_2,~a_4+a_7-a_6-a_5\},\\
z&=\min\{a_2+a_4-a_1-a_5,~ a_4+a_6-a_3-a_7\}.
\end{align*}
Therefore, the map $i_\Delta$ is a bijection.
\end{proof}

Now let $\hat{S}$ be an arbitrary decorated surface.
Fix a reduced web $W\in \mathscr{W}_{\hat{S}}^{\mathcal{A}}$. 
Let $\gamma$ be an oriented path on $\hat{S}$ that connects two components of $\hat{S}\backslash W$. 
We assume that $\gamma$ intersects with $W$ transversely and consider the intersection number $i(W, \gamma)$ via \eqref{definition:int}. Note that $\gamma$ corresponds to a path on the dual surfacoid $W^*$, and $i(W,\gamma)$ is equal to the intersection metric length of that path. Abusing notation, we define the length of $\gamma$ as 
\[|\gamma|:= i(W,\gamma).\]

Let $\mathcal{T}_2$ be a split ideal triangulation of $\hat{S}$. Following Propositions \ref{proposition:gp} and \ref{prop: refinement}, 
we can split the bigons in $\mathcal{T}_2$ into smaller bigons such that the restriction of the dual surfacoids $W^*$ to each smaller bigon is a necklace, and to each triangle is a net. Denote by resulted decomposition of $\hat{S}$ by ${\mathcal{T}}_2'$. After necessarily perturbing the  edges of ${\mathcal{T}}_2'$, we assume that 
\begin{itemize}
\item the path $\gamma$ intersects transitively with the  edges of ${\mathcal{T}}_2'$;
\item the endpoints of $\gamma$ are not on the edges of ${\mathcal{T}}_2'$. 
\end{itemize}
 We define the \emph{crossing sequence} of $\gamma$ to be the sequence of edges $(\alpha_1,\dots, \alpha_n)$ of ${\mathcal{T}}_2'$ that $\gamma$ crosses. Let $(B_1, \ldots, B_n)$ be the corresponding intersecting points of $\gamma$ with the edges $\alpha_i$'s. They divide $\gamma$ into paths $\gamma_0$, ..., $\gamma_n$.
When $\alpha_i=\alpha_{i+1}=\alpha$, let us replace the subpath $\gamma_i$ from $B_{i}$ and $B_{i+1}$ by the path along $\alpha$, and then move it slightly to the other side of $\alpha$. In this way, we obtain a new path $\gamma'$ as in Figure \ref{c}. Note that $\gamma'$ is homotopy equivalent to $\gamma$.

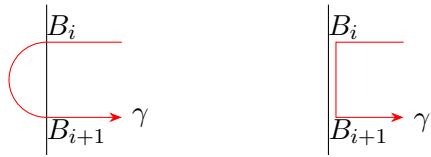
\begin{figure}[H]
    \centering
    \begin{tikzpicture}
    \draw (0,0) --  (0,2);
        \draw [red, -Stealth] (1,1.5) -- (0,1.5) arc (90:270:0.5) -- (1,0.5) node [right, black] {$\gamma$};
        \node at (0.2,1.7) {$B_i$};
        \node at (0.4,0.3) {$B_{i+1}$};
        \end{tikzpicture}\hspace{2cm}
    \begin{tikzpicture}
    \draw (0,0) -- (0,2);
        \node at (0.2,1.7) {$B_i$};
        \node at (0.4,0.3) {$B_{i+1}$};
        \draw [red, -Stealth] (1,1.5) -- (0.1,1.5) -- (0.1, 0.5) -- (1,0.5) node [right, black] {$\gamma'$};
        \end{tikzpicture}
        
    \caption{Local modification of $\gamma$.}
    \label{c}
\end{figure}

\begin{lemma}\label{lem: shorter representative} We have $|\gamma'|\leq |\gamma|$.
\end{lemma}
\begin{proof} It is a direct consequence of  Lemma \ref{lem: homotoping a path}.
\end{proof}

\begin{remark}\label{rmk: shortest representative} As a consequence of Lemma \ref{lem: shorter representative}, when we try to minimize the intersection metric length of a path $\gamma$, we may assume without loss of generality that the crossing sequence of $\gamma$ does not have two identical ideal edges next to each other.
\end{remark}

\begin{lemma}\label{lem: representative of ideal arc} 
Let $W$ and ${\mathcal{T}}_2$ be as above. Let $\alpha$ be an oriented ideal edge of $\mathcal{T}_2$ and let $\gamma$ be an arbitrary path that is homotopic to $\alpha$. Then
$|\alpha| \leq |\gamma|.$
\end{lemma}
\begin{proof} Let $\Delta$ be the ideal triangle adjacent to $\alpha$. Let us perturb the endpoints of $\gamma$ and $\alpha$ slightly, obtaining two paths $\gamma'$ and $\alpha'$, such that 
\begin{itemize}
\item $|\alpha|=|\alpha'|$ and $|\gamma|=|\gamma'|$,
\item $\alpha'$ is contained inside $\Delta$, 
\item $\gamma'$ is homotopic to $\alpha'$.
\end{itemize}
Let $\tilde{S}$ be the universal cover of $\hat{S}$.  We consider the further decomposition $\mathcal{T}'_2$ as above.
Let us lift  $\mathcal{T}'_2$ to a decomposition $\tilde{\mathcal{T}}_2'$ of $\tilde{S}$, and lift $\alpha'$ and $\gamma'$ to one of their representatives, denoted by $\tilde{\alpha}'$ and $\tilde{\gamma}'$ respectively.

Let $(\alpha_1, \alpha_2,\cdots, \alpha_n)$ be the crossing sequence of $\tilde{\gamma}'$ with $\tilde{\mathcal{T}}_2'$. According to Remark \ref{rmk: shortest representative}, we may assume that  $\alpha_i\neq \alpha_{i+1}$ for any $i$.  If the crossing sequence of $\tilde{\gamma}'$ is not empty, then $\tilde{\gamma}'$ will cross a sequence of distinct triangles and bigons in $\tilde{\mathcal{T}}_2'$ (Figure \ref{fig: crossing sequence}). In particular, it will never return back to the ideal triangle that it starts with, which contradicts with the assumption that $\alpha'$ and $\gamma'$ are isotopic. Thus, the crossing sequence of $\tilde{\gamma}'$ must be empty, i.e., $\tilde{\gamma}'$ stays within the same ideal triangle as $\alpha'$ resides. By Lemma \ref{lem: homotoping a path}, we have 
\[
|\alpha|=|\alpha'|=|\tilde{\alpha}'|\leq |\tilde{\gamma}'|=|\gamma'|=|\gamma|.
\]
This proves our lemma.
\end{proof}

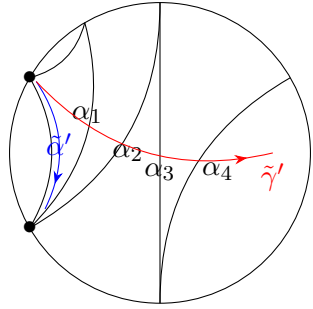
\begin{figure}[H]
    \centering
    \begin{tikzpicture}
        \draw (0,0) circle [radius=2];
        \node at (-150:2) [] {$\bullet$};
        \node at (150:2) [] {$\bullet$};
        \draw (-150:2) to [bend right]  (150:2);
        \draw [blue, decoration={markings, mark=at position 0.8 with {\arrow{Stealth}}}, postaction=decorate] (150:1.9) to [bend left] node [] {$\tilde{\alpha}'$} (-154:1.7);
        \draw (-150:2) to [bend right] node [above] {$\alpha_1$} (120:2);
        \draw (150:2) to [bend right]  (120:2);
        \draw (-150:2) to [bend right] node [below] {$\alpha_2$} (90:2);
        \draw (90:2) -- node[below] {$\alpha_3$} (-90:2);
        \draw (-90:2) to [bend left] node[right] {$\alpha_4$} (30:2);
        \draw [red,decoration={markings, mark=at position 0.9 with {\arrow{Stealth}}}, postaction=decorate] (150:1.9) to [bend right] (1.5,0) node [below] {$\tilde{\gamma}'$};
    \end{tikzpicture}
    \caption{A lift of $\gamma'$ with a non-empty crossing sequence to the universal cover $\tilde{S}$.}
    \label{fig: crossing sequence}
\end{figure}

\begin{remark} Lemma \ref{lem: representative of ideal arc} has been proved by Kuperburg \cite[Lemma 6.5]{K96} using a discrete version of the Gauss-Bonnet theorem. The advantage of our strategy is that it could be used to prove the following key lemma dealing with the tripod.
\end{remark}

Let $W$ and $\mathcal{T}_2$ be as above. Let $\Delta$ be an ideal triangle of $\mathcal{T}_2$ with vertices $a,b,c$. Let $Y$ be an inward tripod going from $a,b,c$ to the center of $\Delta$, as depicted in Figure \ref{figure:triA}(8). 

\begin{lemma}\label{lem: minimal intersection for tripod}  Let $V$ be a reduced $\mathcal{X}$ web that is homotopic to $Y$. The intersection number $i(W,V)$ attains the minimum when $V$ is contained in the triangle $\Delta$.
\end{lemma}
\begin{proof} Let us perturb the three vertices of $V$ such that they are all contained inside the triangle $\Delta$. The web $V$ is homotopic equivalent to the web $Y$ as in the picture. 

\begin{figure}[H]
\centering 
\begin{tikzpicture}
\draw (90:2)--(210:2)--(330:2)--(90:2);
\node at (90:1.6) [] {$\bullet$};
\node at (210:1.6) [] {$\bullet$};
\node at (330:1.6) [] {$\bullet$};
\node[blue] at (30:0.3) {$Y$};
\node[red] at (2.8,1.3) {$V$};
 \draw [red,decoration={markings, mark=at position 0.9 with {\arrow{Stealth}}}, postaction=decorate] (90:1.6) to [bend right]  node [above] {${\gamma}_1$} (3,1.2);
  \draw [red,decoration={markings, mark=at position 0.9 with {\arrow{Stealth}}}, postaction=decorate] (210:1.6) to [bend right]  node [above] {${\gamma}_2$} (3,1.2);
   \draw [red,decoration={markings, mark=at position 0.9 with {\arrow{Stealth}}}, postaction=decorate] (330:1.6) to [bend right]  node [below] {${\gamma}_3$} (3,1.2);
 \draw [blue,decoration={markings, mark=at position 0.9 with {\arrow{Stealth}}}, postaction=decorate] (210:1.6) to  (0,0);
  \draw [blue,decoration={markings, mark=at position 0.9 with {\arrow{Stealth}}}, postaction=decorate] (90:1.6) to  (0,0);
   \draw [blue,decoration={markings, mark=at position 0.9 with {\arrow{Stealth}}}, postaction=decorate] (330:1.6) to  (0,0);
\end{tikzpicture}
\caption{A web $V$ that is homotopic to an inward tripod $Y$.}
\end{figure}
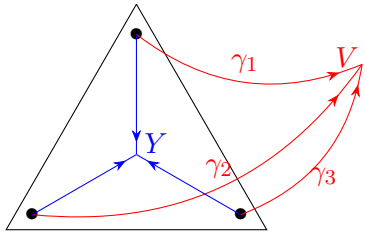

Without loss of generality, let us assume that $\hat{S}$ is a disk. Otherwise, we may lift everything to the universal cover of $\hat{S}$ as in the proof of Lemma \ref{lem: representative of ideal arc}.  Denote the three legs of $V$ by $\gamma_1$, $\gamma_2$, $\gamma_3$ respectively. By Remark \ref{rmk: shortest representative}, we assume that their crossing sequences with the decomposition $\tilde{\mathcal{T}}_2'$ do not have two identical arcs next to each other. Since $\gamma_1$, $\gamma_2$, $\gamma_3$ connects the triangle $\Delta$ to the same vertex, their crossing sequences must be the same. In other words, they cross the same sequences of triangles and bigons. The cross sequence separates $\gamma_i$ into several paths.  As illustrated by Figure \ref{cY}, we  replace the last part of each $\gamma_i$ by a straight path along the ideal arc, and move them slightly back to the other side of the arc. The newly obtained tripod $V'$ is homotopic to $V$. By Lemma \ref{lem: homotoping a tripod}, we can find $V'$ such that 
\[
i(W,V')\leq i(W,V).
\]
Let us repeat the same procedure. Eventually, we obtain a tripod inside the triangle $\Delta$, which concludes the proof of the Lemma. 
\end{proof}
\begin{figure}[H]
    \centering
    \begin{tikzpicture}
    \draw (-0.5,0) --  (-0.5,2);
        \draw [red, decoration={markings, mark=at position 0.4 with {\arrow{Stealth}}}, postaction=decorate] (-2,2) -- (1,1);
        \draw [red, decoration={markings, mark=at position 0.4 with {\arrow{Stealth}}}, postaction=decorate] (-2,1) --  (1,1);
        \draw [red, decoration={markings, mark=at position 0.4 with {\arrow{Stealth}}}, postaction=decorate] (-2,0) --  (1,1);
        \node[red] at (0.5, 1.5) {$V$};
        \end{tikzpicture}\hspace{2cm}
    \begin{tikzpicture}
     \draw (-0.5,0) --  (-0.5,2);
        \draw [red, decoration={markings, mark=at position 0.7 with {\arrow{Stealth}}}, postaction=decorate] (-2,2) -- (-0.6,1.5) -- (-0.6,1.2);
        \draw [red, decoration={markings, mark=at position 0.7 with {\arrow{Stealth}}}, postaction=decorate] (-2,1) -- (-0.7,1)-- (-0.6,1.2);
        \draw [red, decoration={markings, mark=at position 0.8 with {\arrow{Stealth}}}, postaction=decorate] (-2,0) --  (-0.6,0.5) --(-0.6,1.2);
        \node[red] at (-0.7, 2) {$V'$};
        \end{tikzpicture}
        
    \caption{Local modification of $V$.}
    \label{cY}
\end{figure}
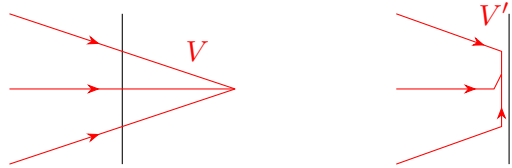

We are now ready to prove the following theorem, which is the first part of Theorem \ref{Int.hive}.

\begin{theorem}
\label{theorem:int}
For any ideal triangulation $\mathcal{T}$ of the marked surface $\hat{S}$, the map $i_\mathcal{T}$ is a bijection from $\mathscr{W}_{\hat{S}}^{\mathcal{A}}$ to ${\bf Hive}(\mathcal{T})$. 
\end{theorem}
\begin{proof}
 Let $W\in \mathscr{W}_{\hat{S}}^\mathcal{A}$ be in a good position with respect to $\mathcal{T}_2$. By Lemma \ref{lem: shorter representative} and Lemma \ref{lem: minimal intersection for tripod}, we may restrict the ideal edges and the tripods within their corresponding ideal triangles. By Proposition \ref{hive=red=tri}, we see that $i_\mathcal{T}(W)$ is a hive in ${\bf Hive}(\mathcal{T})$.

It remains to show that for any hive $\{a_\bullet\}$ in ${\bf Hive}(\mathcal{T})$, there is a unique reduced web $W$ such that $i_\mathcal{T}(W)= \{a_\bullet\}$. By restricting the data $\{a_\bullet\}$ to each ideal triangle in $\mathcal{T}$ and applying Proposition \ref{hive=red=tri}, we obtain a reduced web $W_t$ for each ideal triangle $t$ of $\mathcal{T}$. Meanwhile, by a direct calculation, if the hive $\{a_\bullet\}$ on one side of $t$ are $a_1$ and $a_2$, then the number of the oriented arcs of $W_t$ intersecting with the same side of $t$ are $2a_1-a_2$ and $2a_2-a_1$ respectively.  
\begin{figure}[H]
\begin{tikzpicture}[scale=1.3]
\draw[thick] (0,0)--(60:3)--(3,0)--(0,0);
\node at (1,0) {$\bullet$};
\node at (2,0) {$\bullet$};
\node at (60:1) {$\bullet$};
\node at (60:2)
{$\bullet$};
\begin{scope}[shift={(3,0)}] 
\node at (120:1) {$\bullet$};
\node at (120:2) {$\bullet$};
\end{scope}
\begin{scope}[shift={(2,0)}] 
\node at (120:1) {$\bullet$};
\end{scope}
\node at (1,1.4) {};
\node at (2,1.4) {};
\node at (0.5,.6) {};
\node at (1.5,.6) {};
\node at (2.5,.6) {};
\node at (1,-.3) {$a_1$};
\node at (2,-.3) {$a_2$};
\begin{scope}[shift={(5,0)}]
\draw[thick] (0,0)--(60:3)--(3,0)--(0,0);
\draw[thick, -latex] (0.7,.6)--(0.7,0);
\draw[thick, latex-] (2.3,.6)--(2.3,0);
\node at (0.7,-.3) {$2a_1-a_2$};
\node at (2.3,-.3) {$2a_2-a_1$};
\end{scope}
\end{tikzpicture}
\caption{Oriented arcs that intersect a side of a triangle.}
\end{figure}
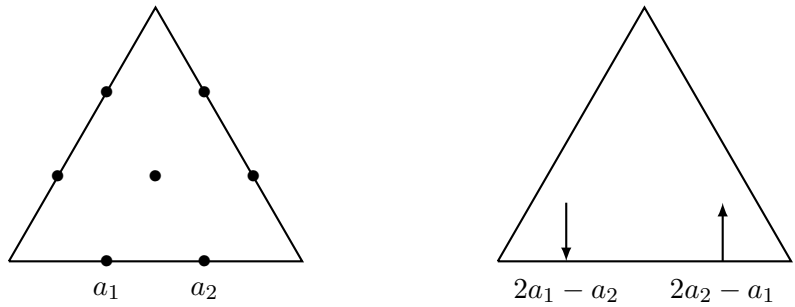

As a result, we see that for any neighbored ideal triangles $t$ and $t'$, the webs $W_t$ and $W_{t'}$ can be connected by a minimal ladder in the bigon separating them.  In this way, we obtain a web $W$ in a good position with respect to $\mathcal{T}_2$. Following the procedure in Section 6.3 of \cite{DS20a}, one may reorder the position of the corner arcs in each ideal triangle, and resolve the ``I'' bars in the bigons, obtaining a reduced web $W'.$ Meanwhile, by Lemma 57 of {\it loc.cit.}, the obtained reduced web $W'$ is unique upto homotopy equivalence on $\hat{S}\times [0,1]$. Hence, the map $i_\mathcal{T}$ is a bijection. 
\end{proof}

\subsection{Compatibility}

In this subsection, we prove the second part of Theorem \ref{Int.hive}, that is, the images of intersection maps for different ideal triangulations are related by octahedron relations.

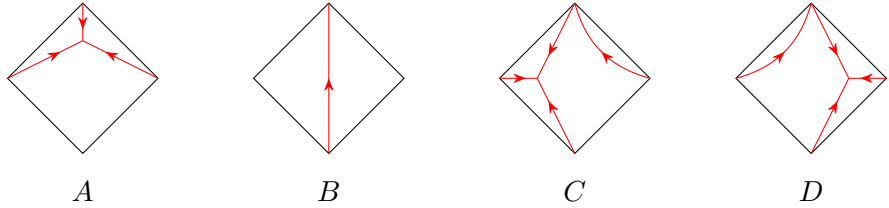
\begin{figure}[H]
    \centering
    \begin{tikzpicture}[baseline=0]
    \node at (0,-1.5) [] {$A$};
        \draw (0,-1) -- (1,0) -- (0,1) -- (-1,0) -- cycle;
    \draw [red, decoration={markings, mark=at position 0.7 with {\arrow{Stealth}}}, postaction=decorate] (0,1) -- (0,0.5);
    \draw [red, decoration={markings, mark=at position 0.7 with {\arrow{Stealth}}}, postaction=decorate] (1,0) -- (0,0.5);
    \draw [red, decoration={markings, mark=at position 0.7 with {\arrow{Stealth}}}, postaction=decorate] (-1,0) -- (0,0.5);
    \end{tikzpicture} 
    \hspace{1cm}
    \begin{tikzpicture}[baseline=0]
    \node at (0,-1.5) [] {$B$};
        \draw (0,-1) -- (1,0) -- (0,1) -- (-1,0) -- cycle;
    \draw [red, decoration={markings, mark=at position 0.5 with {\arrow{Stealth}}}, postaction=decorate] (0,-1) to  (0,1);
    \end{tikzpicture} \hspace{1cm}
    \begin{tikzpicture}[baseline=0]
    \node at (0,-1.5) [] {$C$};
        \draw (0,-1) -- (1,0) -- (0,1) -- (-1,0) -- cycle;
    \draw [red, decoration={markings, mark=at position 0.7 with {\arrow{Stealth}}}, postaction=decorate] (-1,0) -- (-0.5,0);
    \draw [red, decoration={markings, mark=at position 0.7 with {\arrow{Stealth}}}, postaction=decorate] (0,1) -- (-0.5,0);
    \draw [red, decoration={markings, mark=at position 0.7 with {\arrow{Stealth}}}, postaction=decorate] (0,-1) -- (-0.5,0);
    \draw [red, decoration={markings, mark=at position 0.5 with {\arrow{Stealth}}}, postaction=decorate] (1,0) to [bend left] (0,1);
    \end{tikzpicture}\hspace{1cm}
    \begin{tikzpicture}[baseline=0]
    \node at (0,-1.5) [] {$D$};
        \draw (0,-1) -- (1,0) -- (0,1) -- (-1,0) -- cycle;
    \draw [red, decoration={markings, mark=at position 0.7 with {\arrow{Stealth}}}, postaction=decorate] (0,1) -- (0.5,0);
    \draw [red, decoration={markings, mark=at position 0.7 with {\arrow{Stealth}}}, postaction=decorate] (1,0) -- (0.5,0);
    \draw [red, decoration={markings, mark=at position 0.7 with {\arrow{Stealth}}}, postaction=decorate] (0,-1) -- (0.5,0);
    \draw [red, decoration={markings, mark=at position 0.5 with {\arrow{Stealth}}}, postaction=decorate] (-1,0) to [bend right] (0,1);
    \end{tikzpicture}
    \caption{Exchanging branches}
    \label{fig:exchange.webs}
\end{figure}

\begin{lemma}\label{prop:flip invariant} 
Let $W$ be a reduced $\mathcal{A}$ web on $\hat{S}$ and let $A$, $B$, $C$, and $D$ be reduced $\mathcal{X}$ webs in Figure \ref{fig:exchange.webs}. We have
\[
\mathbb{I}([W],[A])+\mathbb{I}([W],[B])=\max\{\mathbb{I}([W],[C]),\mathbb{I}([W],[D])\}.
\]

\end{lemma}

\begin{proof} By the local properties of the intersection maps for tripods and ideal edges (Lemma \ref{lem: representative of ideal arc} and Lemma \ref{lem: minimal intersection for tripod}), it suffices to consider  the case when $\hat{S}$ is a disk with four marked points. Let $W^*$ be the dual surfacoid of $W$ as in Figure \ref{figure:dualgr}. We compute the desired intersection numbers by making the four $\mathcal{X}$-webs travel along $W^*$. 

With respect to the top-bottom triangulations of $\hat{S}$, $W^*$ is separated into three subgraphs: 
\begin{itemize}
\item a net $T(W^*)$ in the top triangle, 
\item a union of necklaces $H(W^*)$ in the middle bigon, 
\item a net $B(W^*)$ in the bottom triangle.
\end{itemize}
The pattern of each region is illustrated by Figure \ref{fig: two types of regions}.
By Lemma \ref{lem: minimal intersection for tripod}, there exists a representative $A$ of $[A]$ attaining the value $\mathbb{I}([W],[A])$, with the trivalent vertex lying inside the mesh of $T(W^*)$. In particular, by Lemma \ref{lem: representative of ideal arc}, we can fix the trajectory of the tripod $A$ along $W^*$ so that its two lower legs travel along the boundary of $T(W^*)$, meeting at the trivalent vertex located at the top corner $x$ of the mesh of $T(W^*)$. 
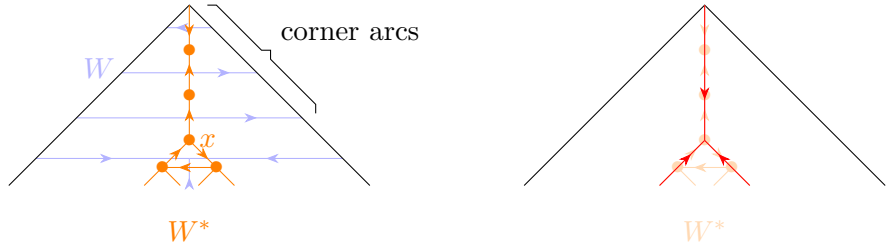
\begin{figure}[H]
    \centering
    \begin{tikzpicture}[scale=1.2]
    \node [lightblue] at (-1,1.3) [] {$W$};
    \node [orange] at (0,-0.5) [] {$W^*$};
    \draw (0.2,2) -- (0.3,2) -- (0.8,1.5) -- (0.9,1.5) node [above right] {corner arcs} -- (0.9,1.4) -- (1.4,0.9) -- (1.4,0.8);
     \draw [lightblue, decoration={markings, mark=at position 0.8 with {\arrow{Stealth}}}, postaction=decorate] (0.25,1.75) -- (-0.25,1.75);
     \draw [lightblue, decoration={markings, mark=at position 0.8 with {\arrow{Stealth}}}, postaction=decorate] (-0.75,1.25) -- (0.75,1.25);
     \draw [lightblue, decoration={markings, mark=at position 0.8 with {\arrow{Stealth}}}, postaction=decorate] (-1.25,0.75) -- (1.25,0.75);
     \draw [lightblue, decoration={markings, mark=at position 0.5 with {\arrow{Stealth}}}, postaction=decorate] (-1.7,0.3) -- (-0,0.3);
     \draw [lightblue, decoration={markings, mark=at position 0.5 with {\arrow{Stealth}}}, postaction=decorate] (1.7,0.3) -- (-0,0.3);
     \draw [lightblue, decoration={markings, mark=at position 0.3 with {\arrow{Stealth}}}, postaction=decorate] (0,0) -- (-0,0.3);
        \draw (-2,0) -- (0,2) -- (2,0);
        \draw [orange, decoration={markings, mark=at position 0.7 with {\arrow{Stealth}}}, postaction=decorate] (0,2) -- (0,1.5);
        \node [orange] at (0,1.5) [] {$\bullet$};
        \draw [orange, decoration={markings, mark=at position 0.7 with {\arrow{Stealth}}}, postaction=decorate] (0,1) -- (0,1.5);
        \node [orange] at (0,1) [] {$\bullet$};
        \draw [orange, decoration={markings, mark=at position 0.7 with {\arrow{Stealth}}}, postaction=decorate] (0,0.5) -- (0,1);
        \node [orange] at (0,0.5) [] {$\bullet$};
        \draw [orange, decoration={markings, mark=at position 0.7 with {\arrow{Stealth}}}, postaction=decorate] (-0.3,0.2) -- (0,0.5);
        \node [orange] at (-0.3,0.2) [] {$\bullet$};
        \draw [orange, decoration={markings, mark=at position 0.7 with {\arrow{Stealth}}}, postaction=decorate] (0,0.5) -- (0.3,0.2);
        \node [orange] at (0.3,0.2) [] {$\bullet$};
        \draw [orange, decoration={markings, mark=at position 0.7 with {\arrow{Stealth}}}, postaction=decorate] (0.3,0.2) -- (-0.3,0.2);
        \draw [orange] (-0.5,0) -- (-0.3,0.2);
        \draw [orange] (-0.3,0.2) -- (-0.1,0);
        \draw [orange] (0.1,0) -- (0.3,0.2);
        \draw [orange] (0.3,0.2) -- (0.5,0);
        \node [orange] at (0,0.5) [right] {$x$};
    \end{tikzpicture} \hspace{1cm}
    \begin{tikzpicture}[scale=1.2]
    \node [lightorange] at (0,-0.5) [] {$W^*$};
        \draw (-2,0) -- (0,2) -- (2,0);
        \draw [lightorange, decoration={markings, mark=at position 0.7 with {\arrow{Stealth}}}, postaction=decorate] (0,2) -- (0,1.5);
        \node [lightorange] at (0,1.5) [] {$\bullet$};
        \draw [lightorange, decoration={markings, mark=at position 0.7 with {\arrow{Stealth}}}, postaction=decorate] (0,1) -- (0,1.5);
        \node [lightorange] at (0,1) [] {$\bullet$};
        \draw [lightorange, decoration={markings, mark=at position 0.7 with {\arrow{Stealth}}}, postaction=decorate] (0,0.5) -- (0,1);
        \node [lightorange] at (0,0.5) [] {$\bullet$};
        \draw [lightorange, decoration={markings, mark=at position 0.7 with {\arrow{Stealth}}}, postaction=decorate] (-0.3,0.2) -- (0,0.5);
        \node [lightorange] at (-0.3,0.2) [] {$\bullet$};
        \draw [lightorange, decoration={markings, mark=at position 0.7 with {\arrow{Stealth}}}, postaction=decorate] (0,0.5) -- (0.3,0.2);
        \node [lightorange] at (0.3,0.2) [] {$\bullet$};
        \draw [lightorange, decoration={markings, mark=at position 0.7 with {\arrow{Stealth}}}, postaction=decorate] (0.3,0.2) -- (-0.3,0.2);
        \draw [lightorange] (-0.5,0) -- (-0.3,0.2);
        \draw [lightorange] (-0.3,0.2) -- (-0.1,0);
        \draw [lightorange] (0.1,0) -- (0.3,0.2);
        \draw [lightorange] (0.3,0.2) -- (0.5,0);
        \draw [red, decoration={markings, mark=at position 0.7 with {\arrow{Stealth}}}, postaction=decorate] (0,2) -- (0,0.5);
        \draw [red, decoration={markings, mark=at position 0.7 with {\arrow{Stealth}}}, postaction=decorate] (-0.5,0) -- (0,0.5);
        \draw [red, decoration={markings, mark=at position 0.7 with {\arrow{Stealth}}}, postaction=decorate] (0.5,0) -- (0,0.5);
    \end{tikzpicture}
    \caption{Left: local picture of $W^*$ when $W$ is in good position with respect to the top-bottom triangulation. Right: trajectory of the tripod $A$ along $W^*$.}
    \label{fig:local picture of W^*}
\end{figure}

With respect to the left-right triangulation of $\hat{S}$, $W^*$ is seperated into three subgraphs:
\begin{itemize}
\item a net $L(W^*)$ in the left triangle;
\item a union of necklaces $V(W^*)$ in the middle triangle;
\item a net $R(W^*)$ in the right triangle.
\end{itemize}
The vertex $x$ of $W^*$ now belongs to the bigon $V(W^*)$ in the middle. By Lemma \ref{lem: representative of ideal arc},  $i(W,B)$ attains its minimum when $B$ travels along a straight path that is a boundary of the bigon region. Note that this straight path must pass through the vertex $x$. By exchanging a leg of $A$ and part of $B$ at the vertex $x$, we get a representative of $[C]$ and a representative of $[D]$ as in Figure \ref{fig:exchange}. This shows that 
\begin{equation}
\label{exchange,acd}
\mathbb{I}([W],[A])+\mathbb{I}([W],[B])\geq \max \{\mathbb{I}([W],[C]),\mathbb{I}([W],[D])\}.
\end{equation}

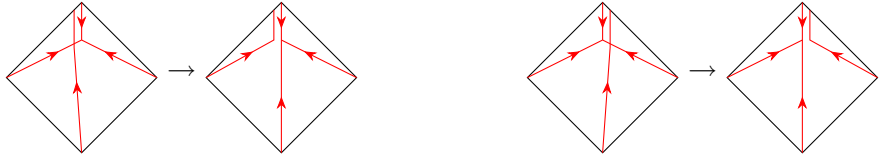
\begin{figure}[t]
    \centering
    \begin{tikzpicture}[baseline=0]
        \draw (0,-1) -- (1,0) -- (0,1) -- (-1,0) -- cycle;
    \draw [red, decoration={markings, mark=at position 0.7 with {\arrow{Stealth}}}, postaction=decorate] (0,1) -- (0,0.5);
    \draw [red, decoration={markings, mark=at position 0.7 with {\arrow{Stealth}}}, postaction=decorate] (1,0) -- (0,0.5);
    \draw [red, decoration={markings, mark=at position 0.7 with {\arrow{Stealth}}}, postaction=decorate] (-1,0) -- (0,0.5);
    \draw [red, decoration={markings, mark=at position 0.5 with {\arrow{Stealth}}}, postaction=decorate] (0,-1) to  (-0.1,0.4) -- (-0.1,0.9);
    \end{tikzpicture} $\rightarrow$
    \begin{tikzpicture}[baseline=0]
        \draw (0,-1) -- (1,0) -- (0,1) -- (-1,0) -- cycle;
    \draw [red, decoration={markings, mark=at position 0.7 with {\arrow{Stealth}}}, postaction=decorate] (0,1) -- (0,0.5);
    \draw [red, decoration={markings, mark=at position 0.7 with {\arrow{Stealth}}}, postaction=decorate] (1,0) -- (0,0.5);
    \draw [red, decoration={markings, mark=at position 0.5 with {\arrow{Stealth}}}, postaction=decorate] (-1,0) -- (-0.1,0.5) -- (-0.1,0.9);
    \draw [red, decoration={markings, mark=at position 0.5 with {\arrow{Stealth}}}, postaction=decorate] (0,-1) to  (0,0.5);
    \end{tikzpicture} \hspace{2cm}
    \begin{tikzpicture}[baseline=0]
        \draw (0,-1) -- (1,0) -- (0,1) -- (-1,0) -- cycle;
    \draw [red, decoration={markings, mark=at position 0.7 with {\arrow{Stealth}}}, postaction=decorate] (0,1) -- (0,0.5);
    \draw [red, decoration={markings, mark=at position 0.7 with {\arrow{Stealth}}}, postaction=decorate] (1,0) -- (0,0.5);
    \draw [red, decoration={markings, mark=at position 0.7 with {\arrow{Stealth}}}, postaction=decorate] (-1,0) -- (0,0.5);
    \draw [red, decoration={markings, mark=at position 0.5 with {\arrow{Stealth}}}, postaction=decorate] (0,-1) to  (0.1,0.4) -- (0.1,0.9);
    \end{tikzpicture} $\rightarrow$
    \begin{tikzpicture}[baseline=0]
        \draw (0,-1) -- (1,0) -- (0,1) -- (-1,0) -- cycle;
    \draw [red, decoration={markings, mark=at position 0.7 with {\arrow{Stealth}}}, postaction=decorate] (0,1) -- (0,0.5);
    \draw [red, decoration={markings, mark=at position 0.5 with {\arrow{Stealth}}}, postaction=decorate] (1,0) -- (0.1,0.5) -- (0.1,0.9);
    \draw [red, decoration={markings, mark=at position 0.7 with {\arrow{Stealth}}}, postaction=decorate] (-1,0) -- (0,0.5);
    \draw [red, decoration={markings, mark=at position 0.5 with {\arrow{Stealth}}}, postaction=decorate] (0,-1) to (0,0.5);
    \end{tikzpicture}
    \caption{Exchanging legs}
    \label{fig:exchange}
\end{figure}

 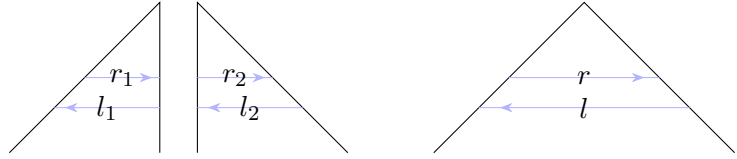
\begin{figure}[H]
\begin{tikzpicture}
\draw (0,-2)--(0,0)--(-2,-2);
\draw (0.5,-2)--(0.5,0)--(2.5,-2);
\draw [lightblue, decoration={markings, mark=at position 0.9 with {\arrow{Stealth}}}, postaction=decorate] (-1,-1)--(0,-1);
\draw [lightblue, decoration={markings, mark=at position 0.9 with {\arrow{Stealth}}}, postaction=decorate] (0,-1.4)--(-1.4,-1.4);
\node at (-.5,-1) {$r_1$};
\node at (-.7,-1.4) {$l_1$};
\draw [lightblue, decoration={markings, mark=at position 0.9 with {\arrow{Stealth}}}, postaction=decorate] (0.5,-1)--(1.5,-1);
\draw [lightblue, decoration={markings, mark=at position 0.9 with {\arrow{Stealth}}}, postaction=decorate] (1.9,-1.4)--(0.5,-1.4);
\node at (1,-1) {$r_2$};
\node at (1.2,-1.4) {$l_2$};
\end{tikzpicture}
\hskip 1cm
\begin{tikzpicture}
\draw (-2,-2)--(0,0)--(2,-2);
\draw [lightblue, decoration={markings, mark=at position 0.9 with {\arrow{Stealth}}}, postaction=decorate] (-1,-1)--(1,-1);
\draw [lightblue, decoration={markings, mark=at position 0.9 with {\arrow{Stealth}}}, postaction=decorate] (1.4,-1.4)--(-1.4,-1.4);
\node at (0,-1) {$r$};
\node at (0,-1.4) {$l$};
\end{tikzpicture}
\caption{Corner arcs with respect to different triangulations}
\label{gluing.tri}
\end{figure}
As illustrated by Figure \ref{gluing.tri}, let $r_1$ and $l_1$ be the numbers of oriented corner arcs of $W$ restricted to the left triangle.
Let $r_2$ and $l_2$ be the numbers of corner arcs restricted to  the right triangle.
Let $r$ and $l$ be the numbers of corner arcs restricted to the top triangle. Following the procedure of gluing ideal triangles, we have 
\[
r=\min\{r_1, r_2\}; \qquad l=\min\{l_1,l_2\}
\]
To show the equality of \eqref{exchange,acd}, we split the proof into two cases. 

\smallskip 

{\bf Case 1}: $l_1\leq l_2$. Let $y$ be the vertex on the top corner of the mesh of $L(W^*)$. 
Recall the vertex $x$ in Figure \ref{fig:local picture of W^*}.  After exchanging the legs as in the right picture of Figure \ref{fig:exchange}, the resulted representative $D$ has two $\mathcal{X}$-web components: a straight path that is parallel to the upper right side of the quadrilateral, and a tripod on the left with trivalent vertex at $x$. Note that one of the branches of the tripod is still traveling along the left boundary of the bigon region until it reaches $x$. By Lemma \ref{lem: representative of ideal arc},  the intersection number between $W$ and the straight path component already attains its minimum.

If $r_1\leq r_2$, then $x$ coincides with $y$. Hence $x$ resides on the mesh of $L(W^*)$.
  By Lemma \ref{lem: minimal intersection for tripod}, the intersection number between $W$ and the left tripod reaches its minimum as well, and hence we can conclude that 
 \[\mathbb{I}([W],[A])+\mathbb{I}([W],[B])=\mathbb{I}([W],[D]).
 \]

If $r_1 >r_2$, then there are $r_1-r_2$ many upward edges from $y$ to $x$. By construction, two of the branches of the tripod travel along the straight path from $y$ to $x$ (left picture in Figure \ref{fig:between x and y}). If we move the trivalent vertex of the tripod from $x$ to $y$ along the straight path, shrinking the two lower branches while extending the upper branch, the intersection number between the tripod and $W$ remains unchanged. In the end, we arrive at a configuration as in the right picture of Figure \ref{fig:between x and y}. Since the trivalent vertex of the tripod is inside the triangle region now, we can again apply Lemma \ref{lem: minimal intersection for tripod} and conclude that 
\[\mathbb{I}([W],[A])+\mathbb{I}([W],[B])=\mathbb{I}([W],[D]).\]

\begin{figure}[H]
    \centering
    \begin{tikzpicture}[scale=0.5]
        \foreach \i in {0,...,3}
        {
        \draw [lightorange, decoration={markings, mark=at position 0.7 with {\arrow{Stealth}}}, postaction=decorate] (\i,\i) -- (\i+1,\i+1);
        \draw [lightorange, decoration={markings, mark=at position 0.7 with {\arrow{Stealth}}}, postaction=decorate] (\i+1,\i+1) -- (\i+2,\i);
        \draw [lightorange, decoration={markings, mark=at position 0.7 with {\arrow{Stealth}}}, postaction=decorate] (\i+2,\i) -- (\i,\i);
        }
        \node [orange] at (4,4) [] {$\bullet$};
        \node [orange] at (4,4) [right] {$x$};
        \node [orange] at (0,0) [] {$\bullet$};
        \node [orange] at (0,0) [left] {$y$};
        \draw [red, decoration={markings, mark=at position 0.7 with {\arrow{Stealth}}}, postaction=decorate] (4,5.5) -- (4,4);
        \draw [red, decoration={markings, mark=at position 0.7 with {\arrow{Stealth}}}, postaction=decorate] (-1.1,-1) -- (3.9,4);
        \draw [red, decoration={markings, mark=at position 0.7 with {\arrow{Stealth}}}, postaction=decorate] (1.1,-1) -- (0.1,0) -- (4.1,4);
    \end{tikzpicture} \hspace{2cm}
    \begin{tikzpicture}[scale=0.5]
        \foreach \i in {0,...,3}
        {
        \draw [lightorange, decoration={markings, mark=at position 0.7 with {\arrow{Stealth}}}, postaction=decorate] (\i,\i) -- (\i+1,\i+1);
        \draw [lightorange, decoration={markings, mark=at position 0.7 with {\arrow{Stealth}}}, postaction=decorate] (\i+1,\i+1) -- (\i+2,\i);
        \draw [lightorange, decoration={markings, mark=at position 0.7 with {\arrow{Stealth}}}, postaction=decorate] (\i+2,\i) -- (\i,\i);
        }
        \node [orange] at (4,4) [] {$\bullet$};
        \node [orange] at (4,4) [right] {$x$};
        \node [orange] at (0,0) [] {$\bullet$};
        \node [orange] at (0,0) [left] {$y$};
        \draw [red, decoration={markings, mark=at position 0.7 with {\arrow{Stealth}}}, postaction=decorate] (4,5.5) -- (4,4) -- (0,0);
        \draw [red, decoration={markings, mark=at position 0.7 with {\arrow{Stealth}}}, postaction=decorate] (-1,-1) -- (0,0);
        \draw [red, decoration={markings, mark=at position 0.7 with {\arrow{Stealth}}}, postaction=decorate] (1,-1) -- (0,0);
    \end{tikzpicture}
    \caption{Local pictures between the vertices $x$ and $y$ in $W^*$.}
    \label{fig:between x and y}
\end{figure}
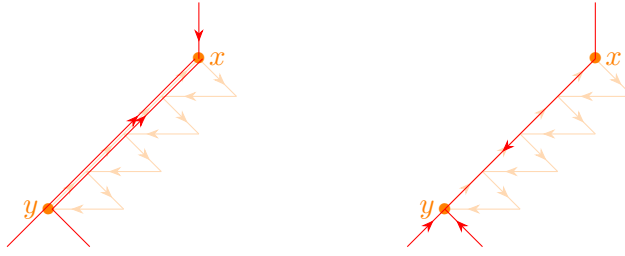

{\bf Case 2}: $l_1\geq l_2$. This is symmetric to the previous case and can be handled analogously, 
 yielding $\mathbb{I}([W],[A])+\mathbb{I}([W],[B])=\mathbb{I}([W],[C])$ instead. 
\end{proof}

 Let $\mathcal{T}$ and $\mathcal{T}'$ be two ideal triangulations related by a flip of a diagonal. Following Figure \eqref{figure:flip}, the four new webs associated with $\mathcal{T}'$ can be obtained in four steps of changes. In each step, we only use the relation in Lemma \ref{prop:flip invariant}. A direct comparison shows that they coincide with the four octahedron relations in \eqref{oact.rec}, which concludes the proof of the second part of Theorem \ref{Int.hive}.

\begin{figure}[htb]
\includegraphics[scale=0.5]{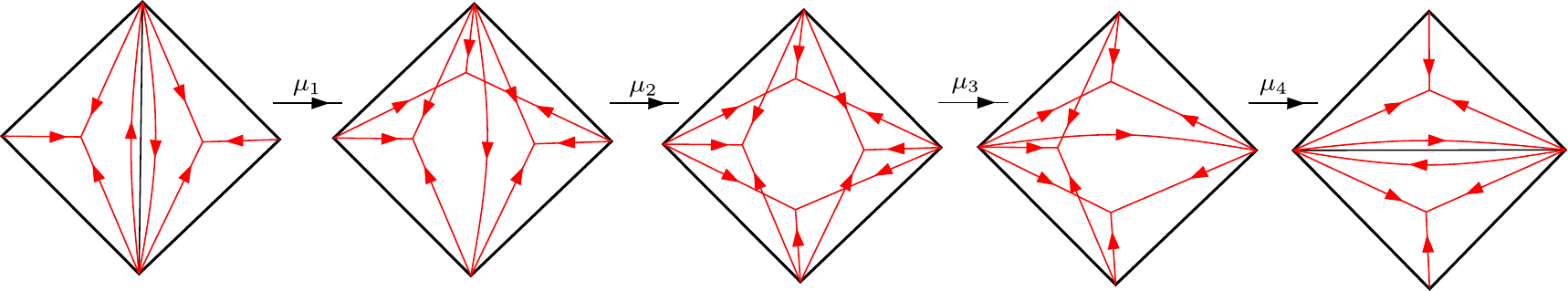}
\caption{The flip $\mu=\mu_4\circ\mu_3\circ\mu_2\circ\mu_1$.}
\label{figure:flip}
\end{figure}

\section{Fock-Goncharov  Moduli Spaces}
\label{section:wa}
In this section, we recall the Fock-Goncharov $\mathcal{A}$-moduli spaces and their cluster structures. We further consider the tropicalization of the $\mathcal{A}$-moduli spaces  restricted by the potential function $P$ following \cite{GS15}.  

\subsection{Definitions}

A \emph{flag} $F$ in a 3-dimensional vector space $E$ is a filtration of subspaces 
\[\{0\}=F^{(0)} \subset F^{(1)} \subset F^{(2)} \subset F^{(3)}=E, \qquad \mbox{where } {\dim }F^{(i)}=i.\] 
A \emph{decorated flag} is  a flag $F$ with a pair of non-zero vectors ${f}^i \in \wedge^i F^{(i)}$ for $i=1,2$. Let $\mathcal{A}$ denote the moduli space of decorated flags. The action of ${\rm SL}_3$ on $E$ induces a transitive action of $\operatorname{SL}_3$ on  $\mathcal{A}$. After fixing a decorated flag, we obtain an isomorphism $\mathcal{A}\stackrel{\sim}{=}{\rm SL}_3/{\rm U}$, where ${\rm U}\subset{\rm SL}_3$ is the subgroup of unipotent upper triangular matrices. 

Let $\hat{S}$ be a decorated surface as in Section \ref{subsec2.1}. Following \cite[Definition 2.4]{FG06}, we recall the Fock-Goncharov moduli space $\mathcal{A}_{\operatorname{SL}_3,\hat{S}}$. 

\begin{defn} Let $\rho$ be a $\operatorname{SL}_3$-local system on $S$.
A {\em decoration on $\rho$}  is a flat section  $\xi$ of the restriction of the associated decorated flag bundle $\rho \underset{\operatorname{SL}_3}{\times} \mathcal{A}$ to $m_b\cup m_p$. The moduli space  $\mathcal{A}_{\operatorname{SL}_3,\hat{S}}$ parametrizes the pairs $(\rho,\xi)$ up to the equivalence $(\rho,\xi)\sim(g\rho g^{-1},g\xi)$ for $g\in {\rm SL}_3$.
\end{defn}

\begin{example} 
\label{triangle.a.space}
Let $\hat{S}=\Delta$ be a disk with three marked points. Note that the fundamental group $\pi_1(\Delta)$ is trivial. Therefore the moduli space $\mathcal{A}_{{\rm SL}_3, \Delta}$ parametrizes a triple of decorated flags modulo the diagonal action of ${\rm SL}_3$:
\[
\mathcal{A}_{{\rm SL}_3, \Delta}= {\rm SL}_3\backslash \mathcal{A}^3.
\]
Recall the quiver $Q_\Delta$. Each vertex $v$ of $Q_\Delta$ corresponds to a triple $(i,j,k)$ of non-negative integers with $i+j+k=3$.

\begin{figure}[H]
\begin{tikzpicture}[scale=1.6]
\draw[gray!50, thick] (0,0)--(60:3)--(3,0)--(0,0);
\draw[red, thick, latex-] (60:2)++(0.15,0) -- ++ (0.65, 0);
\draw[red, thick, -latex] (60:2)++(-60:0.15) -- ++ (-60:0.65);
\draw[red, thick, -latex] (1,0)++(60:1.2) -- ++ (60:0.65);
\draw[red, thick, latex-] (60:1)++(0.2,0) -- ++ (0.4, 0);
\draw[red, thick, -latex] (60:1)++(-60:0.15) -- ++ (-60:0.65);
\draw[red, thick, -latex] (1,0)++(60:0.2) -- ++ (60:0.65);
\draw[red, thick, latex-] (60:1)++(1.4,0) -- ++ (0.4, 0);
\draw[red, thick, -latex] (60:2)++(-60:1.15) -- ++ (-60:0.65);
\draw[red, thick, -latex] (2,0)++(60:0.2) -- ++ (60:0.65);
\node (A) at (1,0) {$(0,2,1)$};
\node (B) at (2,0) {$(0,1,2)$};
\node at (0.3,0.866) {$(1,2,0)$};
\node at (0.8,1.732) {$(2,1,0)$};
\node at (2.7,0.866) {$(1,0,2)$};
\node at (2.2,1.732) {$(2,0,1)$};
\begin{scope}[shift={(2,0)}] 
\node at (120:1) {$(1,1,1)$};
\end{scope}
\node[blue] at (1.5, 2.65) {$(f_1,f_2)$};
\node[blue] at (-.4,0) {$(g_1,g_2)$};
\node[blue] at (3.4,0) {$(h_1,h_2)$};
\end{tikzpicture}
\caption{The quiver $Q_\Delta$.}
\end{figure}
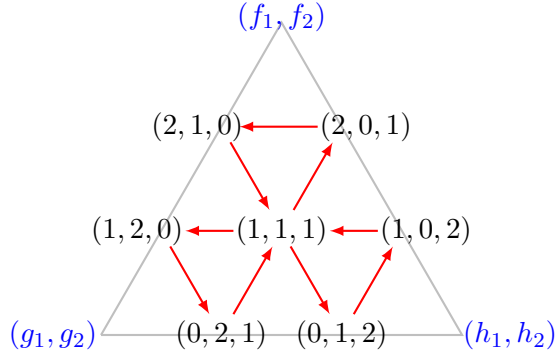
Every marked point of $\Delta$ is associated with a decorated flag. Let
\begin{align*}
(f_1,f_2),\; (g_1,g_2),\; (h_1,h_2)
\end{align*}
be their corresponding pairs of nonzero vectors. Let us fix a volume form $\Omega$ on the 3-dimensional vector space $E$.  Following \cite[Section 9]{FG06}, we set
\begin{equation}
    \label{defn:FGA}
  A_{i,j,k}:=\Omega\left(f^i \wedge g^j \wedge h^k\right).
\end{equation}
\end{example}

In general, let $\mathcal{T}$ be an ideal triangulation of a decorated surface $\hat{S}$.  Recall the quiver $Q_{\mathcal{T}}$ with the vertex set $\Theta_{\mathcal{T}}$. The restriction of $(\rho, \xi)$ to every ideal triangle in $\mathcal{T}$ gives rise to a triple of decorated flags. By applying \eqref{defn:FGA}, we obtain a collection  $\{A_v\}_{v\in \Theta_{\mathcal{T}}}$  of {\em Fock-Goncharov $\mathcal{A}$-coordinates} for $\mathcal{A}_{{\rm SL}_3,\hat{S}}$.

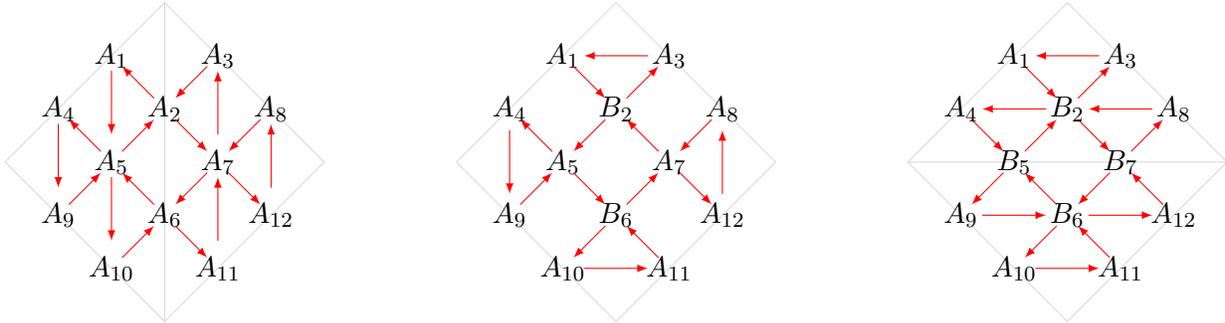
\begin{figure}[H]
\begin{tikzpicture}
\begin{scope}[rotate=-45]
\draw[gray!30] (0,0)--(-3,3)--(0,3)--(0,0)--(-3,0)--(-3,3);
\node at (-1,1) {$A_6$};
\node at (-2,2) {$A_2$};
\node at (0,1) {$A_{11}$};
\node at (0,2) {$A_{12}$};
\node at (-1,0) {$A_{10}$};
\node at (-2,0) {$A_9$};
\node at (-2,3) {$A_3$};
\node at (-1,3) {$A_8$};
\node at (-1,2) {$A_7$};
\node at (-2,1) {$A_5$};
\node at (-3,2) {$A_1$};
\node at (-3,1) {$A_4$};
 \foreach \position in {(-2,1), (-2,2), (-1,1)}
    {\draw[-latex,  red] \position++(-.2,0) -- ++(-.6,0);
    \draw[-latex,  red] \position++(0,-0.8) -- ++(0,0.6);
    \draw[-latex,  red] \position++(-1,0)++(-45:0.2) -- ++(0.6,-0.6);}
 \foreach \position in {(-1,2), (-2,2), (-1,1)}
    {\draw[latex-,  red] \position++(0,0.2) -- ++(0,0.6);
    \draw[latex-,  red] \position++(0.8,0) -- ++(-0.6,0);
    \draw[latex-,  red] \position++(0,1)++(-45:0.2) -- ++(0.6,-0.6);}   
\end{scope}

\begin{scope}[xshift=6cm]
\begin{scope}[rotate=-45]
\draw[gray!30] (-3,3)--(0,3)--(0,0)--(-3,0)--(-3,3);
\node at (-1,1) {$B_6$};
\node at (-2,2) {$B_2$};
\node at (0,1) {$A_{11}$};
\node at (0,2) {$A_{12}$};
\node at (-1,0) {$A_{10}$};
\node at (-2,0) {$A_9$};
\node at (-2,3) {$A_3$};
\node at (-1,3) {$A_8$};
\node at (-1,2) {$A_7$};
\node at (-2,1) {$A_5$};
\node at (-3,2) {$A_1$};
\node at (-3,1) {$A_4$};
 \foreach \position in {(-2,2), (-1,1)}
    {\draw[latex-,  red] \position++(-.2,0) -- ++(-.6,0);
    \draw[latex-,  red] \position++(0,-0.8) -- ++(0,0.6);
    \draw[latex-,  red] \position++(0, 0.8) -- ++(0, -0.6);
    \draw[latex-, red] \position++(0.2,0) -- ++(0.6,0);
   }

\draw[-latex,  red] (-1,0)++(0.2,0.2)--++(0.6,0.6);
\draw[latex-, red] (-3,2)++(0.2,.2)--++(0.6,.6);
\draw[-latex, red] (-2.2,1) -- ++(-.6,0);
    \draw[-latex, red] (-2,.2) -- ++(0,0.6);
    \draw[-latex, red] (-2.8,0.8) -- ++(0.6,-0.6);
    \draw[latex-,  red] (-0.8,2.8) -- ++(0.6,-0.6);
    \draw[-latex, red] (-1,2.8) -- ++(0,-0.6);
    \draw[-latex, red] (-.8,2) -- ++(0.6,0);   
\end{scope}  
\end{scope}

\begin{scope}[xshift=12cm]
\begin{scope}[rotate=-45]
\draw[gray!30] (-3,3)--(0,3)--(0,0)--(-3,0)--(-3,3);
\draw[gray!30] (-3,0)--(0,3);
    \draw[latex-,  red] (-3,2)++(0.2,0.2) -- ++(0.6,0.6);
    \draw[latex-,  red] (-3,1)++(0.2,0.2) -- ++(0.6,0.6);
    \draw[latex-,  red] (-2,2)++(0.2,0.2) -- ++(0.6,0.6);
    \draw[-latex,  red] (-2,0)++(0.2,0.2) -- ++(0.6,0.6);
    \draw[-latex,  red] (-1,0)++(0.2,0.2) -- ++(0.6,0.6);
    \draw[-latex,  red] (-1,1)++(0.2,0.2) -- ++(0.6,0.6);
    
    \draw[-latex,  red] (-3,2)++(0.2,0) -- ++(0.6,0);
    \draw[-latex,  red] (-3,1)++(0.2,0) -- ++(0.6,0);
    \draw[-latex,  red] (-2,2)++(0.2,0) -- ++(0.6,0);
    \draw[latex-,  red] (-2,0)++(0,0.2) -- ++(0,0.6);
    \draw[latex-,  red] (-1,0)++(0,0.2) -- ++(0,0.6);
    \draw[latex-,  red] (-1,1)++(0,0.2) -- ++(0,0.6);

    \draw[-latex,  red] (-2,2)++(0,0.2) -- ++(0,0.6);
    \draw[-latex,  red] (-2,1)++(0,0.2) -- ++(0,0.6);
    \draw[-latex,  red] (-1,2)++(0,0.2) -- ++(0,0.6);
    \draw[latex-,  red] (-2,0)++(0.2,1) -- ++(0.6,0);
    \draw[latex-,  red] (-1,0)++(0.2,1) -- ++(0.6,0);
    \draw[latex-,  red] (-1,1)++(0.2,1) -- ++(0.6,0);

  \node at (-1,1) {$B_6$};
\node at (-2,2) {$B_2$};
\node at (0,1) {$A_{11}$};
\node at (0,2) {$A_{12}$};
\node at (-1,0) {$A_{10}$};
\node at (-2,0) {$A_9$};
\node at (-2,3) {$A_3$};
\node at (-1,3) {$A_8$};
\node at (-1,2) {$B_7$};
\node at (-2,1) {$B_5$};
\node at (-3,2) {$A_1$};
\node at (-3,1) {$A_4$};   
\end{scope}  
\end{scope}

\end{tikzpicture}
\caption{Fock-Goncharov coordinates related by a flip of a diagonal}
\label{Fig:coo,acd.flip}
\end{figure}

The coordinate charts  $\{A_v\}_{v\in \Theta_{\mathcal{T}}}$  corresponding to different 3-triangulations of $\hat{S}$ are related by a sequence of cluster mutations \cite[Section 10]{FG06}. In detail, suppose $\mathcal{T}$ and $\mathcal{T}'$ are related by a flip of a diagonal. Within the quadrilateral containing the diagonal, the left quiver in Figure \ref{Fig:coo,acd.flip} consists of the coordinates associated with $\mathcal{T}$, and the right quiver consists of the coordinates associated with $\mathcal{T}'$. They related by four {\it cluster mutations}:
\begin{equation}
\label{cluster.mutation}
\begin{array}{l}
A_2B_2=A_1A_7+A_3A_5, \qquad A_6B_6=A_5A_{11}+A_7A_{10}\\
A_5B_5=A_4B_6+A_9B_2, \qquad A_7B_7=A_{12}B_2+A_8B_{6}.\\
\end{array}
\end{equation}
The rest coordinates are kept invariant. 
Putting them together, we obtain a canonical cluster $\mathcal{A}$ structure on the moduli space $\mathcal{A}_{{\rm SL}_3,\hat{S}}$. 

\smallskip

The potential $P$, introduced in \cite{GS15}, is a function  on the moduli space $\mathcal{A}_{G, \hat{S}}$ for an arbitrary reductive group $G$. For the purpose of the present paper, below we present an explicit expression of $P$ for $G={\rm SL}_3$ in terms of the Fock-Goncharov coordinates. Let $m$ be a marked point or a puncture of $\hat{S}$. The monodromy $g_m$ of each $(\rho, \xi)\in \mathcal{A}_{{\rm SL}_3, \hat{S}}$ associted with $m$ is a $3\times 3$ unipotent upper triangular matrix. Let $g_{ij}$ be the ${ij}$-th entry of $g_m$. The potential associated with $m$ is defined as the sum of the entries along the sub-diagonal
$P_m:= g_{12}+g_{23}.$
The total potential is 
\[
P:=\sum_{m\in m_b\cup m_p} P_m.
\]

\begin{example} Continuing as in Example \ref{triangle.a.space}, let $a, b, c$ be the marked points of $\Delta$. There are three unit rhombi whose short diagonals are parallel to the side $bc$. 
 As in Lemma 3.1 of \cite{GS15},  we define
\begin{equation}
\label{equation:alpha}
\alpha_g=\frac{A_{1,1,1}}{A_{2,1,0}A_{2,0,1}};\qquad \alpha_b=\frac{A_{2,1,0}A_{0,2,1}}{A_{1,2,0}A_{1,1,1}}; \qquad \alpha_r= \frac{A_{0,1,2}A_{2,0,1}}{A_{1,1,1}A_{1,0,2}}.
\end{equation}
Here the $A$-coordinates on the right correspond to the vertices of each unit rhombus, where the two vertices of the short diagonals are the denominators.  The potential associated with the vertex $a$ is the sum
  \[
  P(\Delta)_a:= \alpha_g+\alpha_b+\alpha_r.
  \]
  Similarly, by taking the unit rhombi with respect to $b$ and $c$, we get the potential associated with the ideal triangle $\Delta$
  \[
  P(\Delta):=P(\Delta)_a+P(\Delta)_b+P(\Delta)_c.
  \]
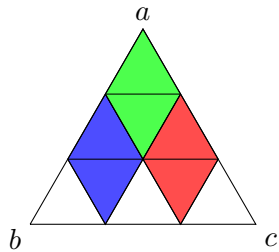
\begin{figure}[H]
\begin{tikzpicture}
\node at (1.5,2.8) {$a$};
\node at (-.2,-.2) {$b$};
\node at (3.2,-.2) {$c$};
\draw [fill=blue!70] (60:1)--++(60:1)--++(-60:1)--++(-120:1)--(60:1);
\draw [fill=green!70] (60:2)--++(60:1)--++(-60:1)--++(-120:1)--++(120:1);
\draw [fill=red!70] (60:1)++(1,0)--++(60:1)--++(-60:1)--++(-120:1)--++(120:1);
\draw (0,0)--(60:3)--(3,0)--(0,0);
\draw (60:2)--++(1,0);
\draw (60:1)--++(2,0);
\draw (1,0)--++(60:2);
\draw (2,0)--++(60:1);
\draw (1,0)--++(120:1);
\draw (2,0)--++(120:2);
\end{tikzpicture}
\caption{The green rhombus contains vertices that appear in the definition of $\alpha_g$; the same goes for the blue rhombus with $\alpha_b$ and the red rhombus with $\alpha_r$.}
\label{figure:acoor1}
\end{figure}
\end{example}

\begin{remark} Let $\mathcal{T}$ be an arbitrary ideal triangulation of $\hat{S}$. By the {\it scissor congruence invariance} property of the potential in \cite{GS15}, the potential $P$ is a function on $\mathcal{A}_{\operatorname{SL}_3,\hat{S}}$ presented by  
\begin{equation*}
P=\sum P(\Delta),
\end{equation*}
where the sum is over all the ideal triangles $\Delta$ in $\mathcal{T}$.
Note that the original geometric definition of $P$ does not depend on the ideal triangulation $\mathcal{T}$ chosen. 
\end{remark}
\begin{remark}
Below we include a list of applications of the  potential  in the higher Teichm\"uller theory.
\begin{enumerate}
\item In \cite{GS15}, the potentials are understood as the mirror Landau--Ginzburg potentials where they formulated a concrete homological mirror symmetry between $(\mathcal{A}_{G,\hat{S}},P)$ and the generalized character variety $\mathcal{L}_{G^L,\hat{S}}$. 
\item When $\hat{S}$ is a disk with three marked points on the boundary and $G=\operatorname{GL}_n$, the tropicalization of the  potential $P$ recovers the Knutson--Tao's hive model \cite{KT98}. 
\item When $\hat{S}$ is a punctured surface, in \cite{HS23}, for each simple root and each puncture, the  partial potential is understood as generalized horocycle length which provides a family of McShane-type identities. 
\item The partial potentials play important roles in the quantization of the moduli spaces of $G$-local systems in \cite{GS19} and the punctured skein relations \cite{SSWa}. In particular, when $\hat{S}$ is a punctured disk with two marked points, the quantized partial potentials correspond to generators of the quantum group $\mathcal{U}_q(\mathfrak{g})$. See \cite{S22}.
\end{enumerate}
\end{remark}

\subsection{Tropicalization.}
\label{sec.42}
Every cluster variety admits a totally positive structure and can be further tropicalized. Below we briefly recall the tropicalization of $\mathcal{A}_{{\rm SL}_3, \hat{S}}$. 

Note that the transition maps \eqref{cluster.mutation} among different cluster charts of $\mathcal{A}_{{\rm SL}_3, \hat{S}}$ are subtraction-free. A {\it positive rational function} is a nonzero function that can be presented as a ratio of two polynomials with positive integer coefficients in one (and therefore every) cluster chart of $\mathcal{A}_{{\rm SL}_3, \hat{S}}$. Denote by $\mathbb{Q}_+(\mathcal{A}_{{\rm SL}_3, \hat{S}})$ the collection of all positive functions. 

A {\it semifield} is a set $F$ equipped with operations of addition and multiplication, so that addition is commutative and associative, multiplication makes $F$ into an abelian group, and they satisfy the usual distributivity: $(a + b)c = ac + bc$ for $a,b,c\in F$. Note that $\mathbb{Q}_+(\mathcal{A}_{{\rm SL}_3, \hat{S}})$ is a semifield. For an arbitrary semifield $F$, we define its set of $F$-points as 
\[
\mathcal{A}_{{\rm SL}_3, \hat{S}}(F):={\bf Hom}_{\mbox{semifield}}\left(\mathbb{Q}_+(\mathcal{A}_{{\rm SL}_3, \hat{S}}), F\right).
\]

\begin{example} Let $F=\mathbb{R}_{>0}$ be the semifield of positive real numbers with usual addition and multiplication. The $\mathbb{R}_{>0}$-points of $\mathcal{A}_{{\rm SL}_3, \hat{S}}$ gives rise to the  total positive part of $\mathcal{A}_{{\rm SL}_3, \hat{S}}$
\begin{equation}
\label{positive.space}
\mathcal{A}_{\operatorname{SL}_3,\hat{S}}(\mathbb{R}_{>0}):=\left\{ x\in \mathcal{A}_{\operatorname{SL}_3,\hat{S}}~ \middle|~ A_v(x)>0~\mbox{for every cluster coordinate $A_v$}\right\}.
\end{equation}
Since the transition maps between any pair of cluster charts are given by positive rational functions, it is enough to require that $A_v(x)>0$ in \eqref{positive.space} for $A_v$ associated with one cluster chart. 
\end{example}

\begin{example}
The tropical semifield $\mathbb{R}^t:=(\mathbb{R},+, \max)$ is the set $\mathbb{R}$ with the usual addition as the multiplication and the $\max$ as the addition. 
Let $f\in \mathbb{Q}_+(\mathcal{A}_{{\rm SL}_3,\hat{S}})$ be a positive rational function. Its {\it tropicalization} is a  function defined tautologically as 
\[
f^t: \mathcal{A}_{\operatorname{SL}_3,\hat{S}}(\mathbb{R}^t)\longrightarrow \mathbb{R}, \qquad f^t(l):=l(f).
\]
The tropicalization $f^t$ is a piecewise linear function that replaces the multiplication in $f$ by addition, and the addition by $\max$.
For example, if $f=\frac{A_1^2+A_2^3}{A_1+2}$, then 
\[f^t=\max\{2A_1^t, 3A_2^t\}- \max\{A_1^t,0\}.\]

Let $n=\dim \mathcal{A}_{{\rm SL}_3, \hat{S}}$.
The tropical set $\mathcal{A}_{\operatorname{SL}_3,\hat{S}}(\mathbb{R}^t)$ has a piecewise-linear structure, isomorphic to $\mathbb{R}^n$ in many different ways.
In detail, let $\alpha_{\mathcal{T}}:=\{A_{1}, \ldots, A_{n}\}$ be the cluster chart associated with an ideal triangulation $\mathcal{T}$ of $\hat{S}$. Its tropicalization is a bijective map 
\[
\alpha_{\mathcal{T}}^t:=(A_1^t, \ldots, A_n^t): ~ \mathcal{A}_{\operatorname{SL}_3,\hat{S}}\left(\mathbb{R}^t\right)\stackrel{\sim}{\longrightarrow} \mathbb{R}^n.
\]
For a different ideal triangulation $\mathcal{T}'$, the transition map $\alpha_{\mathcal{T}'}^t\circ (\alpha_{\mathcal{T}}^t)^{-1}$ 
is given by the tropicalization of the positive birational map $\alpha_{\mathcal{T}'}\circ \alpha_{\mathcal{T}}^{-1}$. Similarly, we may replace $\mathbb{R}$ by $\mathbb{Z}$  in the above construction, and define the subsets of $\mathbb{Z}$- tropical points
\[
\mathcal{A}_{\operatorname{SL}_3,\hat{S}}\left(\mathbb{Z}^t\right)\subset 
\mathcal{A}_{\operatorname{SL}_3,\hat{S}}\left(\mathbb{R}^t\right).
\]

\end{example}

 By  Equation \eqref{equation:alpha}, for each  ideal triangulation $\mathcal{T}$, the tropicalization of the potential $P$  is a piecewise linear function $P^t$ on $\mathcal{A}_{{\rm SL}_3, \hat{S}}(\mathbb{R}^t)$ that can be presented as
 \begin{equation}
 \label{trop.potential}
 P^t=\max\left\{\alpha^t\right\},
 \end{equation}
 where $\alpha$ are the Laurent polynomials as in \eqref{equation:alpha} associated with unit rhombus of $Q_{\mathcal{T}}$.
 By imposing the condition $P^t\leq 0$, we obtain a convex cone
\begin{equation}
\label{posicone}
\mathcal{A}_{{\rm SL}_3,\hat{S}}^+\left(\mathbb{R}^t\right):=\left\{ x\in \mathcal{A}_{{\rm SL}_3,\hat{S}}\left(\mathbb{R}^t \right)~\middle|~P^t(x)\leq 0\right\}.
\end{equation}
\begin{remark} The paper \cite{GS15} uses the tropical semifield $(\mathbb{Z}, +, \min)$, and define the positive cone via requiring $P^t \geq 0$. The positive cone defined in {\it loc.cit.} is isomorphic to \eqref{posicone} via the map taking $t\in \mathbb{Z}$ to $-t$. 
\end{remark}
By replacing the group ${\rm SL}_3$ by ${\rm PGL}_3$, we define the moduli space $\mathcal{A}_{{\rm PGL}_3, \hat{S}}$. Note that $\mathcal{A}_{{\rm PGL}_3, \hat{S}}$ may have several disconnected components. Meanwhile, there is a natural finite-to-one map from $\mathcal{A}_{{\rm SL}_3, \hat{S}}$ to $\mathcal{A}_{{\rm PGL}_3, \hat{S}}$, whose image is a component of $\mathcal{A}_{{\rm PGL}_3, \hat{S}}$ with a totally positive structure. Similarly, we define the tropical points of  $\mathcal{A}_{{\rm PGL}_3, \hat{S}}$ with a different lattice structure:
\[\mathcal{A}_{\operatorname{SL}_3,\hat{S}}(\mathbb{Z}^t)\subset \mathcal{A}_{\operatorname{PGL}_3,\hat{S}}(\mathbb{Z}^t) \subset \mathcal{A}_{\operatorname{PGL}_3,\hat{S}}\left(\mathbb{R}^t\right)=\mathcal{A}_{\operatorname{SL}_3,\hat{S}}\left(\mathbb{R}^t\right).\]
By imposing the condition $P^t \leq 0$, we obtain the cone
\[
\mathcal{A}_{\operatorname{PGL}_3,\hat{S}}^+(\mathbb{Z}^t)= \mathcal{A}_{\operatorname{PGL}_3,\hat{S}}(\mathbb{Z}^t) \cap  \mathcal{A}_{\operatorname{SL}_3,\hat{S}}^+(\mathbb{R}^t)
\]
More explicitly, following \cite{GS15}, we see that 
\begin{equation}
\label{defn:tropweb}
\mathcal{A}_{\operatorname{PGL}_3,\hat{S}}^+(\mathbb{Z}^t):=\left\{x\in \mathcal{A}_{\operatorname{SL}_3,\hat{S}}\left(\mathbb{R}^t\right)\;\bigg|\; \text{for any } \alpha \text{ of $P^t$  in  \eqref{trop.potential}}, \alpha^t(x) \in \mathbb{Z}_{\leq 0} \right\}.
\end{equation}

By definition, the map $\alpha^t_{\mathcal{T}}$ identifies the set $\mathcal{A}_{{\rm PGL}_3,\hat{S}}^+$ with ${\bf Hive}(\mathcal{T})$. The transition maps relating different triangulations are related by the tropicalization of the cluster mutations \eqref{cluster.mutation}, which coincides with octahedron relations \eqref{oact.rec}. 
Following Theorem \ref{Int.hive}, we get a natural isomorphism
\[
\mathcal{A}_{{\rm PGL}_3,\hat{S}}^+(\mathbb{Z}^t)\stackrel{\sim}{=} \mathscr{W}_{\hat{S}}^{\mathcal{A}}.
\]

\subsection{Intersection pairings revisited}

In this subsection, we discuss several different perspectives on the intersection pairings. 

\smallskip 

\noindent {\bf A. Geometric interpretation of reduced $\mathcal{X}$-webs.} The paper \cite{GS15} introduces the moduli space $\mathcal{P}_{{\rm G}, \hat{S}}$, as a refinement of the Fock-Goncharov moduli space $\mathcal{X}_{{\rm G}, \hat{S}}$ in \cite{FG06}. Let $\mathcal{B}\stackrel{\sim}{=} {\rm G}/{\rm B}$ be the flag variety associated with ${\rm G}$. 
\begin{defn} The moduli space $\mathcal{P}_{{\rm G}, \hat{S}}$ parametrizes the ${\rm G}$-orbits of the data $(\rho, \xi, \gamma)$, where
\begin{itemize}
\item $\rho$ is a ${\rm G}$-local system on $S$;
\item $\xi$ is flat section of the associated bundle $\rho \times_{{\rm G}} \mathcal{A}$ restricted to every boundary interval in $\partial{S}\backslash m_b$;

\item $\gamma$ is a flat section of the associated bundle $\rho \times_{{\rm G}}\mathcal{B}$ restricted to every boundary circle in $\partial{S}\backslash m_b$.
\end{itemize}
\end{defn}

When the group ${\rm G}$ is adjoint, e.g., ${\rm G}={\rm PGL}_n$, then the moduli space carries a cluster Poisson structure \cite{GS19}. 
For example, if ${\rm G}={\rm PGL}_3$, recall the quiver $Q_{\mathcal{T}}$ with the vertex set $\Theta_{\mathcal{T}}$. By consider cross ratios associated with flags, the space $\mathcal{P}_{{\rm PGL}_3, \hat{S}}$ is equipped with a collection $\{X_v\}_{v\in \Theta_{\mathcal{T}}}$ of cluster $\mathcal{X}$-coordinates. Locally, a cluster mutation yields the following change of coordinates. 

\begin{figure}[H]
\begin{tikzpicture}[scale=1]
\node (O) at (0,0) {$X$};
\node (D) at (1,1) {$X_4$};
\node (C) at (1,-1) {$X_3$};
\node (B) at (-1,-1) {$X_2$};
\node (A) at (-1,1) {$X_1$};
\draw[latex-,  red] (O) -- (D);
\draw[latex-,  red] (O) --(B);
\draw[latex-,  red] (C)--(O);
\draw[latex-,  red] (A)--(O);
\draw[latex-,  red] (D) -- (C);
\draw[latex-,  red] (B) -- (A);
\end{tikzpicture} \hspace{2cm}
\begin{tikzpicture}[scale=1]
\node (O) at (0,0) {$X^{-1}$};
\node [label=right: $X_4(1+X^{-1})^{-1}$] (D) at (0.7,1) {};
\node [label=right: $X_3(1+X)$] (C) at (0.7,-1) {};
\node [label=left: $X_2(1+X^{-1})^{-1}$] (B) at (-.7,-1) {};
\node [label=left: $X_1(1+X)$] (A) at (-.7,1) {};
\draw[-latex,  red] (O) -- (.8,.8);
\draw[latex-,  red] (O) --(.8,-.8);
\draw[latex-,  red] (-.8,-.8)--(O);
\draw[-latex,  red] (-.8,.8)--(O);
\draw[-latex,  red] (D) -- (A);
\draw[-latex,  red] (B) -- (C);
\end{tikzpicture}
\end{figure}

Note that the transition maps are subtraction-free. Therefore $\mathcal{P}_{{\rm PGL}_3, \hat{S}}$ admits a totally positive structure and can be further tropicalized. Let us impose the potential condition \eqref{posicone} to every boundary interval and further require that the monodromy surrounding each puncture corresponds to a dominant coweight of ${\rm PGL}_3$. In this way, we obtain a cone 
\[
\mathcal{P}_{{\rm PGL}_3,\hat{S}}^+(\mathbb{R}^t)\subset \mathcal{P}_{{\rm PGL}_3,\hat{S}}(\mathbb{R}^t).
\]
Furthermore, We expect the following conjecture to be true.

\begin{conj} There is a natural bijection $\mathcal{P}_{{\rm PGL}_3,\hat{S}}^+(\mathbb{Z}^t)\stackrel{\sim}{=} \mathscr{W}_{\hat{S}}^{\mathcal{X}}$.
\end{conj}

\begin{example}
The  following tropical $\mathcal{X}$-coordinates  corresponds to an outward tripod.
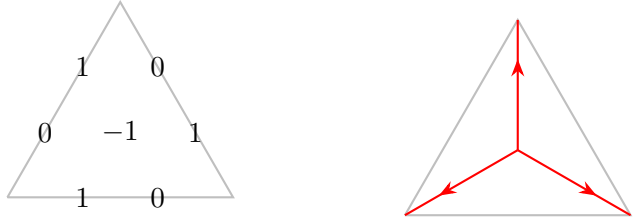
\begin{figure}[H]
\begin{tikzpicture}
\draw[gray!50, thick] (0,0)--(60:3)--(3,0)--(0,0);
\node (A) at (1,0) {$1$};
\node (B) at (2,0) {$0$};
\node at (60:1) {$0$};
\node at (60:2) {$1$};
\begin{scope}[shift={(3,0)}] 
\node at (120:1) {$1$};
\node at (120:2) {$0$};
\end{scope}
\begin{scope}[shift={(2,0)}] 
\node at (120:1) {$-1$};
\end{scope}
\end{tikzpicture}
\hspace{2cm}
\begin{tikzpicture}
\draw[gray!50, thick] (0,0)--(60:3)--(3,0)--(0,0);
 \draw [red, thick, decoration={markings, mark=at position 0.7 with {\arrow{Stealth}}}, postaction=decorate]  (30:1.732) -- (0,0);
  \draw [red, thick, decoration={markings, mark=at position 0.7 with {\arrow{Stealth}}}, postaction=decorate] (30:1.732) -- (60:3);
   \draw [red, thick, decoration={markings, mark=at position 0.7 with {\arrow{Stealth}}}, postaction=decorate] (30:1.732) --(3,0);
\end{tikzpicture}
\caption{Cluster $\mathcal{X}$-coordinates corresponding to an outward tripod.}
\end{figure}
\end{example}
The intersection pairing among reduced webs induces a pairing among tropical points
\begin{equation}
\label{can. pairing}
\mathbb{I}:~\mathcal{A}_{{\rm PGL}_3, \hat{S}}^+(\mathbb{Z}^t) \times \mathcal{P}_{{\rm PGL}_3, \hat{S}}^+(\mathbb{Z}^t) \longrightarrow \frac{1}{3} \mathbb{Z}.
\end{equation}

\noindent {\bf B. The canonical pairing from the perspective of cluster duality.} We expect that the pairing \eqref{can. pairing} coincides with the canonical pairing in the setting of cluster ensembles due to Fock and Goncharov \cite[Conj. 4.3]{FG09}. Following the notation of {\it loc.cit.}, let $\mathcal{A}$ be a cluster $\mathcal{A}$-variety and let $\mathcal{X}^\vee$ be the cluster $\mathcal{X}$-variety of Langlands dual type. For example, our main example $(\mathcal{A}_{{\rm SL}_3, \hat{S}}, \mathcal{P}_{{\rm PGL}_3, \hat{S}})$ is such a pair. 

The Fock-Goncharov duality conjecture asserts that every tropical point $l \in \mathcal{A}(\mathbb{Z}^t)$ naturally corresponds to a positive regular function $\theta_l$ on $\mathcal{X}^\vee$. Its tropicalization induces a pairing
\[
\mathbb{I}_{FG}: \mathcal{A}(\mathbb{Z}^t) \times \mathcal{X}^\vee(\mathbb{Z}^t) \longrightarrow  \mathbb{Z},\qquad  (l,m) \longmapsto \theta_l^t(m).
\]
Meanwhile, every tropical point $m\in \mathcal{X}^\vee(\mathbb{Z}^t)$ corresponds to a function $\vartheta_m$ on $\mathcal{A}$. We expect that it will give rise to the same pairing $\theta_l^t(m)=\vartheta_m^t(l)$. 

Each cluster seed ${\bf i}$ provides a coordinate coordinate system $\alpha_{\bf i}$ for $\mathcal{A}$ and a cluster coordinate system $\chi_{\bf i}$ for $\mathcal{X}^\vee$. 
Let $(a_1,\ldots, a_n)$ be the coordinate of $l\in \mathcal{A}(\mathbb{Z}^t)$ under $\alpha_{\bf i}$ and let $(x_1,\ldots, x_n)$ be the coordinates of $m\in \mathcal{X}^\vee(\mathbb{Z}^t)$ under $\chi_{\bf i}$. By the cluster duality, we shall have
\[
\langle l, m\rangle_{\bf i}:= \sum_j a_j x_j \leq  \mathbb{I}_{FG}(l,m).
\]
\begin{conj} We conjecture that 
\[
\mathbb{I}_{FG}(l,m) = \max_{\bf i} ~\langle l, m\rangle_{\bf i}.
\]
\end{conj}

\bibliographystyle{alphaurl-a}
\bibliography{main}

\end{document}